\documentclass[reqno,10pt]{amsart}

\title[]{Asymptotic properties of infinitesimal characters and applications}

\author[]{Andr\'es Sambarino}
\date{}

\usepackage[colorlinks = true,backref = page,hyperindex,breaklinks]{hyperref} 
\usepackage{ stmaryrd }
\usepackage{amssymb}
\usepackage{mathtools}      
\usepackage{euscript}
\usepackage[bb = Fourier, cal = dutchcal, calscaled=1.1, scr = boondox]{mathalfa}	
\usepackage{enumitem}       
\usepackage[english]{babel}
\usepackage{tikz}			
\usepackage{tikz-cd}	
\usepackage{tkz-euclide}
\usepackage{pgfkeys}
\usetikzlibrary{intersections}
\usetikzlibrary{shapes.geometric}
\usepackage[mark=o]{dynkin-diagrams}
\pgfkeys{/Dynkin diagram}

\tikzcdset{arrows=dash}
\usepackage[font = small]{caption}	
\usepackage{nicefrac}

\usetikzlibrary{calc}
\usepackage[percent]{overpic}
\usepackage{mathdots}

\usepackage{array,multirow}
\usepackage{graphicx}
\usepackage{afterpage}
\usepackage[section]{placeins}
\usepackage{ wasysym }
\usepackage{ upgreek }
\usepackage{nicematrix}
\usepackage{rank-2-roots}

\renewcommand*{\backref}[1]{}
\renewcommand*{\backrefalt}[4]{\quad \tiny
  \ifcase #1 (\textbf{NOT CITED.})%
  \or    (Cited on page~#2.)%
  \else   (Cited on pages~#2.)%
  \fi}

\makeatletter

\def\MRbibitem{\@ifnextchar[\my@lbibitem\my@bibitem}

\def\mybiblabel#1#2{\@biblabel{{\hyperref{http://www.ams.org/mathscinet-getitem?mr=#1}{}{}{#2}}}}

\def\myhyperanchor#1{\Hy@raisedlink{\hyper@anchorstart{cite.#1}\hyper@anchorend}}

\def\my@lbibitem[#1]#2#3#4\par{%
  \item[\mybiblabel{#2}{#1}\myhyperanchor{#3}\hfill]#4%
  \@ifundefined{ifbackrefparscan}{}{\BR@backref{#3}}%
  \if@filesw{\let\protect\noexpand\immediate
    \write\@auxout{\string\bibcite{#3}{#1}}}\fi\ignorespaces%
}

\def\my@bibitem#1#2#3\par{%
  \refstepcounter\@listctr
  \item[\mybiblabel{#1}{\the\value\@listctr}\myhyperanchor{#2}\hfill]#3%
  \@ifundefined{ifbackrefparscan}{}{\BR@backref{#2}}%
  \if@filesw\immediate\write\@auxout
    {\string\bibcite{#2}{\the\value\@listctr}}\fi\ignorespaces%
}

\makeatother

\newcommand{\xqedhere}[2]{%
  \rlap{\hbox to#1{\hfil\ledrappierap{\ensuremath{#2}}}}}

\newcommand{\Z}{\mathbb{Z}} 
\newcommand{\Q}{\mathbb{Q}}
\newcommand{\R}{\mathbb{R}}
\newcommand{\C}{\mathbb{C}}
\newcommand{\N}{\mathbb{N}}
\renewcommand{\P}{\mathbb{P}}

\newcommand{\K}{\mathbb K}
\renewcommand{\H}{\mathbb H}

\newcommand{\T}{\Theta}

\newcommand{\lb}{\llbracket}
\newcommand{\rb}{\rrbracket}
\newcommand{\eps}{\varepsilon}
\newcommand{\G}{\Upgamma}    
\newcommand{\gh}{{\sf \Gamma}}    
\newcommand{\GA}{\Uplambda}

\newcommand{\<}{\langle}
\renewcommand{\>}{\rangle}
\newcommand{\E}{{\sf E}}
\newcommand{\g}{\gamma}

\newcommand{\bord}{\partial}
\newcommand{\Lie}{\mathrm{Lie}}

\newcommand{\posgen}{\EuScript F^{(2)}}

\newcommand{\UG}{\sf{U}\gh}
\newcommand{\vt}{\uptheta}

\newcommand{\p}{{\mathsf{f}}}

\newcommand{\wk}{\check}
\renewcommand{\aa}{\alpha}
\newcommand{\bb}{\beta}
\renewcommand{\/}{\backslash}

\newcommand{\grupo}{\Lambda} 

\newcommand{\Bcone}{{\EuScript L}}
\newcommand{\conodual}[1]{\inte({\Bcone_{#1})^*}}

\newcommand{\Mod}{\mathrm{Mod}}

\newcommand{\gedossplit}{{\mathfrak{G}_2}}
\newcommand{\Gedossplit}{{\mathsf{G_2}}}

\newcommand{\scr}{\mathscr}
\renewcommand{\sf}[1]{{\mathsf{#1}}}

\newcommand{\cal}{\mathcal}
\renewcommand{\frak}{\mathfrak}

\newcommand{\ledrappier}{{\EuScript J}}
\newcommand{\re}{{\scr r}}

\newcommand{\neudim}{\operatorname{neudim}}

\newcommand{\A}{{\sf A}}
\newcommand{\B}{{\sf B}}
\newcommand{\Ce}{{\sf C}}
\newcommand{\D}{{\sf D}}
\newcommand{\F}{{\sf F}}
\newcommand{\Ge}{{\sf G}_2}

\newcommand{\car}{{\frak X}}

\newcommand{\peso}{\varpi}
\newcommand{\yy}{\sf{y}}
\newcommand{\xx}{\sf{x}}

\newcommand{\ua}{\upalpha}

\newcommand{\Hff}{\mathrm{Hff}}

\newcommand{\symp}{\upomega}
\newcommand{\entropy}[2]{\hh_{#2}^{#1}}

\renewcommand{\t}{\vartheta}

\newcommand{\centro}{{\frak Z}}

\newcommand{\undi}{\underline{i}}
\newcommand{\undiX}{\underline{i}^{\frak X}}

\newcommand{\length}{\psi}

\DeclareMathOperator{\Var}{var}
\DeclareMathOperator{\covar}{covar}

\newcommand{\caracteres}{{\frak X(\gh,\sf G)}}
\newcommand{\caracG}{{\frak X(\G,\sf G)}}

\newcommand{\prin}{\updelta}
\newcommand{\fund}{\phi}
\newcommand{\rep}{\upphi}

\newcommand{\tangcero}[2]{\mathsf{T}_{#1}^{#2}}

\newcommand{\gibbs}{\upnu}

\newcommand{\crossm}{\mathsf{B}}
\newcommand{\cross}{\frak{C}}
\newcommand{\crossg}{\frak{G}}
\newcommand{\crossG}{\frak{\Gamma}}

\newcommand{\affinelim}{\mathscr{A}}
\newcommand{\jordanlim}{\mathscr{V\!\!J}}

\newcommand{\full}{\G_{\mathrm{full}}}
\newcommand{\kt}{{\varkappa}}
\newcommand{\ktv}{\upkappa}

\newcommand{\umargulis}{\breve{\margulis}}
\newcommand{\margulis}{{\sf m}}
\newcommand{\lin}[1]{{\dot{#1}}}
\newcommand{\tra}[1]{v_{#1}}
\newcommand{\espacio}{{\EuScript V}^0}

\newcommand{\co}{{\sf u}}
\newcommand{\coclase}{[{\sf u}]}

\newcommand{\morfi}{{\rho(\gh)}}
\newcommand{\tangente}{{v}}
\newcommand{\factor}[1]{V_{#1}}
\newcommand{\mass}[2]{\cal p_{#1}{#2}} 
\newcommand{\trivia}{{\EuScript T}}
\newcommand{\raff}{{\breve{\rr}}}
\newcommand{\raffN}{\rr_\Vlongest}
\newcommand{\rr}{\cal{r}}
\newcommand{\raffi}{\rr_{\trivia}}

\DeclareMathOperator{\supp}{supp}

\newcommand{\kahl}{{\upvarphi}}

\newcommand{\vp}{\mu}

\renewcommand{\b}{{\frak b}}
\newcommand{\n}{{\frak n}}
\newcommand{\gl}{{\frak{gl}}}
\renewcommand{\sl}{\frak{sl}}
\renewcommand{\ge}{{\frak g}}
\newcommand{\m}{{\frak m}}

\newcommand{\e}{{\frak e}}
\newcommand{\f}{{\frak f}}
\newcommand{\s}{{\frak s}}
\renewcommand{\a}{{\frak a}}
\newcommand{\h}{{\frak h}}
\renewcommand{\u}{{\frak u}}
\renewcommand{\p}{{\frak p}}
\renewcommand{\k}{{\frak k}}
\newcommand{\so}{\frak{so}}
\renewcommand{\sp}{\frak{sp}}
\newcommand{\su}{{\frak{su}}}

\newcommand{\gedosl}{\rep_{\peso_\aa}(\gedossplit)}
\newcommand{\Gedosl}{\rep_{\peso_\aa}(\Gedossplit)}
\newcommand{\dual}{\wp}
\newcommand{\lf}[1]{\underline{#1}}

\newcommand{\nrm}{{\Psi}}

\newcommand{\medidas}{\EuScript M}
\newcommand{\Holder}{\mathrm{H\ddot ol}}
\newcommand{\Anosov}{\frak A}

\newcommand{\deriva}[2]{\frac{\partial}{\partial #1}\Big|_{#1=#2}}

\newcommand{\Weyl}{W}
\newcommand{\longest}{{w_0}}

\newcommand{\X}{\EuScript X}

\newcommand{\simple}{{\sf\Delta}}
\renewcommand{\root}{{\sf\Phi}}
\newcommand{\poids}{{\sf\Pi}}
\newcommand{\sroot}{{\sigma}}
\newcommand{\slroot}{{\upsigma}}

\newcommand{\encono}{{\X}}

\newcommand{\adg}{\sf g}

\newcommand{\dd}{\delta} 

\newcommand{\pp}{{\mathrm P}}

\newcommand{\jordan}{\uplambda}
\DeclareMathOperator{\Int}{\mathrm{Inn}}
\DeclareMathOperator{\isom}{Isom}

\DeclareMathOperator{\ii}{i}
\DeclareMathOperator{\spa}{span}
\newcommand{\class}{\mathrm{C}}
\DeclareMathOperator{\End}{End}
\DeclareMathOperator{\SL}{{\mathsf{SL}}}
\DeclareMathOperator{\PSL}{{\mathsf{PSL}}}
\DeclareMathOperator{\GL}{{\mathsf{GL}}}
\DeclareMathOperator{\SO}{{\mathsf{SO}}}
\DeclareMathOperator{\PGL}{{\mathsf{PGL}}}

\DeclareMathOperator{\id}{id}
\newcommand{\inte}{\mathrm{int}\,}
\DeclareMathOperator{\relinte}{\mathrm{rel-int}}

\newcommand{\hitchin}{\EuScript{H}}

\DeclareMathOperator{\traza}{Trace}
\DeclareMathOperator{\ad}{ad}
\DeclareMathOperator{\Ad}{Ad}

\DeclareMathOperator{\rk}{rank}

\DeclareMathOperator{\Fix}{Fix}

\DeclareMathOperator{\diag}{diag}
\DeclareMathOperator{\hess}{Hess}

\newcommand{\diff}{\triangle}

\newcommand{\vacio}{\emptyset}

\newcommand{\tpos}{\EuScript P_\T(S,\sf G)}

\newcommand{\poif}{\poids_\rep}
\newcommand{\Vlongest}{\EuScript N}

\newcommand{\cero}{{\EuScript O}}
\newcommand{\area}{\mathrm{area}}

\newcommand{\hh}{{\scr h}}
\newcommand{\mmax}{\upmu}
\newcommand{\II}{\mathbf{I}}
\newcommand{\JJ}{\mathbf{J}}
\newcommand{\PP}{\mathbf{P}}

\renewcommand{\tt}{\mathrm}

\newcommand{\M}{\mathrm{MS}}

\newcommand{\VV}[1]{{\mathbb V}^{#1}}
\newcommand{\cartan}{\mu}

\newcommand{\grass}{\mathrm{Gr}}
\renewcommand{\angle}{\measuredangle}

\newcommand{\transverse}{\pitchfork}

\newcommand{\FF}{\mathrm F}
\newcommand{\trs}{\mathrm{T}}

\DeclareMathOperator{\PSp}{{\sf{PSp}}}

\newcommand{\barra}[1]{[#1]}

\newcommand{\Aff}{{\sf{Aff}}}

\newcommand{\varjor}[2]{\mathrm d\jordan^{#1}{(#2)}}
\newcommand{\varjort}[2]{\mathrm d\jordan_\t^{#1}{(#2)}}

\newcommand{\bpm}{\begin{pmatrix}}
\newcommand{\epm}{\end{pmatrix}}

\renewcommand{\o}{{\cal o}}
\newcommand{\lenf}{\mathfrak L}

\hypersetup{bookmarksdepth  =  3} 
\numberwithin{equation}{section}     

\setlist[enumerate,1]{label = {\upshape(\roman*)},ref = \roman*}
\setlist[enumerate,2]{label = {\upshape(\alph*)},ref = \alph*}

\newtheorem{thmA}{Theorem}

\newtheorem*{corN}{Corollary}

\newtheorem{thm}{Theorem}[section]

\newtheorem{cor}[thm]{Corollary}
\newtheorem{lemma}[thm]{Lemma}
\newtheorem{prop}[thm]{Proposition}

\theoremstyle{definition}

\newtheorem{defi}[thm]{Definition}

\newtheorem{ex}[thm]{Example}

\newtheorem{assu}{Assumption}

\theoremstyle{remark}
\newtheorem{obs}[thm]{Remark}

\newtheorem*{ack}{Acknowledgements}

\begin{document}

\begin{abstract}Inspired by Benoist, we study objects linked to integrable tangent vectors on the character variety of a semi-group $\G$ with values in a semi-simple real-algebraic group $\sf G$. We prove the \emph{cone of Jordan variations} has non-empty interior and, when $\sf G$ is split, establish non-empty interior of the set of \emph{length-normalized variations}. We apply these techniques to pressure forms on Anosov representations and higher-rank Teichmüller spaces. We identify an explicit functional $\kahl\in\a^*$ whose pressure form is compatible with Goldman's symplectic form at Fuchsian points in the Hitchin component. Finally, we show the degeneration of the Hausdorff dimension of \emph{higher}-quasi-circles is governed by a Diophantine equation.
\end{abstract}

\maketitle

\tableofcontents

\section{Introduction}

Let $\G$ be a semi-group and $\sf G$ a Zariski-connected semi-simple real-algebraic Lie group of the non-compact type. The \emph{character variety} of $\G$ with values in $\sf G$, of morphisms up to conjugation, is denoted by $\caracG=\hom(\G,\sf G)/\sf G.$ In this paper we investigate several objects associated to an integrable tangent vector $$v\in\sf T_\rho\caracG.$$

\noindent
We will think of $v$ as the (quotient projection of the) derivative of a curve $(\rho_t)_{t\in(-\eps,\eps)}$ in $\hom(\G,\sf G)$ with $\rho_0=\rho$ and such that for every $\g\in\G$ the curve $t\mapsto\rho_t(\g)$ is real-analytic in some neighborhood of $0$.

Let $\a$ be a Cartan subspace of $\sf G$, $\a^+\subset\a$ a closed Weyl chamber and $\jordan:\sf G\to\a^+$ be the Kostant-Jordan-Lyapunov-projection: up to signs, $\exp(\jordan(g))$ is conjugated to the $\R$-diagonalizable element of Jordan's decomposition of $g$. Commonly, $g$ is \emph{loxodromic} if $\jordan(g)\in\inte \a^+.$ For $\g\in\G$ we let $\jordan^\g:\caracG\to\a$ be the map $$\jordan^\g(\eta)=\jordan\big(\eta(\g)\big)\textrm{ and }\varjor\g{v}=\deriva t0\jordan(\rho_t(\g))\in\a$$ its differential at $v$.  For $\varphi\in\a^*$ we let $\varphi^\g:\caracG\to\R$ be the composition $$\varphi^\g=\varphi\circ\jordan^\g: \eta\mapsto\varphi\big(\jordan\big(\eta(\g)\big)\big).$$

Recall that \emph{Benoist's limit cone} of $\rho$ is defined by $\Bcone_\rho=\overline{\{\R_+\cdot\jordan^\g(\rho):\g\in\G\}}.$ Its dual $(\Bcone_\rho)^*$ consists on linear forms $\length\in\a^*$ such that $\length|\Bcone_\rho\geq0$.

\subsection{The cone of Jordan variations} 

We introduce \emph{the cone of Jordan variations} $$\jordanlim_{v}:=\overline{\Big\{\R_+\cdot\varjor{\g}{v}:\g\in\G\textrm{ with loxodromic }\rho(\g)\Big\}}\subset\a$$ and for $\length\in\conodual\rho$ we introduce the \emph{set of normalized variations} $$\VV{\length}_{v}:=\overline{\left\{\frac{\varjor\g{v}}{\length^{\g}(\rho)}:\g\in\G\textrm{ with loxodromic }\rho(\g)\right\}}\subset\a.$$

We will rule off variations that occur in proper normal subgroups of $\sf G$. Write $\ge=\bigoplus_{i\in I} \ge_i$ with $\ge_i$ a simple ideal and assume we've chosen the Cartan subspaces $\a_i$ of $\ge_i$ so that $\a=\bigoplus_i \a_i$. Let $p_i:\a\to\a_i$ be the associated projections. Then $v$ \emph{has full loxodromic variation} if for every $i\in I$ one has $p_i(\jordanlim_v)\neq\{0\}$, so \emph{full} stands for 'non-trivial variation in every simple factor of $\sf G$' and loxodromic stands for 'the variation is seen on loxodromic elements'. A $\g\in\G$ with loxodromic $\rho(\g)$ has \emph{full variation} if $\forall i\ p_i\big(\varjor\g v\big)\neq0$.

Let us simplify terminology and say that $v$ \emph{has Zariski-dense base-point} if $v\in\sf T_\rho\caracG$ and $\rho(\G)$ is Zariski-dense in $\sf G$. The following statements can be found (respectively) in Corollary \ref{phi-cotangent}, Proposition \ref{fullZariskidenso} and Proposition \ref{normconvex}.

\begin{thmA}\label{A}Let $v\in\sf T_\rho\caracG$ have Zariski-dense base point and full loxodromic variation. Then, $\jordanlim_{v}$ is convex and has non-empty interior. Moreover, full variation elements of $\rho(\G)$ are Zariski-dense in $\sf G$ and their Jordan projections intersect any open subcone of $\Bcone_\rho$. For every $\length\in\conodual\rho$ the set $\VV\length_v$ is convex.\end{thmA}

The most involved statement is to guarantee non-empty interior of $\jordanlim_v$. 
This is of course analogous, and inspired by, the classical result by Benoist \cite{limite} stating that if $\rho(\G)$ is Zariski-dense then $\Bcone_\rho$ is convex and has non-empty interior.

Theorem \ref{A} is stablished by means of the affine geometry $\sf G\ltimes_{\Ad}\ge$. If $(g,x)\in\sf G\ltimes\ge$ has loxodromic linear part, i.e. $g$ is loxodromic, then its \emph{Margulis projection} is well defined. This is a conjugacy invariant introduced by Margulis \cite{MargulisAffineRuso,MargulisAffine} for the affine group $\SO_{2,1}\ltimes\R^3$, when he proved existence of non-abelian free groups acting properly discontinuously on $\R^3$ by affine transformations. In the current context, this projection was defined by Smilga \cite{smilga1}.

The bridge between Theorem \ref{A} and the affine geometry is given by Proposition \ref{margulisderivada} below, independently established Kassel-Smilga \cite{KS} and also by Ghosh \cite{ghoshJordan} who further requires that $\sf G$ is split. Recall that a variation $v\in\sf T_\rho\caracG$ induces a $1$-cocycle $\co_v:\G\to\ge$ defined by \begin{equation}\label{deco}\co_v(\g)=\left.\frac{\partial}{\partial t}\right|_{t=0}\rho_t(\g)\rho(\g)^{-1},\end{equation} and $(\rho,\co_v)$ is a group morphism $\G\to\sf G\ltimes\ge$. Then, Proposition \ref{margulisderivada} states that $\varjor\g v$ coincides with the $\a$-coordinate of the Margulis projection of $(\rho(\g),\co_v(\g))$.

We then study more general affine groups $\sf G\ltimes_\rep V$ for a class of representations $\rep:\sf G\to\SL(V)$ and establish  \emph{non-empty interior} results in this setting (Corollary \ref{interiorAL} for irreducible $\rep$ and Corollary \ref{conococyclos} when $\rep$ is reducible but \emph{disjoined}). Similar versions of Corollary \ref{interiorAL} will also appear in Kassel-Smilga \cite{KS} and in Ghosh \cite{Souravpersonalcomm}, in particular \cite{KS} obtains the convexity stated in Theorem \ref{A}.

To mimic Benoist's proof in \cite{limite} we introduce the concept of \emph{affine ratio}, an invariant of four affine flags in general position; that can also be found in the independent work of Ghosh \cite{ghoshJordan} (for split groups, who also proves convexity of spectrum in this case). We then rely on Smilga's work \cite{SmilgaAnnalen} to  relate the defect of additivity of Margulis's invariants to this affine ratio. These results are achieved in Part \ref{affineactions}, however this viewpoint is used in the sequel, specially for Theorem \ref{thmPrinc}.


As a consequence  we obtain the following, generalizing previous work of Mess \cite{Mess}, Goldman-Labourie-Margulis \cite{GLM} (see also the alternative proof given by Danciger-Gu\'eritaud-Kassel \cite{DGKLorentz}) on $\H^2_\R$. If $\sf H$ is rank 1 and simple, and $\rho:\G\to\sf H$ is convex co-cocompact, then we let $\scr h(\rho)$ be the Hausdorff dimension of its limit set. Recall from Bridgeman-Canary-Labourie-S. \cite{pressure} that $\scr h$ is analytic about $\rho$.

\begin{corN}[Corollary \ref{levelsets} - Deformations along level sets of $\entropy{}{}$ give non-proper actions] Let $\sf H$ be the (identity component of the) isometry group of $\H^n_\R$ $n\neq3$, $\H^n_\H$ $n\geq2$, or the Cayley hyperbolic plane. Let $v\in\sf T_\rho\frak X(\G,\sf H)$ have Zariski-dense and convex-co-compact base-point, if $\mathrm d\scr h(v)=0$ then the action $(\rho,\co_v)$ on $\h$ is not proper.
\end{corN}

\subsection{Base point Liv\v sic-independence}

Denote by $\simple$ the set of simple restricted roots associated to $\a^+$. For $\sroot\in\simple$ let $\ge_\sroot$ be its root-space (Eq. \eqref{rootspace}). For a non-empty $\t\subset\simple$ the subspace $\a_\t=\bigcap_{\sroot\in\simple-\t}\ker\sroot$ comes equipped with a natural projection $\pi_\t:\a\to\a_\t$ (see \S\,\ref{s.Levi}). We let \begin{alignat*}{2}\jordan_\t&=\pi_\t\circ\jordan,\\\VV{\length}_{\t,v}&=\pi_\t\big(\VV\length_v\big).\end{alignat*}

In Corollary \ref{double-density} we prove the following:

\begin{thmA}[Double-density for roots with multiplicity $1$]\label{teoB}  Let $v\in\sf T_\rho\caracG$ have full loxodromic variation and Zariski-dense basepoint. Let $\t\subset\simple$ be such that $\dim\ge_\sroot=1$ for all $\sroot\in\t$, then the additive group spanned by $$\big\{\big(\varjort\g v,\jordan^\g(\rho)\big):\g\in\G\textrm{ with loxodromic $\rho(\g)$}\big\}$$ is dense in $\a_\t\times\a$. In particular, for any $\length\in\conodual\rho$ the convex set $\VV\length_{\t,v}$ has non-empty interior.\end{thmA}

Before passing to the next subsection we show two applications of Theorem \ref{teoB}. The first one can be found in Corollary \ref{HilbertCritica}, we refer the reader to \S\ref{anosovpre} for the definition of Anosov representations, introduced by Labourie \cite{labourie} and generalized by Guichard-Wienhard \cite{olivieranna}. Recall also that for $\rho\in\car(\gh,\SL(3,\R))$ the \emph{Hilbert entropy} is defined by $$\entropy{\sf H}{\rho}=\lim_{t\to\infty}\frac1t\log\#\Big\{[\g]\in[\gh]:\frac{(\jordan_1^\g-\jordan_3^\g)(\rho)}2\leq t\Big\}.$$

\begin{corN}[No proper actions above level sets of entropy] Consider a $\simple$-Anosov $\rho:\gh\to\SL(3,\R)$ with Zariski-dense image and  $ v\in\sf T_\rho\car(\gh,\SL(3,\R))$. If $\mathrm d\entropy{\sf H}{}(v)=0$ then the affine action on $\sl(3,\R)$ via $\co_v$ is not proper. Moreover, there is a neighborhood $\cal U$ of $(\rho,v)$ in $\sf T\car(\gh,\SL(3,\R))$ such that for all $(\eta,w)\in\cal U$ the action via $\co_w$ is also not proper.\end{corN}

The second one can be found in Corollary \ref{rootsmall}, the definition of the Hitchin component $\hitchin_\ge(S)$ of $S$ associated to a simple split $\ge$ can be found in \S\ref{IntroPress}.

\begin{corN}[Curves with arbitrary small root-variation] Let $\ge$ be simple split, $\slroot\in\simple$, $\peso_\slroot\in\a^*$ be the associated fundamental weight and $0\neq v\in\sf T_\rho\hitchin_{\ge}(S)$ have Zariski-dense base-point. Then,  there exists $h>0$ such that for positive $\eps$ and $\delta$ there exists $C>0$ with $$\#\big\{[\g]\in[\pi_1S]\textrm{ primitive}:\peso_\slroot^\g(\rho)\in(t-\eps,t]\textrm{ and }|\mathrm d\slroot^\g(v)|\leq\delta\big\}\sim C \frac{e^{ht}}{t^{3/2}}.$$ In particular, for every $\delta>0$ there exists $\g\in\pi_1S$ with arbitrary large translation length and such that $|\mathrm d\slroot^\g(v)|\leq\delta$.
\end{corN}

\subsection{Pressure forms for higher-rank Teichm\"uller spaces}\label{IntroPress}

A fundamental question in higher rank Teichm\"uller theory consists on finding an analog of the Weil-Petterson K\"ahler metric for the Hitchin component.

Let now $\ge$ be a simple split Lie algebra and $\Int\ge$ be its group of inner automorphisms. Recall from Kostant \cite{kostant} that $\ge$ contains a remarkable $\Int\ge$-conjugacy class of $\sl_2(\R)$ embeddings called \emph{the principal $\sl_2$'s}. The \emph{Hitchin component of $\ge$} (or of $\Int\ge$, or of the type of $\ge$) of a closed connected orientable surface $S$ with genus $\geq2$, is a(ny) connected component of the character variety $$\hitchin_{\ge}(S)\subset\frak X(\pi_1S,\Int\ge)$$ characterized by the following fact: there exists a discrete and faithful $\rho\in\hitchin_{\ge}(S)$ whose Zariski-closure is a principal $\PSL(2,\R)$ in $\Int\ge$. The latter representations are called \emph{Fuchsian} and the space of Fuchsian representations forms a natural embedding of the Teichm\"uller space  $\cal T(S)=\hitchin_{\A_1}(S)$ of $S$ inside $\hitchin_\ge(S).$

Hitchin \cite{Hitchin} showed that $\hitchin_\ge(S)$ is a contractible  analytic manifold and Labourie \cite{labourie} - Beyrer-Guichard-Labourie-Pozzetti-Wienhard \cite{BGLPW} show that every $\rho\in\hitchin_\ge(S)$ is faithful with discrete image (see also Fock-Goncharov \cite{FG}).

The space $\hitchin_\ge(S)$, being a subset of a surface-group character variety, is naturally equipped with Goldman's \cite{GoldmanSymplectic} symplectic form $\symp$. Moreover, Bridgeman-Canary-Labourie-S. \cite{pressure} construct, for each $\rho\in\hitchin_\ge(S)$ and each linear form $\length\in\conodual\rho$, a semi-definite symmetric bilinear form $\PP^\length_\rho$ on $\sf T_\rho\hitchin_\ge(S)$, called the \emph{$\length$-pressure form} (see \S\,\ref{thermopre} for references on similar constructions).

The question from the beginning of this section can be interpreted as a compatibility question between the pressure forms $\PP^\length$, for different choices of $\length$, and  $\symp$. Combining Labourie-Wentworth \cite{Labourie-Wentworth} with Corollary \ref{factor-cotangente} and \S\ref{ktsec} we establish:

\begin{corN}[Corollary \ref{compa1}] We let $\ge$ have type $\A$, $\B$, $\Ce$ or $\Ge$. Then there exist a unique (up to scaling) and explicit form $\kahl\in\a^*$ such that $\PP^\kahl$ is compatible with Goldman's symplectic form on $\hitchin_\ge(S)$ at the Fuchsian points.
\end{corN}

The form $\kahl$ is explicit but rather involved to compute. For the rank 2 simple split Lie algebras one has (see Remark \ref{calculokahl}), up to scaling:\begin{alignat}{2}\label{kahlrg2}\kahl_{\sl(3,\R)}(a) &=a_1-a_3-\frac{\sqrt2}{2}a_2;\nonumber\\ 
 \kahl_{\sp(4,\R)}(a) &=\big(3+\frac{\sqrt{10}}{30}\big)a_1+\big(1-\frac{\sqrt{10}}{10}\big)a_2;\nonumber\\ 
\kahl_{\gedossplit}(a) & =\Big(8 + \frac{\sqrt{42}}{315}\Big)a_1 + \Big(2 - \frac{\sqrt{42}}{210}\Big)a_2.\end{alignat} Observe that in all these cases $\kahl\in(\a^+)^*$ so Theorem \ref{thmPrinc} below implies that $\PP^\kahl$ is Riemannian on the corresponding Hitchin component. It would be interesting to understand the relation between $\PP^\kahl$ and the 1-parameter family of K\"ahler metrics on rank-$2$ Hitchin components, found by Kim-Zhang \cite{Kim-Zhang} and Labourie \cite{labourie-cyclicsurfaces}.

Although $\PP^\kahl$ is compatible with $\symp$ at Fuchsian points, it is unclear it should be compatible elsewhere. It seems moreover natural that the form $\length$ whose pressure metric is compatible with $\symp$ depends on the basepoint $\rho\in\hitchin_\ge(S).$ Thus we introduce the following definition:


\begin{defi}\label{funclong} A \emph{length function on $\hitchin_\ge(S)$} is a smooth $\mathrm{Out}(\pi_1S)$-invariant map $\uppsi:\hitchin_\ge(S)\to\a^*$ such that for all $\rho$ one has $\uppsi(\rho)\in\conodual\rho.$
\end{defi}


The forms $\PP^\length$ only depend on $\length\in\a^*$ up to scaling, so non-constant length functions are purely a higher-rank phenomenon. A natural choice is, for example, Quint's growth form: we fix a norm $\sf N$ on $\a$ and let $\uppsi(\rho)=$ the unique $\length$ in the critical hyper-surface of $\rho$ minimizing the dual norm $\sf N^*$ (see Eq. \eqref{qh1}). In this paper we prove the following (see Corollary \ref{tposlengthfunction} for the $\T$-positive case).

\begin{corN}[Corollary \ref{pathmetricHitchin}] For any length function $\uppsi:\hitchin_\ge(S)\to\a^*$ the associated pressure semi-norm $\rho\mapsto\PP^{\uppsi(\rho)}$ induces a $\Mod(S)$-invariant path metric on $\hitchin_\ge(S)$. If moreover $\ge$ has type $\A$, $\B$, $\Ce$, $\D$, or $\Ge$ and $\uppsi$ is chosen as to not verify any of the degenerations in Theorem \ref{thmPrinc}, then $\rho\mapsto\PP^{\uppsi(\rho)}$ is a $\Mod(S)$-invariant Riemannian metric on $\hitchin_\ge(S)$, as regular as $\uppsi$.
\end{corN}

The proof boils down to understanding the degeneration set of $\PP^\length$ for a fixed $\length\in\conodual\rho$. This is the content of Theorem \ref{thmPrinc} below which we now explain.

Let us chose a principal $\sl_2$, $\s$, whose semi-simple element lies in $\a$ (and with Weyl chamber contained in $\a^+$). Recall from Kostant \cite{kostant} that the decomposition of $\ge$ into irreducible $\ad\s$-factor has $\rk \ge$ factors,  each of them of odd dimension $2e+1$. The numbers $e$ appearing in this decomposition are called \emph{the exponents} of $\ge$ and we denote by $V_e$ the associated irreducible factor, so that $$\ge=\bigoplus_{e\textrm{ exponent of }\ge} V_e$$ is the decomposition of $\ge$ into irreducible $\ad\s$-factors. Table \ref{exponentes} in \S\ref{ktsec} gives the exponents for each type of $\ge.$ \begin{defi} If $e$ is an exponent of $\ge$ then we consider the line $\kt^e=V_e\cap\a$ and call it \emph{the $e$-th Kostant line}.\end{defi} The family $\{\kt^e:e\textrm{ exponent of }\ge\}$ spans $\a$.

Identifying the tangent space at $\rho$ to the character variety with the first twisted cohomology group $H^1_{\Ad\rho}(\pi_1S,\ge)$ as in Eq. \eqref{deco}, the above decomposition of $\ge$ in $\s$-modules yields a splitting at a Fuchsian point $\prin\in\hitchin_\ge(S)$,  $$\sf T_\prin\hitchin_\ge(S)=\bigoplus_{e\textrm{ exponent of }\ge} H^1_{\Ad\prin}(\pi_1S,V_e)=\bigoplus_{e\textrm{ exponent of }\ge}\tangcero\prin e,$$ where we have simplified notation and writen $\tangcero\prin e=H^1_{\Ad\prin}(\pi_1S,V_e)$.

Denote by $\ii:\a\to\a$ the opposition involution. If non-trivial, it is realized by an external involution of $\Int\ge$ that induces an involution $\underline\ii:\frak X(\pi_1S,\Int\ge)\to\frak X(\pi_1S,\Int\ge)$. Points in $\hitchin_\ge(S)$ that are fixed by $\underline\ii$ will be called \emph{self-dual}. If $\rho$ is self-dual then $\mathrm d_\rho\underline\ii$ is an involution on $\sf T_\rho\hitchin_\ge(S).$

In the special case of $\D_4$, its Dynkin diagram has an order three automorphism $\uptau$ that induces an order three automorphism $\underline\uptau$ of $\hitchin_{\D_4}(S)$. Also, in this case $3$ appears twice as an exponent (see Table \ref{exponentes}), let us denote by $V_{3,\mathrm a}$ the $\ad\s$-factor that is not contained in the natural representation $\so(3,4)\to\so(4,4)$ preserving a non-isotropic line, see \S\,\ref{tricota} for details.

With these notations at hand we can completely describe the degenerations of pressure forms on the Hitchin component of classical type.

\begin{thmA}[Pressure degenerations are Lie-theoretic]\label{thmPrinc} Let $\ge$ be simple split of type $\A$, $\B$, $\Ce$, $\D$ or $\Ge$. Consider $\rho\in\hitchin_\ge(S)$ and $\length\in\conodual\rho$ then, the pressure form $\PP^\length_\rho$ is degenerate at $\tangente\in\sf T_\rho\hitchin_\ge(S)$ if and only if either of the following hold:\begin{itemize}

\item[-] $\rho$ is Fuchsian and $${\displaystyle \tangente\in\bigoplus_{e:\length(\kt^e)=0}\tangcero{\rho}{e}},$$

\item[-] $\rho$ is self dual, $\length$ is $\ii$-invariant and $\tangente$ is $\mathrm d_\rho\overline\ii$-anti-invariant. 

\item[-] $\ge$ is of type $\sf A_6$ or $\Ce_3$, the Zariski closure of $\rho(\pi_1S)$ is conjugate to the $7$-dimensional irreducible representation of the real split group $\sf G_2$, $\length(\kt^3)=0$ and $\tangente\in H^1_{\Ad\rho}(\pi_1S,V_3)$.

\item[-] $\ge$ is of type $\sf D_4$, the Zariski closure of $\rho(\pi_1S)$ has Lie algebra conjugate to the spin representation $\so(3,4)\to\so(4,4)$, $\tangente\in H^1_{\Ad\rho}\big(\pi_1S,\underline\uptau (V_{3,\mathrm a})\big)$ and $\length(-1,1,1,-1)=0.$

\end{itemize}

\end{thmA}

For example, in $\PSL(4,\R)$ the Kostant lines are \begin{alignat*}{2}
\kt^1 & = \R\cdot(3,1,-1,-3),\\ 
\kt^2 & =\R\cdot(1,-1,-1,1),\\ 
\kt^3 &= \R\cdot(1,-3,3,-1),
\end{alignat*} so the strongly dominant weight $2\peso_1+\peso_2:a\mapsto3a_1+a_2$ contains $\kt^3$ in its kernel. Consequently, Theorem \ref{thmPrinc} implies that the Pressure form $\PP^{2\peso_1+\peso_2}$ degenerates \emph{only} at the Fuchsian locus and in the directions given by $\tangcero\rho3.$ See Table \ref{ejemplosKV} in \S\ref{ktsec} for the list of Kostant lines on $\sl(d,\R)$ for $d\leq8$.

\begin{obs} The question of non-degeneration on the Hitchin component has been  dealt with in the previously mentioned work by Bridgeman-Canary-Labourie-S. \cite{pressure}, where it is stablished that:\begin{itemize}\item[-] if we let $\peso_1$ be the first fundamental weight $\peso_1:a\mapsto a_1$, then $\PP^{\peso_1}$ is Riemannian on $\hitchin_{\sl(d,\R)}(S)$ and, \item[-] for every strongly dominant weight $\chi$ the form $\PP^{\chi}$ is Riemannian on the space $\{\rho\in\hitchin_\ge(S):\rho(\pi_1S)\textrm{ is Zariski-dense}\}$;
\end{itemize}
and in B.-C.-L.-S. \cite{Liouvillepressure} where it is stablished that $\PP^{\slroot_1}$ is Riemannian, where $\slroot_1:a\mapsto a_1-a_2$ is the first simple root. In \cite{pressure} it is also established the following (to be compared with Theorem \ref{teoB}): \emph{Let $\rho_t:\G\to\SL(d,\R)$ be an analytic curve of irreducible representations with speed $v$ and assume there exist $g,h\in\G$ such that $\rho_0(g)$ and $\rho_0(h)$ are bi-proximal and transverse (see Def. \ref{stronglytransversal}), assume also that $\mathrm d\peso_1^g(v)\neq0$, then the set of pairs $$\big\{(\mathrm d \peso_1^\g(v),\peso_1^\g(\rho)\big):\g\in\G\textrm{ with $\rho_0(\g)$ proximal}\big\}\subset\R^2$$ is not contained in a line.}

Finally, pressure forms for $\peso_1$ and $\slroot_1$ are Riemannian on the Hitchin components of geometrically finite Fuchsian groups by Bray-Canary-Kao-Martone \cite{pressurecusped}.
\end{obs}

\subsection{Hausdorff dimension of higher-quasi-circles} We use Theorem \ref{thmPrinc} to study deformations of higher-rank Teichm\"uller spaces inside the complexified group.

Let us fix $\rho\in\hitchin_\ge(S)$. Then it follows from Labourie \cite{labourie} that there exists a $\rho$-equivariant H\"older-continuous map \begin{equation}\label{aquim}\zeta_\rho:\bord\pi_1S\to\EuScript F_\simple(\Int\ge)\end{equation} from the Gromov-boundary of $\pi_1S$ to the full flag space of $\Int\ge$. While the circle $\zeta(\bord\pi_1S)\subset\EuScript F_\simple(\Int\ge)$ is Lipschitz, each of its projections into the maximal flags $$\mathbf{L}_{\rho,\sroot}:=\zeta^{\sroot}_\rho(\bord\pi_1S)\subset\EuScript F_{\sroot}(\Int\ge),$$ for $\sroot\in\simple$, is a $\class^{1+\alpha}$ circle (Labourie \cite{labourie} and Pozzetti-S.-Wienhard \cite{PSW1}). In this paper we deform these circles inside the complex maximal flag variety $\EuScript F_{\sroot}(\Int(\ge_\C))$. More precisely, the equivariant map from Equation \eqref{aquim} can also be defined for representations neighboring $\rho$ in the complex characters $\frak X\big(\pi_1S,\Int(\ge_\C)\big)$ (\cite{labourie}). Moreover, there exists a neighborhood $\EuScript U$ of $\hitchin_\ge(S)\subset\frak X\big(\pi_1,\Int(\ge_\C)\big)$ such that for every $\eta\in\EuScript U$ and every $\sroot\in\simple$ the function \begin{alignat*}{2}\Hff_\sroot:\,&\EuScript U\to[1,\infty)\\ &\eta\mapsto\Hff({\bf{L}}_{\eta,\sroot})\end{alignat*} is real-analytic\footnote{due to \cite{PSW1} together with Bridgeman-Canary-Labourie-S. \cite{pressure}}, where $\Hff$ denotes the Hausdorff dimension for a Riemannian metric on $\EuScript F_{\sroot}(\Int(\ge_\C))$. We can thus study its Hessian at the critical points $\hitchin_\ge(S).$

If we let $\sf J$ be the tensor squaring $-\id$  on the complex characters $\frak X(\pi_1S,\Int(\ge_\C))$, induced by the complex structure of $\Int(\ge_\C)$, then the tangent space at a Hitchin point $\rho\in\hitchin_\ge(S)$ naturally splits as $$\sf T_\rho\frak X\big(\pi_1S,\Int(\ge_\C)\big)=\sf T_\rho\hitchin_\ge(S)\oplus \sf J\Big(\sf T_\rho\hitchin_\ge(S)\Big).$$ Deformations along $\sf T\hitchin_\ge(S)$ are understood, $\Hff_\sroot\equiv1$, so we turn into the complementary factor. By means of Bridgeman-Pozzetti-S.-Wienhard \cite{HessianHff} and Theorem \ref{thmPrinc}, in \S\,\ref{Hffdeg} we establish the following (an analogous result for $\T$-positive representations can be found on Corollary \ref{HffTpos}):

\begin{corN}[Corollary \ref{corcho}] Let $v\in\sf T_\rho\hitchin_\ge(S)$ be non-zero and have Zariski-dense base point, then $$\hess_\rho\Hff_\sroot(\sf Jv)>0.$$ In particular there exists a neighborhood (in the complex characters)  of points in $\hitchin_\ge(S)$ with Zariski-dense image where $\Hff_\sroot$ is rigid, i.e. such that if $\Hff_\sroot(\eta)=1$ then $\eta$ has values in the real characters.
\end{corN}

The second statement is a local analog of a classical result of Bowen \cite{bowen-quasicircles} who deals with $\ge$ of type $\A_1$. The first one is inspired by Bridgeman-Tayor \cite{WP-QF} and McMullen \cite{McMWP}, again for $\PSL(2,\C)$. We note that a higher-rank version of Bowen's Theorem, by considering the Hausdorff dimension in the full flag variety, has been recently obtained by Farre-Pozzetti-Viaggi \cite{FPVBowen}.

Actually we can be much more precise. Zariski closures of Hitchin representations have been classified (Guichard \cite{clausura}, S. \cite{clausurasPos}) so we can also look at the \emph{intermediate strata}. It turns out that the most subtle situation is actually the Fuchsian case, which we now explain, the complete picture can be found on \S\,\ref{Hffdeg}.

\begin{corN}[Corollary \ref{corcho}]Let $v\in\sf T_\prin\hitchin_\ge(S)$ be tangent to a Fuchsian representation $\prin$. If $v\in\tangcero\prin e$ and $\kt^e\subset\ker\sroot$, then $\hess_\prin\Hff_\sroot(\sf Jv)=0.$ If $\ge$ has classical type then the converse is also true: if $\hess_\prin\Hff_\sroot(\sf Jv)=0$ then $v\in\bigoplus_{e:\kt^e\subset\ker\sroot}\tangcero\prin e$.
\end{corN}

For example, when $\ge=\sl(d,\R)$ with the standard Cartan subspace $\a=\{a\in\R^d:\sum a_i=0\}$ and simple roots $\slroot_j(a)=a_j-a_{j+1}$; the above corollary reduces the question of understanding $\hess_\prin\Hff_{\slroot_j}\sf J v=0$ to describing the triples of integers $(d,e,j)$ such that $\kt^e\subset\ker\slroot_j$. This condition can be rephrased in terms of an explicit Diophantine equation (see Equation \eqref{diofanto}).  For example, understanding  degenerations on the second Grassmannian $\zeta^2(\bord\pi_1S)\subset\grass_2(\R^d)\subset\grass_2(\C^d)$ reduces to the (elementary) equation $$d-1=\frac{e(e+1)}{2},$$ so $\hess\Hff_{\slroot_2}$ is only degenerate when $d=4,7,11,16...$ However, for the remaining Grassmannians the equation is more involved. For $\grass_3(\C^d)$ the equation becomes $$e^4-6de^2+2e^3+6d^2-6de+11e^2-18d+10e+12=0,$$ which turns out to be a genus-$1$ complex curve, whose integer solutions can be completely described via the Elliptic logarithm method $\frak{Ellog}$ (see for example Tzanakis \cite{TzanakisBook}), so we obtain in \S\,\ref{3raraiz}:

\begin{cor}\label{Hff3} Consider $v\in\sf T_\rho\hitchin_{\sl(d,\R)}(S),$ then one has $\hess_{\rho}\Hff_{\slroot_3}(\sf Jv)=0$ if and only if one of the following holds : \begin{itemize}\item[-]  $d=6$, $\rho$ has values in $\PSp(6,\R)$ and $v\in\tangcero\rho2\oplus\tangcero\rho4$, \item[-] $d=17$, $\rho$ is Fuchsian and $v\in\tangcero\rho4\oplus\tangcero\rho8$,\item[-] $d=58$, $\rho$ is Fuchsian and $v\in\tangcero\rho8$.\end{itemize}
\end{cor}

The Diophantine equation associated to the 4th root is 

$$11d - \frac{13}{2}e - \frac{23}{3}e^2- \frac{19}{8}e^3 - \frac{31}{24}e^4 - \frac{1}{8}e^5 - \frac{1}{24}e^6 + \frac{13}{2}de  - \frac{3}{2}d^2e + 7de^2 - \frac{3}{2}d^2e^2 + de^3 + \frac{1}{2}de^4 + d^3 - 6d^2=6,$$ which according to Maple is a genus-4 complex curve with one singular point. To our knowledge, no general method to  solve this kind of equation over $\Z$ is known.

\begin{ack} I would like to thank Luca Battistella, Martin Bridgeman, Richard Canary, Jeffrey Danciger, Fran\c cois Labourie,  Marco Macullan, Alejandro Passeggi, Rafael Potrie, Germain Poullot and Maria Beatrice Pozzetti for enlightening discussions related to this paper. 
\end{ack}

\section{Preliminaries}


Throughout this paper we will let $\sf G$ be the real points of a Zariski-connected semi-simple real-algebraic group of the non-compact type with Lie algebra $\ge$. We will also let $V$ be a finite-dimensional real vector space, $\G$ a semi-group and $\gh$ a finitely generated word-hyperbolic group.

\subsection{Notations from Lie theory}Let us fix  $o:\ge\to\ge$ a Cartan involution with associated Cartan decomposition $\ge=\k\oplus\p.$ Let $\a\subset\p$ be a maximal abelian subspace and let $\root\subset\a^*$ be the set of restricted roots of $\a$ in $\ge.$ For $\aa\in\root$ let us denote by \begin{equation}\label{rootspace}\ge_\aa=\{u\in\ge:[a,u]=\aa(a)u\,\forall a\in\a\}\end{equation} its associated root space. One has the (restricted) root space decomposition $\ge=\ge_0\oplus\bigoplus_{\aa\in\root}\ge_\aa,$ where $\ge_0$ is the centralizer of $\a.$ Fix a Weyl chamber $\a^+$ of $\a$ and let $\root^+$ and $\simple$ be, respectively, the associated sets of positive roots and of simple roots.

Let us denote by $(\cdot,\cdot)$ the Killing form of $\ge,$ its restriction to $\a,$ and its associated dual form in the dual of $\a,$ $\a^*.$ For $\chi,\psi\in\a^*$ define $$\<\chi,\psi\>=2\frac{(\chi,\psi)}{(\psi,\psi)}.$$ 

The \emph{Weyl group} of $\root,$ denoted by $\Weyl,$ is the group generated by, for each $\aa\in\root,$ the reflection $\re_\aa:\a^*\to\a^*$ on the hyperplane $\aa^\perp,$ $$\re_\aa(\chi)=\chi-\<\chi,\aa\>\aa.$$ It is a finite group with a unique \emph{longest element} $\longest$ (w.r.t. the word metric on the generating set $\{\re_\aa:\aa\in\simple\}$). This longest element sends $\a^+$ to $-\a^+.$

The \emph{restricted weight lattice} is defined by $$\poids=\{\varphi\in\a^*:\<\varphi,\aa\>\in\Z\ \forall\aa\in\root\},$$ it is spanned by the \emph{fundamental weights}: $\{\peso_\sroot:\sroot\in\simple\}$ where $\peso_\sroot$ is defined by $\<\peso_\sroot,\bb\>=d_\sroot\delta_{\sroot\bb}$ for every $\sroot,\bb\in\simple,$ where $d_\sroot=1$ if $2\sroot\notin\root^+$ and $d_\sroot=2$ otherwise.
The set $\poids_+$ of \emph{dominant restricted weights} is defined by $\poids_+=\poids\cap (\a^+)^*.$

If $\ge$ has \emph{reduced} root system, then it is customary to denote by $$\dd=\frac12\sum_{\aa\in\root^+}\aa=\sum_{\sroot\in\simple}\peso_\sroot,$$ where the last equality can be found in Humphreys \cite[\S 13.3]{james}. For every $\sroot\in\simple$ one has $\<\dd,\sroot\>=1$ (\cite[\S 10.2 ]{james}).

\subsection{Some $\sl_2$'s of $\ge$}\label{xalpha} For $\varphi\in\a^*$ let $u_\varphi\in\a$ be such that for all $v\in\a$ one has $$\varphi(v)=(u_\varphi,v).$$

For $\aa\in\root$ let $h_\aa\in\a$ be defined by, for all $\varphi\in\a^*$, $\varphi(h_\aa)=\<\varphi,\aa\>.$ The vectors $u_\aa$ and $h_\aa$ are related by the simple formula $h_\aa=2u_\aa/(u_\aa,u_\aa).$ Recall that for $x\in\ge_\aa$ one has $[x,o(x)]=(x,o(x))u_\aa$. Thus, for each $\aa\in\root^+$ and $\xx_\aa\in\ge_\aa$ there exists $\yy_\aa\in\ge_{-\aa}$ such that $$(\begin{smallmatrix}0 & 1 \\ 0 & 0\end{smallmatrix}) \mapsto \xx_\alpha,\quad (\begin{smallmatrix}0 &  0\\ 1 & 0\end{smallmatrix})  \mapsto \yy_\alpha \ \textrm{ and }\ (\begin{smallmatrix}1 & 0 \\ 0 & -1\end{smallmatrix}) \mapsto h_\alpha, $$

\noindent
is a Lie algebra isomorphism between $\sl_2(\R)$ and the span of $\{\xx_\aa,\yy_\aa,h_\aa\}.$ Let us fix such a choice of $\xx_\aa$ and $\yy_\aa$ from now on.

One says that $\ge$ is \emph{split} if the complexification $\a\otimes\C$ is a Cartan subalgebra of $\ge\otimes\C.$ Equivalently, $\ge$ is split if the centralizer $\centro_\k(\a)$ of $\a$ in $\k$ vanishes. Assume that $\ge$ is split. Following Kostant \cite[\S 5]{kostant}, consider the dual basis of $\{u_\sroot:\sroot\in\simple\}$ relative to $(\cdot,\cdot)$: $(\epsilon_\aa,u_\bb)=\delta_{\aa\bb},$ and let $\epsilon_0=\sum_{\sroot\in\simple}\epsilon_\sroot\in\a.$ Upon writing $$2\epsilon_0  =\sum_{\sroot\in\simple}r_\sroot u_\sroot$$ for some $r_\sroot\in\Z_{>0},$ define ${\displaystyle e^+=  \sum_{\sroot\in\simple}\xx_\sroot\ \textrm{ and }\ e^-=  \sum_{\sroot\in\Delta}r_\sroot \yy_\sroot.}$ Since $(2\epsilon_0,u_\sroot)=2$ for every $\sroot\in\simple,$ the identification $$(\begin{smallmatrix}0 & 1 \\ 0 & 0\end{smallmatrix}) \mapsto e^+, \quad (\begin{smallmatrix}0 &  0\\ 1 & 0\end{smallmatrix})  \mapsto e^-\ \textrm{ and }\   (\begin{smallmatrix}1 & 0 \\ 0 & -1\end{smallmatrix}) \mapsto 2\epsilon_0,$$

\noindent
is also a Lie algebra isomorphism between $\sl_2(\R)$ and the span $\s$ of $\{e^+,e^-,2\epsilon_0\}.$ The Lie algebra $\s$ is called \emph{a principal $\sl_2$} and the $\Int\ge$-conjugacy class of this representation is called \emph{the principal $\sl_2$} of $\ge.$

\subsection{Cartan decomposition}

Let $\sf K\subset\sf G$ be a compact group that contains a representative for every element of the Weyl group $\Weyl.$ This is to say, such that the normalizer $N_{\sf G}(\sf A)$ verifies $N_{\sf G}(\exp\a)=(N_{\sf G}(\exp\a)\cap \sf K)\exp\a.$ Cartan's decomposition asserts the existence of a map, called the \emph{Cartan projection} of $\sf G,$  $$\cartan:\sf G\to \a^+$$ such that for every $g\in\sf G$ there exist $k,l\in\sf K$ with $g=k\exp(\cartan(g))l$ and such that $\forall g_1,g_2\in \sf G$ one has $g_1\in \sf Kg_2 \sf K$ if and only if $\cartan(g_1)=\cartan(g_2).$

\subsection{Jordan decomposition}\label{jordan}It states that every $g\in\sf G$ can be written as a commuting product $g=g_eg_hg_n,$ where $g_e$ is elliptic, $g_n$ is unipotent and $g_h$ is conjugate to en element $z_g\in \exp(\a^+)$. We let $\jordan(g)=\cartan(z_g)\in\a^+,$ and the map $\jordan:\sf G\to\a^+$ is called \emph{the Jordan projection} of $\sf G.$

\subsection{Flag manifolds of $\sf G$}

A subset $\t\subset\simple$ determines a pair of opposite parabolic subgroups $\sf P_\t$ and $\check{\sf P}_\t$ whose Lie algebras are defined by \begin{alignat*}{2}\frak p_\t & =\bigoplus_{\sroot\in\root^+\cup\{0\}}\ge^{\sroot} \oplus\bigoplus_{\sroot\in\<\simple-\t\>}\ge^{-\sroot},\\ \check{\frak p}_\t & =\bigoplus_{\sroot\in\root^+\cup\{0\}}\ge^{-\sroot} \oplus\bigoplus_{\sroot\in\<\simple-\t\>}\ge^{\sroot}.\end{alignat*}

\noindent
The group $\check{\sf P}_\t$ is conjugated to the parabolic group $\sf P_{\ii\t}.$  We denote the \emph{flag space} associated to $\t$ by $\EuScript F_\t=\sf G/\sf P_\t.$ The $\sf G$ orbit of the pair $([\sf P_{\t}],[\check{\sf P}_{\t}])$ is the unique open orbit for the action of $\sf G$ in the product $\EuScript F_\t\times\EuScript F_{\ii\t}$ and is denoted by $\posgen_\t.$

\begin{obs} If $g\in\sf G$ is such that $\sroot(\jordan(g))>0$ for all $\sroot\in\t$ then $g$ acts proximally on $\EuScript F_\t$.We will denote by $(g^+, g^-)\in\posgen_{\vt}(\sf G)$ the corresponding \emph{attracting and repelling flags}, so that every flag $y\in\EuScript F_{\vt}(\sf G)$ in general position with $ g^-$ verifies $ g^ny\to  g^+.$
\end{obs}

\subsection{The center of the Levi group $\sf P_{\t}\cap\check{\sf P}_{\t}$}\label{s.Levi}

We now consider ${\displaystyle \a_\t=\bigcap_{\sroot\in\simple-\t}\ker\sroot.}$

Denoting by $\Weyl_\t=\{w\in \Weyl:w(v)=v\quad \forall v\in\a_\t\}$ the subgroup of the Weyl group generated by reflections associated to roots in $\simple-\t$, there is a unique projection $\pi_\t:\a\to\a_\t$ invariant under  $\Weyl_\t$.

The dual $(\a_\t)^*$ is canonically identified with the subspace of $\a^*$ of $\pi_\t$-invariant linear forms. Such space is spanned by the fundamental weights of roots in $\t$,
\begin{equation}\label{at*}(\a_\t)^*=\{\varphi\in\a^*:\varphi\circ\pi_\t=\varphi\}=\<\peso_\sroot|\a_\t:\sroot\in\t\>.\end{equation} 
We consider the projections $\cartan_\t =\pi_\t\circ \cartan:\sf G\to \a_\t$ and $\jordan_\t  =\pi_\t\circ\jordan:\sf G\to \a_\t.$


\subsection{Representations of $\sf G$}\label{representaciones}
The standard references for the following are Fulton-Harris \cite{FultonHarris}, Humphreys  \cite{james} and Tits \cite{Tits}. Let $\rep:\sf G\to\GL(V)$ be a finite dimensional rational\footnote{Namely a rational map between algebraic varieties.} representation. We also denote by $\rep:\frak g\to\gl(V)$ the Lie algebra homomorphism associated to $\rep.$  If $\rep:\ge\to\gl(V)$ is irreducible then we say that $\rep:\sf G\to\GL(V)$ is \emph{strongly irreducible}. Equivalently $\rep(\sf G)$ does not preserve finitely many subspaces of $V$. Since $\sf G$ is Zariski-connected, an irreducible representation is strongly irreducible.

The \emph{weight space} associated to $\chi\in\a^*$ is the vector space 
$$V^\chi=\{v\in V:\rep(a) v=\chi(a) v\ \forall a\in\sf A \}.$$ 

\noindent
We say that $\chi\in\a^*$ is a \emph{restricted weight} of $\rep$  if $V\chi\neq0$ and let $$\poids_\rep=\big\{\chi\in\a^*:V\chi\neq\{0\}\big\}.$$ Tits \cite[Theorem 7.2]{Tits} (see also \cite[\S13.4 Lemma B]{james})) states that $\poids_\rep$ has a unique maximal element with respect to the partial order $\chi\succ\psi$ if $\chi-\psi$ belongs to the semi-group spanned by $\simple.$ This is called \emph{the highest weight} of $\rep$ and denoted by $\chi_\rep.$ Thus, every $\chi\in\poids_\rep$ has the form $$\chi_\rep-\sum_{\sroot\in\simple} n_\sroot\sroot\textrm{ for }n_\sroot\in\N.$$
If we let  $\jordan_1$ be the logarithm of the spectral radius then for every $g\in\sf G$ one has
\begin{equation}\label{eq:spectralrep}
\jordan_1\big(\rep(g)\big)=\chi_\rep(\jordan(g)\big).
\end{equation}

\begin{defi}\label{trep}
	Let $\rep:\sf G\to\PGL(V)$ be a representation. We denote by $\t_\rep$ the set of simple roots $\sroot\in\simple$ such that $\chi_\rep-\sroot$ is still a weight of $\rep$. Equivalently 
	 \begin{equation}\label{peso<}
		\t_\rep=\big\{\sroot\in\simple:\<\chi_\rep,\sroot\>\neq0\big\}.
	\end{equation}
\end{defi}

We denote by $\|\,\|_\rep$ an Euclidean norm on $V$ invariant under $\rep \sf K$ and such that $\rep\exp\a $ is self-adjoint, see Benoist-Quint \cite[Lemma 6.33]{Benoist-QuintLibro}. We also let $\|\,\|_\rep$ be the induced operator norm. By definition of $\chi_\rep$, $\|\,\|_\rep$ and Eq.\eqref{eq:spectralrep}, for every $g\in \sf G $
\begin{equation}\label{eq:normayrep}
\log\|\rep g\|_\rep=\chi_\rep(\cartan(g)).
\end{equation}

Denote by $W_{\chi_\rep}$ the $\rep\sf A $-invariant complement of $V_{\chi_\rep}.$ The stabilizer in $\sf G $ of $W_{\chi_\rep}$ is $\wk{\sf P}_{\t_\rep},$ and thus one has a map of flag spaces 
\begin{equation}\label{maps}
	(\zeta_\rep,\zeta^*_\rep):\EuScript F_{\t_\rep}^{(2)}(\sf G )\to \mathrm{Gr}_{\dim V_{\chi_\rep}}^{(2)}(V).
\end{equation} This is a proper embedding which is an homeomorphism onto its image. Here, as above, $\mathrm{Gr}_{\dim V_{\chi_\rep}}^{(2)}(V)$ is the open $\PGL (V )$-orbit in the product of the Grassmannian of $(\dim V_{\chi_\rep})$-dimensional subspaces and the Grassmannian of $(\dim V-\dim V_{\chi_\rep})$-dimensional subspaces. One has  the following (see Humphreys \cite[Chapter XI]{LAG}).

\begin{prop}[Tits \cite{Tits}]\label{prop:titss} 
For each $\sroot\in\simple$ there exists a finite dimensional rational irreducible representation $\rep_\sroot:\sf G\to\PSL(V_\sroot),$ such that $\chi_{\rep_\sroot}$ is an integer multiple $l_\sroot\peso_\sroot$ of the fundamental weight and $\dim V_{\chi_{\rep_\sroot}}=1.$ \end{prop}

We will fix from now on such a set of representations and call them, for each $\sroot\in\simple,$ the \emph{Tits representation associated to $\sroot$}.

\subsection{Gromov product, additive cross ratio and representations}\label{angulosCasa}\label{RD}  Consider $\t\subset\simple$. Recall from \cite{orbitalcounting} that the \emph{Gromov product}
$
(\mathord{\cdot}|\mathord{\cdot}):\posgen_\t\to\a_\t
$
is defined by
$$
\peso_\alpha\big((x|y)\big) = -\log\sin\angle_{\rep_\alpha o}(\xi_{\rep_\alpha}x,\xi^*_{\rep_\alpha}y) = -\log\frac{|\varphi(v)|}{\|\varphi\|_{\rep_\alpha o}\|v\|_{\rep_\alpha o}}\geq0
$$
for all $\alpha\in\t$, where $v\in\xi_{\rep_\alpha}x-\{0\}$ and $\ker\varphi = \xi_{\rep_\alpha}^*y$. Note that 
\begin{equation}\label{minangulo}\max_{\alpha\in\t}\peso_\alpha((x|y)_o) = -\log\min_{\alpha\in\t}\sin\angle_{\rep_\alpha o}(\xi_{\rep_\alpha}x,\xi^*_{\rep_\alpha}y).
\end{equation} and that, since $\{\peso_\alpha|\a_\t\}_{\alpha\in\t}$ is a basis of $(\a_\t)^*$, the right hand side of equation (\ref{minangulo}) is comparable to the norm $\|(x|y)\|$. One has the following:

\begin{obs}\label{GromovyRep} Let $\rep:\sf G\to\PSL(V)$ be a proximal  irreducible representation. If $(x,y)\in\posgen_{\t_\rep}$ then
${\displaystyle
\big(\Xi_\rep x |\Xi_\rep^* y\big)_{\rep o} = \chi_\rep\big((x|y)\big) = \sum_{\alpha\in\t_\rep}\<\chi_\rep,\sroot\>\,\peso_\sroot\big((x|y)\big).
}$
\end{obs}

Consider now the space $\EuScript F_\t^{(4)}$ of pairs $(x,y),(z,t)\in\EuScript F_\t^{(2)}$ with the extra transversality condition that both pairs $(x,t)$ and $(z,y)$ are also in general position. The (\emph{aditive}) \emph{cross ratio} is the $\sf G$-invariant map $\cross_\t:\EuScript F_\t^{(4)}\to\a_\t$ defined by $$\cross_\t(x,y,z,t)=(x|y)-(x|t)+(z|t)-(z|y).$$

\subsection{Proximality}\label{proxBenoist}


Recall that $g\in\End(V)$ is \emph{proximal} if it has a unique eigenvalue with maximal modulus and it has multiplicity $1.$ The associated eigenline is denoted by $g^+\in\P( V)$ and $g^-$ is its $g$-invariant complementary subspace. We consider then $\beta_g\in V^*$ and $v_g\in g^+$ such that $\ker\beta_g=g^-$ and $\beta_g(v_g)=1,$ we also let $\pi_g(w)=\beta_g(w)v_g$ be the projection over $g^+$ with kernel $g^-.$ We finally let $V_2(g)$ be the generalized eigenspace of $g$ associated to the second (in modulus) eigenvalue and $\tau_g$ be the only projection over $V_2(g)$ whose kernel is $g$-invariant.

\begin{defi}\label{stronglytransversal}If $g,h$ are proximal, we say they are \emph{transversally proximal} if $\beta_g(v_h)\beta_h(v_g)\neq0$ and \emph{strongly transversally proximal} if moreover $\beta_h(\tau_gv_h)\neq0.$
\end{defi}

We say that $g\in\sf G$ is $\t$-\emph{proximal} if $\rep_\sroot(g)$ is proximal $\forall\sroot\in\t.$ For such $g$ there exists $(g^-_\t,g^+_\t)\in\posgen_\t,$ defined by, for every $\sroot\in\t,$ $\Xi_{\rep_\sroot}\big(g^+_\t\big)=\rep_\sroot(g)^+.$ Every flag $x\in\EuScript F_\t$ in general position with $g^-_\t$ verifies $g^nx\to g^+_\t.$ Two elements $g,h\in\sf G$ are \emph{transversally $\t$-proximal} if they are $\t$-proximal and moreover $\{(g^+_\t,h^-_\t),(h^+_\t,g^-_\t)\}\subset\posgen_\t$. One has the following from Benoist \cite{limite}.

\begin{thm}[{Benoist \cite{limite}}]\label{CR} Let $g,h\in\sf G$ be transversally $\t$-proximal then, $$\lim_{n\to\infty}\jordan_\t(g^nh^n)-\jordan_\t(g^n)-\jordan_\t(h^n)=\cross_\t(g^-_\t,g^+_\t,h^-_\t,h^+_\t)=:\crossg_\t(g,h).$$\end{thm}

\begin{lemma}[{Benoist-Quint \cite[Lemma 7.10]{Benoist-QuintLibro}}]\label{pares} Let  $g,h\in\sf G$ be loxodromic, then there exists a non-empty Zariski-open subset $\mathrm G_{gh}$ of $\sf G^2$ such that whenever $(f,q)\in\mathrm G_{gh}$ the limit $$\lim_{\min\{n,m\}\to\infty}\jordan(g^mfh^nq)-m\jordan(g)-n\jordan(h)$$ exists. If $g$ and $h$ are transversally loxodromic, then $(\id,\id)\in\mathrm G_{gh}$.
\end{lemma}

\begin{lemma}\label{modif}Let $g\in\sf G$ be loxodromic, then there exists a non-empty open subset $\mathrm G_g$ of $\sf G^2$ such that whenever $(f,q)\in\mathrm G_g$ the following limit exists: $$\lim_{n\to\infty}\jordan(fg^nq)-n\jordan(g).$$
\end{lemma}

\begin{proof}The equation $\big(f(g^+),q(g^-)\big)\notin\posgen_\simple(\sf G)$ with variables $f$ and $q$, is described by polynomials, so its complement $\mathrm G_g$ is a (non-empty) Zariski-open set of $\sf G^2$. We commence by writing, for every $\sroot\in\simple$, $$\rep_\sroot(g^n)=\Big(\pm\exp\big(\peso_\sroot(\jordan(g))\big)\Big)^n\pi_{g,\sroot}+\check P_{n},$$ where $\pi_{g,\sroot}$ is the projection with image $\zeta_\sroot(g^+)$ and kernel $\zeta_\sroot^*(g^-)$ and where the spectral radius of $\check P_n$ is $\leq\exp(n(\peso_\sroot-\sroot)(\jordan(g)))$. Thus, if $(f,q)\in\mathrm G_g$  one has $$\lim_{n\to\infty}\frac{\rep_\sroot\big(fg^nq\big)}{(\pm)^n\exp(\peso_\sroot(n\jordan(g)))}=f\pi_{g,\sroot}q.$$ By the condition on $(f,q)$, one has $\traza(f\pi_{g,\sroot}q)\neq0$, thus for big enough $n$, $\rep_\sroot(fg^nq)$ is proximal and together with the above we get that $\jordan_1\rep_\sroot\big((fg^nq)\big)-n\peso_\sroot(\jordan(g))$ converges as $n\to\infty$. Since this holds $\forall\sroot\in\simple$ the lemma follows.\end{proof}

Given $r,\eps$ positive we say that $g$ is $(r,\eps)$-\emph{proximal} if it is proximal, $$\|(g^-|g^+)\|\leq r$$ and for every $x\in\EuScript F_\t$ with $\|(g^-|x)\|\leq \eps^{-1}$ one has $d_{\EuScript F_\t}(gx,g^+)\leq\eps.$ The following is from Benoist \cite[Corollaire 6.3]{Benoist-HomRed}, a proof can also be found in S. \cite[Lemma 5.6]{quantitative}.

\begin{thm}[{Benoist \cite{Benoist-HomRed}}]\label{proxCartan}For every $\delta>0$ there exist $r,\eps>0$ such that if $g\in\sf G$ is $(r,\eps)$-proximal then $\big\|\cartan_\t(g)-\jordan_\t(g)+(g_\t^-|g^+_\t)\big\|\leq\delta.$
\end{thm}

\begin{lemma}\label{contraccionproximal} There exists $C$ only depending on $\sf G$ such that for every $\t$-proximal $g\in\sf G$ and a flag $B\in\EuScript F_{\t}$ transverse to $g^-$ one has $$d(gB,g^+)\leq Ce^{(g^-|B)}e^{-\min\{\sroot(\jordan(g)):{\sroot\in\t}\}+\|(g^-|g^+)\|}.$$
\end{lemma}

\begin{proof} See, for example, Bochi-Potrie-S. \cite[Lemma A.6]{BPS}.\end{proof}

We record the following lemma that will be needed in \S\,\ref{cocycleviewpoint}.

\begin{lemma}\label{esquiva}Consider a finite collection of irreducible representations $\{\rep_i:\sf G\to V_i\}_{i\in I}$. For each $i\in I$ fix $v_i\in V_i-\{0\}$ and a non-empty finite collection of strict subspaces $\{W_k^i:k\in F_i\}$ on $V_i$. Then $\big\{ g\in\sf G:\forall i\in I\ gv_i\notin\bigcup_{k\in F_i} W_k\big\}$ is Zariski-open and non-empty.
\end{lemma}

\begin{proof} The lemma follows as the required equations are described by polynomials. Moreover, since as $\sf G$ is a group, Zariski-connected implies Zariski-irreducible; thus a finite familly of Zariski-open non-empty subsets has non-empty intersection.\end{proof}
 

\subsection{Zariski-dense sub-semigroups of $\sf G$}

Let $\grupo<\sf G$ be a semi-group. Its \emph{limit cone} is $\Bcone_\grupo=\overline{\{\R_+\jordan(g):g\in\grupo\}}\subset\a^+.$ One has the following fundamental result:

\begin{thm}[Benoist \cite{limite,benoist2}]\label{interiorsi} Let $\grupo<\sf G$ be a Zariski-dense sub-semi-group, then $\Bcone_\grupo$ is convex and has non-empty interior. Moreover, the group spanned by $\{\jordan(g):g\in\grupo\}$ is dense in $\a$.
\end{thm}

We will moreover need the following:

\begin{prop}[{Benoist \cite[Proposition 5.1]{limite}}]\label{subsemi} Let $\grupo<\sf G$ be a Zariski-dense sub-semi-group and let $\scr C\subset\Bcone_\grupo$ be a closed convex cone with non-empty interior. Then there exists a Zariski-dense sub-semi-group $\grupo'<\grupo$ such that $\Bcone_{\grupo'}=\scr C$. If $\scr C$ is $\ii$-invariant then $\grupo'$ can be furthermore chosen to be a group. \end{prop}

Moreover, the proof of Benoist \cite[Lemma 4.3]{limite} gives:

\begin{obs}\label{G'schottky}If $g\in\grupo$ is loxodromic and $\scr C$ is a convex closed cone with non-empty interior and $\jordan(g)\in\scr C\subset\Bcone_{\grupo}$, the semi-group $\grupo'$ from the statement can be chosen to be a Schottky semi-group that contains a high enough power of $g$.
\end{obs}

\subsection{Thermodynamics}\label{thermopre}

We begin by recalling some facts on Thermodynamical formalism over hyperbolic systems developed by Bowen, Ruelle, Parry, Pollicott among others, see for example  \cite{bowenruelle} and Parry-Pollicott \cite{parrypollicott}.

We work with \emph{metric-Anosov flows}, also called Smale flows by Pollicott \cite{smaleflows} who transferred to this more general setting the classical theory carried out for hyperbolic systems. We will only state some needed results and refer the reader to, for example, Pollicott \cite{smaleflows} or  Bridgeman-Canary-Labourie-S. \cite{pressure} for the definition.

Let $X$ be a compact metric space equipped with a continuous flow $\phi:X\to X$. The space of $\phi$-invariant probability measures on $X$ is denoted by $\medidas^\phi.$ The \emph{metric entropy} of $m\in\medidas^\phi$ will be denoted by $h(\phi,m).$ Via the variational principle, we will define the  \emph{pressure} of a function $f:X\to\R$ as \begin{equation}\label{pressuredefi}P(f)= \sup_{m\in\medidas^{\phi}} h(\phi,m)+\int_X fdm.\end{equation}  A probability measure $m$ realizing the $\sup$ is called an \emph{equilibrium state} of $f.$

Two continuous maps $f,g:X\to V$ are \emph{Liv\v sic-cohomolo\-gous} if there exists a $U:X\to V,$ of class $\mathrm C^1$ in the direction of the flow\footnote{i.e. such that if for every $x\in X,$ the map $t\mapsto U(\phi_tx)$ is of class $\mathrm C^1,$ and the map $x\mapsto \left.\frac{\partial }{\partial t}\right|_{t=0}U(\phi_tx)$ is continuous}, such that for all $x\in X$ one has $$f(x)-g(x)=\left.\frac{\partial}{\partial t}\right|_{t=0} U(\phi_tx).$$ The $f$-\emph{period} of a periodic orbit $\tau$ with period $p_\phi(\tau)$ is $$\ell_\tau(f)=\int_\tau f=\int_0^{p(\tau)}f(\phi_sx)ds\in V.$$

Let $f:X\to\R_{>0}$ be continuous. For every $x\in X$ the function $\kappa_f:X\times\R\to \R,$ defined by $\kappa(x,t)=\int_0^tf(\phi_sx)ds,$ is an increasing homeomorphism of $\R.$ There is thus a continuous function $\alpha:X\times\R\to\R$ such that for all $(x,t)\in X\times\R,$

$$\alpha\big(x,\kappa(x,t)\big)=\kappa\big(x,\alpha(x,t)\big)=t.$$

\noindent
The \emph{reparametrization} of $\phi$ by $f:X\to\R_{>0}$ is the flow $\phi^f=(\phi^f_t:X \to X)_{t\in\R}$ defined, for all $(x,t)\in X\times\R$ by $$\phi^f_t(x)=\phi_{\alpha(x,t)}(x).$$ 

The \emph{Abramov transform} of $m\in\medidas^\phi$ is the measure $m^\#\in\medidas^{\phi^f}$ defined by \begin{equation}\label{abramovdef}m^\#=\frac{f\cdot m}{\int fdm}.\end{equation}

We assume from now on that $\phi$ is a H\"older-continuous metric-Anosov. In this situation, Liv\v sic proved that the Liv\v sic-cohomology class of a H\"older-continuous function is uniquely determined by its periods. Moreover, if $f$ is real-valued and H\"older-continuous, it has a unique equilibrium state, denoted by $\gibbs_f$. We also let $\mmax_\phi$ be the unique probability measure maximizing entropy of $\phi$.

Recall that the co-variance is defined, for H\"older-continuous $g,h:X\to\R$ with 0$=\int gd\gibbs_f=\int hd\gibbs_f$, by $$\covar_f(g,h):=\lim_{t\to\infty}\int_X\frac1t\left(\int_0^tg(\phi_sx)ds\right)\left(\int_0^th(\phi_sx)ds\right)d\gibbs_f,$$ and the variance by $\Var_f(g)=\covar_f(g,g)\geq0$.

\begin{thm}[{Parry-Pollicott \cite[Prop. 4.10,4.11]{parrypollicott}}]\label{pder} Let $f,g:X\to\R$ be H\"older continuous. Then one has $(\partial/\partial t)|_{t=0} P(f+tg)=\int gd\gibbs_f.$ If $\int g d\gibbs_f=0$ then $$\left.\frac{\partial^2 P(f+tg)}{\partial t^2}\right|_{t=0}=\Var_f(g),$$ and if $\Var_f(g)=0$ then $g$ is Liv\v sic-cohomologous to zero.
\end{thm}

The space of H\"older-continuous functions with exponent $\alpha$ is naturally a Banach space, and by the previously mentioned result by Liv\v sic,  the space of Liv\v sic-cohomologically trivial functions is a closed subspace. The quotient $\Holder^\alpha(X)/\sim$ is thus a Banach space. Since $P$ is invariant under Liv\v sic-cohomology, Proposition \ref{pder} equips the space of (classes of) pressure zero functions for a fixed exponent $$\mathcal P^\alpha(X)=\big\{f\in\Holder^\alpha(X)/\sim:P(-f)=0\big\}$$ with a natural Riemannian metric: if $g\in\sf T_f\cal P(X)=\{g:\int gd\gibbs_f=0\}$ then we let $$\PP_f(g):=\frac{\Var_f(g)}{\int fd\gibbs_f}.$$

We now turn to concepts from Bridgeman-Canary-Labourie-S. \cite{pressure}. Similar ideas also appeared in Bonahon \cite{B-intersection}, Bridgeman \cite{martincriticos}, Burger \cite{burger}, Croke-Fathi \cite{CrokeFathi}, Fathi-Flaminio \cite{fathi-flaminio},  Katok-Knieper-Weiss \cite{Jcritico}, Knieper \cite{kni95} and McMullen \cite{McMWP}.

Let $f:X\to\R$ be H\"older-continuous and positive, consider, for $t>0,$ the finite set $\sf R_t(f)=\big\{\tau\textrm{ periodic}:\ell_f(\tau)\leq t\big\}$ and define the \emph{entropy of $f$} by $$\hh^f=\lim_{t\to\infty}\frac1t\log\# \sf R_t(f).$$ It is the topological entropy of $\phi^f$ and $\mmax_{\phi^f}=\gibbs_{-\hh^ff}^\#$ (see S. \cite{quantitative}). If $g:X\to\R$ is also H\"older continuous then in \cite{pressure} they define its \emph{dynamical intersection} with $f$ by $$\II_f(g)=\II(f,g)=\lim_{t\to\infty}\frac1{\#\sf R_t(f)}\sum_{\tau\in \sf R_t(f)}\frac{\ell_g(\tau)}{\ell_f(\tau)}.$$


The functions $\hh$ and $\II$ are well defined and vary analytically on $\Holder^\alpha(X,\R_+)$ and $\Holder^\alpha(X,\R_+)\times\Holder^\alpha(X,\R)$ respectively. If $g$ is also positive then we define its \emph{normalized dynamical intersection} with $f$ by $$\JJ_f(g)=\JJ(f,g)=\frac{\hh^g}{\hh^f}\II_f(g).$$ We have the following two rigidity results: one global and one infinitesimal.

\begin{thm}[\cite{pressure}]\label{JJ} One has $\JJ(f,g)\geq1$ and equality holds only if for every periodic orbit $\tau$ one has $\hh^f\ell_f(\tau)=\hh^g\ell_g(\tau).$ Let $(f_t)_{t\in(-\eps,\eps)}\in\Holder^\alpha(X,\R_+)$ be a $\class^2$ curve with $f_0=f,$ then $\hess_f\JJ_f(\vec f,\vec f)=\PP_f(\vec f),$ in particular $\hess_f\JJ(\vec f,\vec f)=0$ if and only if for $\forall\tau$ periodic one has \begin{equation}\label{P=0}\frac{\partial}{\partial t}\Big|_{t=0}\hh^{f_t}\ell_{f_t}(\tau)=0,\end{equation} or equivalently, $\partial^{\log}\hh$ and $\partial^{\log} f$ are Liv\v sic-cohomologous w.r.t. the flow $\phi^f$.
\end{thm}

In the above we have denoted by $\partial^{\log{}}$ the logarithmic derivative at $0$ $$\partial^{\log{}} g=\frac{(\partial/\partial t)|_{t=0} g_t}{g_0}.$$ We record the following consequence of $\JJ_f(\cdot)$ being critical at $f,$ giving the following formula for the variation of entropy, to be compared with Katok-Knieper-Weiss \cite{Jcritico}.

\begin{cor}[\cite{pressure}]\label{derivadaentropia}Let $(f_t)_{t\in(-\eps,\eps)}$ be a $\class^1$ curve of H\"older-continuous positive functions with $f_0=f$ and denote by $\hh_t=\hh^{f_t},$ then $$\partial^{\log{}} \hh=-\int\partial^{\log{}} fd\mmax_{\phi^{f}}=-\frac{\int \vec fd\gibbs_{-\hh f}}{\int fd\gibbs_{-\hh f}}.$$
\end{cor}

\begin{assu}\label{assuA}
Let now $F:X\to V$ be H\"older-continuous and assume the vector space spanned by the periods of $F$ is $V$. Assume moreover that $F$ and $1$ are Liv\v sic-cohomologically independent. 
\end{assu}

The compact convex subset of $V$ $$ \medidas^\phi(F)=\Big\{\int_XFd\mu:\mu\in\medidas^\phi\Big\}$$ has hence non-empty interior. On the other hand, for each $\varphi\in V^*$ one can consider the pressure of the function $\varphi(F):X\to\R$: $$\pp(\varphi)=P(-\varphi\circ F).$$ Proposition \ref{pder} implies that $\pp:V^*\to\R$ is analytic and strictly convex. Moreover, using the natural identification $(V^*)^*=V,$ one has, for $\varphi\in V^*$ that \begin{equation}\label{dpres}d_\varphi\pp=\int Fd\gibbs_{-\varphi(F)}.\end{equation}  One has the following:

\begin{prop}[{Babillot-Ledrappier \cite[Prop. 1.1]{babled}}]\label{babled1} Under Assumption \ref{assuA}, the map $\dual:V^*\to V$ defined by $\varphi\mapsto d_\varphi\pp$ is a diffeomorphism between $V^*$ and the interior of $\medidas^\phi(F).$ \end{prop}

Observe that our Assumption \ref{assuA} is slightly weaker than that of Babillot-Ledrappier \cite{babled}, however this does not affect the proof of \cite[Prop. 1.1]{babled}.

\begin{obs}\label{maxinterior} By Proposition \ref{babled1} one has ${\displaystyle\mass F{}:=\dual(0)=\int Fd\mmax_\phi\in\inte\medidas^\phi(F)}$.\end{obs}

\subsection{Anosov representations}\label{anosovpre}

Anosov representations where introduced by La\-bou\-rie \cite{labourie} for fundamental groups of closed negatively curved manifolds and extended to arbitrary (finitely generated) word-hyperbolic groups by Guichard-Wie\-nhard \cite{olivieranna}. They have, since then, been object of numerous works. We will present here a very summarized situation based on \cite{labourie}, \cite{olivieranna}, Gu\'eritaud-Guichard-Kassel-Wienhard \cite{GGKW}, Kapovich-Leeb-Porti \cite{KLP-Morse} and Bochi-Potrie-S. \cite{BPS}.

Let $\gh$ be a finitely generated group and denote, for $\g\in\gh,$ by $|\g|$ the word length w.r.t. a fixed finite symmetric generating set of $\gh.$

\begin{defi}\label{defAnosov}Let $\t\subset\simple$ be non-empty then, a representation $\rho:\gh\to\sf G$ is $\t$-\emph{Anosov} if there exist $c,\mu$ positive such that for all $\g\in\gh$ and $\sroot\in\t$ one has \begin{equation}\label{defA}\sroot\Big(\cartan\big(\rho(\g)\big)\Big)\geq \mu|\g|-c.\end{equation} We will denote by $\Anosov_\t(\gh, \sf G)\subset\caracteres$ the space of $\t$-Anosov characters. If $\sf G=\SL(d,\K)$ for $\K=\R$ or $\C$, a $\{\slroot_1\}$-Anosov representation is called \emph{projective Anosov}.
\end{defi}

If follows readily that for every $\sroot\in\t$ the representation $\rep_\sroot\circ\rho:\gh\to\GL(V_\sroot)$ is projective-Anosov. The theorem below can also be found in Bochi-Potrie-S. \cite{BPS}.

\begin{thm}[{Kapovich-Leeb-Porti \cite{KLP-Morse}}]\label{A-A} If $\rho:\gh\to\sf G$ is $\t$-Anosov then $\gh$ is word-hyperbolic. \end{thm}

The group $\gh$ has thus a \emph{Gromov-boundary} $\bord\gh$ and a \emph{space of geodesics} $$\bord^2\gh=\{(x,y)\in(\bord\gh)^2:x\neq y\}.$$ The following proposition can be found in {Bochi-Potrie-S. \cite[Lemma 4.9]{BPS}}, Kapovich-Leeb-Porti \cite{KLP-Morse} and Gu\'eritaud-Guichard-Kassel-Wienhard \cite{GGKW} and relates Definition \ref{defAnosov} to Labourie's original definition.

\begin{prop}\label{conical}
If $\rho:\gh\to\sf G$ is $\t$-Anosov then there exist $\rho$-equivariant H\"older-continuous maps $$\xi^\t:\bord\gh\to\EuScript F_\t(\sf G)\ \textrm{ and }\ \xi^{\ii\t}:\bord\gh\to\EuScript F_{\ii\t}(\sf G)$$ such that if $x,y\in\bord\gh$ are distinct, then $(\xi^\t(x),\xi^{\ii\t}(y))\in\posgen_\t.$ Moreover, if $\g\in\gh$ is hyperbolic, then $\rho(\g)$ is $\t$-proximal with attracting point $\xi^\t(\g^+)=\rho(\g)_\t^+.$

\end{prop} 

By Labourie \cite{labourie} and Guichard-Wienhard \cite{olivieranna} $\Anosov_\t(\gh, \sf G)$ is an open subset of the character variety. We also let $\Anosov_\t^{*\mathrm{Z}}(\gh, \sf G)$ be the space of regular points $\rho$ such that $\rho(\gh)$ is Zariski-dense in $\sf G$. Since we are assuming that $\sf G$ is connected, one has the following:

\begin{prop}[{Bridgeman-Canary-Labourie-S. \cite[Proposition 7.3]{pressure}}] Assume $\t$ contains at least one simple root of each factor of $\sf G$. Then the space $\Anosov_\t^{*,\mathrm{Z}}(\gh,\sf G)$ is an analytic manifold.
\end{prop}

For surface groups one has the following description of the regular points of characters, that can be found on  Labourie's book \cite[\S\,5]{LabourieLectures}, from which we borrow the terminology of \emph{very regular points}: $$\hom^{\mathrm{vr}}(\pi_1S,\sf G)=\big\{\rho:\sf Z_{\sf G}\big(\rho(\pi_1S)\big)=\sf Z(\sf G)\big\}.$$ Since $\sf G$ is connected, morphisms with Zariski-dense image are very regular.

\begin{thm}\label{regularsurface} Both $\hom^{\mathrm{vr}}(\pi_1S,\sf G)$ and $\hom^{\mathrm{vr}}(\pi_1S,\sf G)/\sf G$ are analytic manifolds.
\end{thm}

\subsubsection{A reference flow for a group admitting an Anosov representation}\label{referenceFlow}
 
Assume that $\gh$ admits a $\t$-Anosov representation $\rho_0$, onto some $\sf G$ and some non-empty $\t$. Fix $\sroot\in\t$. Then, $\rep_\sroot\circ\rho_0:\gh\to\GL(V_\sroot)$ is projective Anosov and thus there exist $$\xi^1  :\bord\gh\to\P(V_\sroot)\ \textrm{ and }\ \xi^{d-1} :\bord\gh\to\P\big(V_\sroot^*\big)$$ such that for every $(x,y)\in\bord^2\gh$ one has $\ker\xi^{d-1}(x)\oplus\xi^1(y)=V_\sroot.$ We use the equivariant maps to construct a bundle $\R\to\widetilde{\sf F}\to \bord^2\gh$ whose fiber at $(x,y)\in\bord^2\gh$ is $$\widetilde{\sf F}_{(x,y)}=\big\{(\varphi,v)\in\xi^{d-1}(x)\times\xi^1(y):\varphi(v)=1\big\}/\sim,$$ where $(\varphi,v)\sim(-\varphi,-v).$ This bundle is equipped with a $\gh$-action $\g(\varphi,v)=\big(\varphi\circ\rho(\g)^{-1},\rho(\g) v\big)$ and an $\R$-action  $\widetilde{\phi}_t\cdot(\varphi,v)=(e^t\varphi,e^{-t}v).$ Let $$\sf U\gh=\gh\/\widetilde{\sf F}$$ and denote by $\phi=\big(\phi_t:\sf F\to\sf F\big)_{t\in\R}$ the induced flow on the quotient (it is usually called \emph{the geodesic flow of $\rep_\sroot\circ\rho_0$}).

\begin{thm}[Bridgeman-Canary-Labourie-S. \cite{pressure}]\label{tutti}The above $\gh$-action is properly discontinuous, co-compact, and $\phi$ is H\"older-continuous and metric-Anosov.
\end{thm}

The flow-space $\sf U\gh$ and flow $\phi$ will be fixed from now on and used as a reference flow. The set of hyperbolic elements of $\gh$ will be denoted by $\gh_{\mathrm{h}}$ and for $\g\in\gh_{\mathrm{h}}$ we will denote by $\ell(\g)$ the period for $\phi$ of the periodic orbit $[\g]$ associated to $\g$.

\subsubsection{The Ledrappier potential}\label{thermoSS}

We now recall a combination of facts from Bridgeman-Canary-Labourie-S. \cite{pressure}, Potrie-S. \cite{exponentecritico} and S. \cite{quantitative,dichotomy}. Recall from the previous subsection that a base flow is fixed for $\gh$.

\begin{thm}\label{teoLedr}Let $\rho:\gh\to\sf G$ be $\t$-Anosov. Then there exists a H\"older-continuous map $\ledrappier_\rho:\sf U\gh\to\a_\t,$ called the  \emph{Ledrappier potential of $\rho$}, such that for every $\g\in\gh_{\mathrm{h}}$ $$\ell_{[\g]}(\ledrappier_\rho)=\jordan_\t^\g(\rho).$$ Moreover, if $\{\rho_u:\G\to\sf G\}_{u\in D}$ is an analytic family of $\t$-Anosov representations, then the map $u\mapsto\ledrappier_{\rho_u}$ is analytic.

\end{thm}
Let us consider the $\t$-limit cone of $\rho$ defined by \begin{equation}\label{conti}\Bcone_{\t,\rho}=\overline{\{\R_+\cdot\jordan_\t(\rho(\g)):\g\in\gh\}}=\R_+\cdot\medidas^\phi(\ledrappier_\rho)\subset\a_\t.\end{equation} Elements of $\inte(\Bcone_{\t,\rho})^*$ will be called \emph{length functionals}. Indeed (see S. \cite[Lemma 3.4.2]{dichotomy}),  $\length\in\inte(\Bcone_{\t,\rho})^*$ if and only if $\length(\ledrappier_\rho)$ is Liv\v sic-cohomologous to a strictly positive function, or equivalently there exist $c$ positive such that for every $\g\in\gh$ $$\length^\g(\rho)\geq c\ell(\g).$$
We can thus define the \emph{entropy} of $\length$ as $\entropy\psi\rho=\hh^{\psi(\ledrappier\rho)}$ and the \emph{critical hypersurface} 
\begin{alignat}{2}\label{qh1}
\EuScript Q_{\t,\rho}  & =\Big\{\length\in\inte(\Bcone_{\t,\rho})^*:\entropy\length\rho=1\Big\}. 
\end{alignat}

It follows that  $\EuScript Q_{\t,\rho}$ is a closed co-dimension-one analytic sub-manifold that  bounds a convex set, which is strictly convex if $\rho(\gh)$ is Zariski-dense, see S. \cite[\S\,5.9 and 5.10]{dichotomy} where details and more information can be found.

The Ledrappier potential embeds the space of $\t$-Anosov representations in the space of H\"older-continuous potentials $\Holder(\sf U\gh,\a_\t)$ \begin{alignat*}{2}\ledrappier:\Anosov_\t(\gh,\sf G)&\to\Holder(\UG,\a_\t)\\ \rho&\mapsto\ledrappier_\rho,
\end{alignat*} and by Bridgeman-Canary-Labourie-S. \cite{pressure} it is a real analytic map. Its differential at $v$ is (a Liv\v sic-cohomlogy class of) a H\"older-continuous map $$\vec{\ledrappier}_{v}:\UG\to\a_\t$$ whose periods are, by definition, ${\displaystyle \int_{[\g]} \vec{\ledrappier}_{v}=\mathrm d\jordan_\t^\g(v)}.$

By continuity of $\rho\mapsto\Bcone_{\t,\rho}$ (Eq. \eqref{conti}), every $\length\in(\a_\t)^*$ defines an open subset \begin{equation}\label{Upsi}\EuScript U_\length=\{\eta\in\Anosov_\t(\gh, \sf G):\length\in\conodual{\t,\eta}\}\end{equation} and gives in turn a map $\EuScript U_\length \to\Holder(\sf U\gh,\R),$ $\rho\mapsto\length(\ledrappier_\rho).$

\begin{defi}The $\length$-\emph{pressure form} $\PP^\length$ on $\EuScript U_\length$ is the pullback of the pressure metric on $\cal P(\sf U\gh)$ by the map $\rho\mapsto\entropy\length\rho\cdot\length(\ledrappier_\rho)$.\end{defi}

For $\length\in\EuScript Q_{\t,\rho}$ and $\eta\in\EuScript U_\length$, we define the $\length$-dynamical intersection of $(\rho,\eta)$ as \begin{equation}\label{parai}\II^\length(\rho,\eta)=\II\big(\length(\ledrappier_\rho),\length(\ledrappier_\eta)\big)=\lim_{t\to\infty}\frac1{\#\sf R_t^\length(\rho)}\sum_{\g\in \sf R_t^\length(\rho)}\frac{\length\big(\jordan(\eta(\g))\big)}{\length\big(\jordan(\rho(\g))\big)},\end{equation} where $\sf R_t^\length(\rho)=\big\{\g\in\gh_{\mathrm{h}}:\length^\g(\rho)\big)\leq t\}.$ By Theorem \ref{JJ}, upon denoting by $\JJ^\length$ the associated normalized intersection, one has $\hess_\rho\JJ^\length_\rho=\PP^\length_\rho.$

The opposition involution $\ii$ of $\a$ is induced by an external automorphism $\undi :\sf G\to\sf G$, which acts whence on characters $\undi:\caracteres\to\caracteres$. Eq. \eqref{parai} gives:


\begin{cor}\label{tauisom}For every $\length\in\EuScript Q_{\t,\rho}$ and $\eta\in\EuScript U_\length$ one has $\II^\length(\rho,\eta)=\II^{\length}(\undi \rho,\undi\eta)$. In particular, the involution $\undi$ on $\caracteres$ is an isometry of any pressure form $\PP^\length$.
\end{cor}

\subsubsection{Entropy regulating form}

Consider the set of \emph{$\length$-normalized variations} of $v$ $$ \VV\length_{\t,v}=\overline{\Big\{\frac{\varjort\g v}{\length\big(\jordan^\g(\rho)\big)}:\g\in\gh_{\mathrm h}\Big\}}.$$

\begin{obs}\label{varphi=cte}By Eq. \eqref{P=0}, the Pressure form $\PP^\length_\rho$ degenerates at $v$ if and only if $\length(\VV\length_{\t,v})=\partial^{\log{}}\entropy\length{}$, in particular $\VV\length_{\t,v}$ must be contained on a level set of $\length$.
\end{obs}

\begin{lemma}\label{cc}If $\rho$ is $\t$-Anosov then $\VV\length_{\t,v}$ is compact and convex. Moreover the map $v\mapsto\VV\length_{\t,v}$ is continuous on the open set $\sf T\EuScript U_\length$.\end{lemma}

\begin{proof}Recall we have a base-flow $\phi=(\phi_t:\sf U\gh\to\sf U\gh)_{t\in\R}$. Consider the Ledrappier potential $\ledrappier:\sf U\gh\to\a_\t$ together with its variation $\vec\ledrappier:\sf U\gh\to\a_\t$ associated to $v$. By definition, for every $\g\in\gh$ $$\int_\g\ledrappier=\jordan^\g(\rho)\textrm{ and }\int_\g\vec\ledrappier=\mathrm d\jordan_\t^\g(v).$$ If we let $\phi^{\length(\ledrappier)}$ be the reparametrization of $\phi$ by $\length(\ledrappier)$ (recall \S\,\ref{thermopre}) and we let $\mu\mapsto\mu^\#$ be the Abramov transform on invariant measures, then $$\frac{\mathrm d\jordan_\t^\g(v)}{\length^\g(\rho)}=\frac1{\ell_{\phi^{\length(\ledrappier)}}(\g)}\int_{\g^\#}\frac{\vec\ledrappier}{\length(\ledrappier)}$$ so $\VV{\length}_{\t,v}=\medidas^{\phi^{\length(\ledrappier)}}\big(\vec\ledrappier/\length(\ledrappier)\big),$ which yields the desired conclusion.\end{proof}

It is convenient to name a particular point of $\VV\length_{\t,v}$, to do so we introduce the \emph{length cone-bundle} that fibers over Anosov representations:
\begin{alignat*}{2}\lenf_\t(\gh,\sf G) & =\Big\{(\rho,\length)\in\Anosov_\t(\gh,\sf G)\times(\a_\t)^*: \length\in\conodual\rho\Big\}.\end{alignat*}

\begin{cor}[Entropy regulating form]\label{masa}\label{minte} There exists an analytic fibered map $$\mass{}{}:\lenf_\t(\gh,\sf G)\to\sf T^*\Anosov_\t(\gh,\sf G)\otimes\a_\t$$  such that if $(\rho,\length)\in\lenf_\t(\gh,\sf G)$ and $v\in\sf T_\rho\Anosov(\gh,\sf G)$ then $\mass{\length} (v)\in\relinte\VV\length_{\t,v}$ and \begin{alignat}{2}\label{masaentropia}\length\big(\mass{\length} {(v)}\big)&=-\mathrm d\log\entropy\length{}(v).\end{alignat} Moreover, if $\vec\ledrappier$ and $\length(\ledrappier)$ are Liv\v sic-co\-ho\-mo\-lo\-gi\-cally independent, then $\VV\length_{\t,v}$ has non-empty interior and thus $\mass{\length} (v)\in\inte\VV\length_{\t,v}$.
\end{cor}

\begin{proof} 
For $(\rho,\length)\in\lenf_\t(\gh,\sf G)$ and $v\in\sf T_\rho\Anosov_\t(\gh,\sf G)$ we let $\mass\length{v}=\int \frac{\vec\ledrappier}{\length(\ledrappier)}d\mmax_{\phi^\length}.$ The corollary follows from Corollary \ref{derivadaentropia}, Proposition \ref{babled1} and Remark \ref{maxinterior}.\end{proof}

\subsubsection{A needed Lemma} The following is only needed in Corollary \ref{levelsets}.

\begin{lemma}\label{masainte}Assume $\sf G$ is simple and has rank $1$. Let $\rho\in\Anosov_{\simple}(\gh,\sf G)$ have Zariski-dense image and consider a non-zero $v\in\sf T_\rho\caracteres$. Then, the set $$\VV{}_v=\overline{\left\{\frac {\mathrm{d}\lambda^\g(v)}{\lambda^\g(\rho)}:\g\in\G\right\}}\subset\R$$ is a compact interval, has non-empty interior and $-\mathrm{d}\log \entropy{}{}(v)\in\inte\VV{}_v.$\end{lemma}

\begin{proof} Since $\sf G$ has rank $1$ the Ledrappier potential $\ledrappier$ and $\vec\ledrappier$ are real valued. In order to apply Corollary \ref{minte} we need to show that they are Liv\v sic-co\-ho\-mo\-lo\-gi\-cally independent, a fact contained in Theorem \ref{indeplambda1} below.\end{proof}

We conclude the section by stating the following result from Bridgeman-Canary-Labourie-S. \cite{pressure}. If  $\rho:\gh\to\sf G\subset\SL(d,\R)$ then we say that $\rho$ is $\sf G$-\emph{generic} if $\rho(\gh)$ contains an element that is loxodromic in $\sf G$.

\begin{thm}[{\cite[Lemma 9.8+Prop. 10.1]{pressure}}]\label{indeplambda1} Let $\rho:\gh\to\sf G\subset \SL(d,\R)$ be projective Anosov and irreducible. Let $(\rho_t)_{t\in(-\eps,\eps)}$ be a smooth curve through $\rho$ of $\sf G$-generic representations. If there exists $K\in\R$ such that for all $\g\in\gh$ one has $$\deriva t0\jordan_1(\rho_t\g)=K\jordan_1(\rho\g)$$ then $K=0$ and the class $\co\in H^1_{\Ad_{\sf G}\rho}(\gh,\sf G)$ associated to $\vec\rho$ is trivial.
\end{thm}

\subsubsection{Cross ratios}\label{crog}

A decomposition $\R^d=\ell\oplus V$ into a line and a hyperplane defines a rank-1 projection denoted by $\pi_{(\ell,V)}$. If $\R^d=r\oplus W$ is such that $\dim r=1$ and moreover $r\cap V=\{0\}=\ell\cap W$ (we say in this case that the decompositions are \emph{transverse}) then we define the (multiplicative) \emph{cross ratio} by \begin{equation}\label{B1}\crossm_1(\ell,V,r,W)=\traza(\pi_{\ell,V}\pi_{r,W}).\end{equation}

Let now $\eta:\gh\to\PGL(d,\R)$ be a projective Anosov representation with equivariant maps $\xi^1:\bord\gh\to\P(\R^d)$ and $\xi^{d-1}:\bord\gh\to\P((\R^d)^*).$ A pair $(x,y)\in\bord^2\gh$ defines a decomposition $\R^d=\xi^1(y)\oplus\xi^{d-1}(x)$. If we let $\bord^{(4)}\gh $ the space of pairwise distinct four-tuples then an element $(x,y,z,t)\in\bord^{(4)}\gh$ defines two transverse decompositions so we define the \emph{cross ratio map} $\mathrm b_\eta:\bord^{(4)}\gh\to\R$ by $$\mathrm b_{\eta}(x,y,z,t)=\crossm_1\Big(\xi^1(y),\xi^{d-1}(x),\xi^1(t),\xi^{d-1}(z)\Big).$$

If $\rho\in\Anosov_\t(\gh,\sf G)$ then we define the $\a_\t$-valued cross ratio map $\crossG_\rho:\bord^{(4)}\gh\to\a_\t$ by, for every $\sroot\in\t$, $\peso_\sroot(\crossG_\rho(x,y,z,t))=\log|\mathrm b_{\rep_\sroot}(x,y,z,t)|.$

\begin{thm}[\cite{pressure}]\label{crosan} If $\{\rho_u\}_{u\in D}\subset\Anosov_\t(\gh,\sf G)$ is an analytic family then for every four-tuple $(x,y,z,t)\in\bord^{(4)}\gh$ the map $u\mapsto\crossG_{\rho_u}(x,y,,z,t)$ is real-analytic.
\end{thm}


\addtocontents{toc}{\protect\setcounter{tocdepth}{1}}

\part{Affine actions}\label{affineactions}

\section{Margulis invariant: basics}

\subsection{An elementary lemma}\label{s.unmar}

Let $V$ be a finite dimensional real vector space and consider the \emph{affine group} $\Aff(V)=\GL(V)\ltimes  V.$ An element $f\in\Aff(V)$ has  a \emph{linear part} $\lin f\in\GL(V)$ and a \emph{translation part} $\tra f$ so that $\forall u\in V$ one has $$f(u)=\lin f(u)+\tra f.$$ 

Let us consider the (possibly trivial) generalized eigenspace of $\lin f$ associated to the eigenvalue $1$ $$\cero=\cero(\lin f)=\big\{w\in V:\exists n\geq0\textrm{ with }(\lin f-\id)^nw=0\big\}$$ and define the \emph{un-normalized Margulis invariant of $f$,} $\umargulis(f)\in \cero,$ as follows. Jordan's decomposition of $\lin f$ guarantees the existence of a $\lin f$-invariant decomposition $V=W\oplus \cero$, let $\pi^1:V\to \cero$ be the associated projection and define \begin{equation}\label{unmargu}\umargulis(f)=\pi^1(\tra f).\end{equation}

\begin{obs}\label{inversa} It follows at once that $\umargulis(f^{-1})=-\umargulis(f).$
\end{obs}

\begin{lemma}\label{elem}Consider $f\in\Aff(V),$ then there exists $o\in V$ such that the translate $\espacio(f)=\cero(\lin f)+o$ is invariant by $f.$ Moreover, the transformation $\cero\to \cero$ defined by $v\mapsto f(v+o)-o$ has linear part $\lin f|\cero$ and  translation part $\umargulis(f).$\end{lemma}

\begin{proof} Indeed, $1$ is not an eigenvalue of $\lin f|W$ and thus $\lin f|W-\id_W$ is invertible. Let $\pi:V\to W$ the projection following the decomposition $V=W\oplus \cero.$ The transformation $\pi f=\pi\circ(f|W):W\to W$ $$\pi f(w)=\lin fw+\pi(\tra f)$$ has a unique fixed point $\o=\o_f,$ defined by $\o=-(\lin f -\id)^{-1}\pi(\tra f).$ For $v\in \cero$ one has \begin{alignat*}{2}f(\o+v) & =\lin f(\o+v)+v_f\\ & =\lin f(\o)+\lin f(v)+\pi v_f+\umargulis(f)\\ & =\lin f(\o)+\pi v_f+\lin f(v)+\umargulis(f) =\o+\lin f(v)+\umargulis(f)\in \o+\cero,\end{alignat*} concluding the proof.\end{proof}

We end this subsection with the following remarks:

\begin{obs}\label{Invariancia}\item\begin{itemize}\item[-]For every $h\in\GL(V)$ one has $h\big(\umargulis(f)\big)=\umargulis(hfh^{-1}).$\item[-] Assume that $\lin f|\cero=\id,$ then for all $u\in V$ the maps $f$ and $f+u-\lin f u$ have the same un-normalized Margulis invariant, that is  $$\umargulis(f)=\umargulis(f+u-\lin fu).$$ Indeed $\pi^0\big(u-\lin f(u)\big)=0.$ \end{itemize}
\end{obs}

\subsection{Around the $0$-restricted-weight of an irreducible representation}

We consider now a reductive real-algebraic group $\sf G$ and an irreducible representation $\rep:\sf G\to\SL(V) $ with $0\in\poif$ and define \begin{alignat*}{2}\poif^{\longest} & =\big\{\chi\in\poif:\chi\circ\longest=\chi\big\},\\ \Vlongest & =\bigoplus_{\chi\in\poif^\longest}V^\chi.\end{alignat*} The vector space $\Vlongest$ will be called the \emph{ideally neutral space} of $\rep$. For $v\in\a$ we let \begin{alignat*}{2}\poif^{v,+} &=\{\chi\in\poif:\chi(v)>0\},\nonumber\\\poif^{v,-} & =\{\chi\in\poif:\chi(v)<0\}.\end{alignat*}

An element of $\Fix(\ii)$ necessarily annihilates all $\chi\in\poif^{\longest}$. We fix an $X_0\in\a^+\cap\Fix(\ii)$ which does not belong to the finite union of kernels $\ker\chi,$ $\chi\in\poif\setminus\poif^{\longest}$. This choice provides a partition $\poif=\poif^{\longest}\cup\poif^+\cup\poif^-$, where \begin{alignat*}{2}\poif^+ &:=\poif^{X_0,+},\\\poif^- & :=\poif^{X_0,-},\end{alignat*} with the advantage that $\longest(\poif^+)=\poif^-$ and $\{\chi\in\poif:\chi(X_0)=0\}=\poif^\longest$. 

We now consider the set of simple roots $\vt$ such $$\simple-\vt=\{\sroot\in\simple:\re_\sroot(\poif^+)=\poif^+\}.$$

\noindent
Then one has the following (observe we are using the opposite convention for parabolic groups than Smilga \cite{SmilgaAnnalen}).

\begin{prop}[{Smilga \cite[Prop. 6.4, Lemma 6.5]{SmilgaAnnalen}}]\label{amas} If $\sroot\in\simple-\vt$ then $\re_\sroot$ necessarily fixes each $\chi\in\poif^\longest$. The set $\vt$ is $\ii$-invariant and the parabolic group $\sf P^{\vt}$ is the stabilizer in $\sf G$ of each of $$a^+ :=\bigoplus_{\chi\in\poif^+}V^\chi\textrm{ and }A^+  :=a^+\oplus\Vlongest.$$ The group $\check{\sf P^\vt}$ is the stabilizer of  $a^- :=\bigoplus_{\chi\in\poif^-}V^\chi\textrm{ and }A^-  :=a^-\oplus\Vlongest.$ \end{prop}

\begin{lemma}[{Smilga \cite[Lemme 7.12]{SmilgaAnnalen}}]\label{pv0}Consequently, for each $\chi\in\poif^\longest$ the algebra $\ge^0\bigoplus_{\slroot\in\<\simple-\vt\>}\ge^\slroot\oplus\ge^{-\slroot}$ acts trivially on $V^\chi$ and thus  $\sf L_\vt$ acts as $\rep(\sf M\sf A)$ on $\Vlongest$.
\end{lemma}

\begin{proof} Fix $\slroot\in\<\simple-\vt\>$, the reflection $\re_\slroot$ also stabilizes $\poif^+$. Since $\chi$ is $\longest$-invariant one readily sees that $\chi+\slroot$ is not. Now, $\re_\slroot(\chi)\in\poif^\longest\cup\poif^-$ (or else $\chi=\re_\slroot(\re_\slroot\chi)\in\poif^+$). So if $\chi+\slroot\in\poif$ then necessarily $\chi+\slroot(X_0)>0$, i.e. $\chi+\slroot\in\poif^+$. However, $$\re_\slroot(\chi+\slroot)(X_0)=(\re_\slroot\chi-\slroot)(X_0)<0,$$ i.e. $\re_\slroot(\chi+\slroot)\in\poif^-$ achieving a contradiction. Whence $\chi+\slroot\notin\poif$ and thus $ \fund(\ge_\slroot) V^\chi \subset V^{\chi-\slroot}=\{0\}$ concluding the proof.  \end{proof}

We now consider the action of $\rep(\sf M\sf A)$ on $\Vlongest$. The first remarkable subspace, the \emph{neutralizing space} of $\rep$, is the fixed point subspace of $\rep(\sf M\sf A)$ on $\Vlongest,$ and coincides with the fixed point set of $\rep(\sf M)$ on $V^0$. It is denoted by  $$\trivia:=V^{\rep(\sf M\sf A)}=V^{0,\rep(\sf M)}=\big\{v\in V^0:\rep(\sf M)v=v\big\}.$$

\begin{defi} The \emph{neutralizing dimension} of $\rep$ is $\neudim\rep:=\dim\trivia$.
\end{defi}
 
Let $\trivia^{\perp_\sf M}\subset \Vlongest$ be the $\rep(\sf M)$-invariant complement, the decomposition \begin{equation}\label{deconeu}\Vlongest=\trivia^{\perp_\sf M}\oplus\trivia
\end{equation} is canonical given $\rep.$ Let $\pi^{\trivia}:\Vlongest\to\trivia$ the associated projection.

\begin{ex}\item \begin{itemize}\item[-] If $\sf G$ is split then $\trivia=V^0,$ see for example Ghosh \cite[Lemma 2.5]{SouravIso}.
\item[-] For the defining representation of $\SO_{1,n}$ the neutralizing space is trivial.
\item[-] If we let $\rep=\Ad:\sf G\to\SL(\ge)$ then $\Vlongest=\ge^0=\m\oplus\a,$ and writing $\m=[\m,\m]\oplus\centro(\m)$ one has that $\sf M$ preserves each factor of $[\m,\m]\oplus\centro(\m)\oplus\a$ and acts trivially on $\centro(\m)\oplus\a$ so  $$\trivia=\centro(\m)\oplus\a.$$ The typical situation when $\centro(\m)\neq\{0\}$ is when $\sf G$ is a complex group considered as a real group, in which case $\centro(\m)=\m$ and $\centro(\m)\oplus\a$ is a Cartan subalgebra of $\ge$ (as a complex Lie algebra).
\end{itemize}\end{ex}

If we let $k=\dim a^+$ then, as $\rep(\longest)a^+=a^-$, one has $k=\dim a^-$, let also $n=k+\dim\Vlongest$, and denote by $\FF_{k,n}$ the space of partial flags $(u_k\subset u_n)$ where $u_i\in\grass_i(V)$. Proposition \ref{amas} and Eq. \eqref{maps} give an algebraic $\rep$-equivariant map \begin{alignat}{2}\label{equiv}  :\EuScript F_{\vt}(\sf G)& \to  \FF_{k,n}(V)\nonumber\\ x & \mapsto \big(x_\rep,X_\rep\big) \end{alignat} such that if $(x,y)\in\posgen_{\vt}(\sf G)$ then $\big(x_\rep,X_\rep\big)$ and $\big(y_\rep,Y_\rep\big)$ are in general position. 
\begin{cor}\label{anguloscasa} There exists $C$ only depending on $\rep$ such that if $(x,y)\in\posgen_\vt$ then $$\|(x|y)\|\frac1C\leq\big\|\big((x,X)|(y,Y)\big)\big\|_{\rep o}\leq C\|(x|y)\|.$$
\end{cor}

\begin{proof} Follows from Proposition \ref{amas} and Remark \ref{GromovyRep}.\end{proof}

A pair $(x,y)\in\posgen_\vt(\sf G)$ determines the $n-k$-dimensional space $$X_\rep\cap Y_\rep.$$  By Lemma \ref{pv0}, such a space carries a natural decomposition into \emph{its neutralizing space} and the "$\sf M$-invariant complement". Indeed, if $g_0,g_1\in\sf G$ are such that 
$g_0(x,y)=\big([\sf P^\vt],[\wk{\sf P}^\vt]\big)=g_1(x,y)$ then $g_0g_1^{-1}$ belongs to $\sf L_\vt$, which preserves the decomposition \eqref{deconeu}. Consequently,  both subspaces of $X_\rep\cap Y_\rep$ \begin{alignat*}{2}\trivia^{(x,y)} & = \rep(g_0)^{-1}\trivia \\ \trivia^{(x,y),\perp_\sf M} & = \rep(g_0)^{-1}(\trivia^{\perp_{\sf M}})\end{alignat*} are independent of $g_0$. For each $(x,y)\in\posgen_\vt$ we fix $\psi_{(x,y)}\in\sf G$ such that \begin{equation}\label{endereza}\psi_{(x,y)}\big([\sf P^\vt],[\wk{\sf P}^\vt]\big)=(x,y).\end{equation} By Bochi-Potrie-S. \cite[Prop. 8.12]{BPS} it can be chosen so that $\|\cartan(\psi_{(x,y)})\|$ is coarsely comparable to the norm of the Gromov product $(x|y)_o^\vt.$

\subsection{$(\rep,X_0)$-compatible elements}
If $\lin g\in\sf G$ is $\vt$-proximal we have, via \eqref{equiv}, \begin{equation}\label{desoggg}V=\big(\lin g^+_\rep\oplus \lin g^-_\rep\big)\oplus \lin G^0_\rep \end{equation} where $\lin G^0_\rep:=\lin G^+_\rep \cap \lin G^-_\rep.$ Observe that \begin{enumerate}\item $\rep(\lin g)|\lin G^0_\rep$ decomposes as sum of roto-hometheties,
\item\label{ii} the subspace $\lin g^+_\rep$ is not necessarily attracting for $\rep(\lin g)$ on $\grass_k(V)$. This only happens if $\jordan(\lin g)$ lies in a specific neighborhood of $X_0$: the open cone $$\{v\in\a^+:\chi^+(v)>\chi^0(v),\ \forall \chi^+\in\poif^+,  \chi^0\in\poif^\longest\}.$$ Symmetrizing in order to deal with inverses we consider the sub-cone of $\a^+$: 
\begin{equation}\label{encono}\encono_\fund=\big\{v\in\a^+:\poif^+(v)>\poif^\longest(v)>\poif^-(v)\big\}.
\end{equation}

\end{enumerate}

Because of item \eqref{ii} one introduces the following definition.

\begin{defi} An element $g\in\sf G\ltimes V$ is \emph{$(\rep,X_0)$-compatible} if $\jordan(\lin g)\in\encono_\fund$. 
\end{defi}

\begin{lemma}\label{proximaltheta} A $(\rep,X_0)$-compatible element $g$ has $\vt$-proximal linear part. Moreover the flag $(\lin g^+_\rep,\lin G^+_\rep)\in\FF_{k,n}(V)$ is attracting for $\rep(\lin g)$ with repelling flag $(\lin g^-_\rep,\lin G^-_\rep)$.\end{lemma}

\begin{proof}By definition, for every $\sroot\in\vt$ there exists $\chi\in\poif^+$ such that $\re_\sroot\chi\notin\poif^+$. Thus, if $g$ is $(\rep,X_0)$-compatible one has, for such $\chi$, \begin{alignat*}{2}\chi(\jordan(\lin g)) & >(\re_\sroot\chi)(\jordan(\lin g))\\&=(\chi-\<\chi,\sroot\>\sroot)(\jordan(\lin g))\\&=\chi(\jordan(\lin g))-\<\chi,\sroot\>\sroot(\jordan(\lin g)),\end{alignat*} giving $\sroot(\jordan(\lin g))>0$. The second statement follows by definition.\end{proof}

Observe that although the action $\rep(\lin g)$ on $\lin G^0_\rep=\lin G^+_\rep\cap\lin G^-_\rep$ might have some expanding/contracting spaces, this expansion is dominated by the one on $\lin g^+_\rep$ and the same holds for the contraction and $\lin g^-_\rep$. Moreover, recall notation from \S\,\ref{s.unmar}, $$\trivia^{(\lin g^+_\rep,\lin g^-_\rep)}\subset\cero(\lin g)\subset \lin G^0_\rep.$$

Let now $g\in\sf G\ltimes V$ be $(\rep,X_0)$-compatible. Equation \eqref{unmargu} provides an unnormalized Margulis invariant $$\umargulis(g)\in\cero(\lin g)\subset \lin G^0_\rep=\trivia^{(\lin g^+_\rep,\lin g^-_\rep)}\oplus\trivia^{(\lin g^+_\rep,\lin g^-_\rep),\perp_\sf M}$$ that we further project onto $\trivia^{(\lin g^+_\rep,\lin g^-_\rep)}$ and push to $\trivia$ via $\psi_{(\lin g^+,\lin g^-)}$ to obtain the Margulis invariant of $g$:

\begin{defi}\label{defimar} The \emph{Margulis invariant} of a $(\rep,X_0)$-compatible $g\in\sf G\ltimes V$ is $$\margulis(g):=\pi^\trivia\big(\psi_{(\lin g^+,\lin g^-)}\big(\umargulis(g)\big)\big)\in\trivia.$$
\end{defi}

\begin{obs}\label{inv2}\item\begin{itemize}\item[-] By both items on Remark \ref{Invariancia}, it is invariant under conjugation by $\sf G\ltimes V.$ \item[-]The Weyl group acts on $\trivia$ and Remark \ref{inversa} yelds $\margulis(g^{-1})=\longest\cdot\big(-\margulis(g)\big).$\end{itemize}\end{obs}

\subsection{Invariant flags} Assume $g\in\sf G\ltimes V$ is $(\rep,X_0)$-compatible and consider the projections parallel to the decomposition in Equation \eqref{desoggg}\begin{alignat*}{2}\pi^{\pm}_g & :V\to \lin g^+_\rep\oplus \lin g^-_\rep,\\ \pi_g^0 & :V\to \lin G^0_\rep.
\end{alignat*} 

\begin{defi}\label{origen}The unique fixed point of the affine map $\pi^{\pm}_g\circ g:\lin g^+_\rep\oplus \lin g^-_\rep\to\lin g^+_\rep\oplus \lin g^-_\rep$ will be denoted by $\o_g$.\end{defi} 

Whence, each of the following affine flags is $g$-invariant: $$ G^+  =\lin G^+_\rep+\o_g,\qquad G^-  =\lin G^-_\rep+\o_g,\qquad G^0 =\lin G^0_\rep+\o_g.$$Moreover, inside $G^+$ the notion of \emph{parallel to $\lin g^+_\rep$} is preserved by $g$. Indeed $$g(\lin g^+_\rep+\o_g)=\lin g^+_\rep +\o_g+\pi^0(\tra g).$$ The same thing happens for $G^-$ and $\lin g^-_\rep$.

\begin{figure}[h]\centering
\begin{tikzpicture}\begin{scope}[rotate=0,scale=0.7]

\begin{scope}[xslant=0.2,yslant=0.1,shift={(2.2,2.5)}]
\filldraw[fill=gray!10, draw=black] (0,0) rectangle (5.4,-2.5);
\draw[dashed] (0.3,0) -- (0.3,-2.5);
\end{scope}

\begin{scope}[xslant=1,shift={(0,0)}]
\filldraw[fill=gray!45,  draw=black] (0,0.8) rectangle (5,4.5);
\draw[dashed] (4.75,0.8) -- (4.75,4.5);
\draw[->] (0.5,1) ..  controls (1,2.5) ..   (4,2.8);
\end{scope}

\begin{scope}[xslant=0.2,yslant=0.1,shift={(2.2,2.5)}]
\filldraw[fill=gray!10, draw=black] (0,0) rectangle (5.4,3.5);
\draw[dashed] (0.3,0) -- (0.3,3.5);
\draw[->] (1.5,0) -- (3,0);
\draw[->] (0.5,0.5) .. controls (4,1) .. (4.5,3);
\end{scope}


\end{scope}
\end{tikzpicture}
\caption{Dynamics of a $(\rep,X_0)$-compatible element of $\sf G\ltimes V$.}
\end{figure}
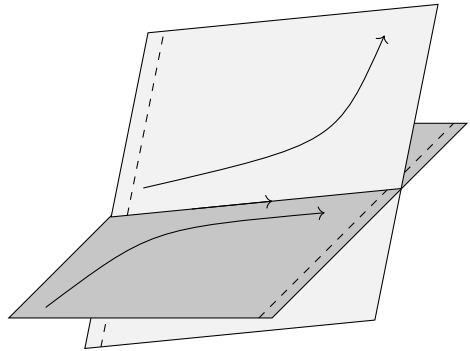

For an affine $p$-dimensional subspace $U\subset V$ we consider $\lin U=\{u-v:u,v\in U\}$, it is a vector subspace of $V$. We also consider the distance on the space of affine $p$-dimensional spaces defined by $$d_{\Aff}(U,W):=d_{\grass_p(V)}(\lin U,\lin W)+\inf\big\{\|u-w\|_\rep:u\in U\textrm{ and }w\in W\big\}.$$

\begin{lemma}\label{atraco}There exists $c$ only depending on $\rep$ so that if $g$ is $(\rep,X_0)$-compatible and $B\subset V$ is affine, co-dimension $k$, transverse to ${\lin g}^-_\rep$ and $\{b\}=B\cap g^-_\rep$ then,$$d_\Aff(gB,G^+)\leq  \frac{c}{e^{-(\lin g^-_\rep|\lin B)_{\rep o}}}e^{-\min\{\sroot(\jordan(\lin g)):{\sroot\in\vt}\}+\|(\lin g^+|\lin g^-)\|}+\|\lin g|\lin g^-_\rep\|\|b-\o_g\|.$$\end{lemma}

\begin{proof} One has $gB=\lin g\lin B+g(b).$ Since by Lemma \ref{proximaltheta} $\lin G^+_\rep$ is attracting for $\lin g$ with repelling flag $\lin g^-_\rep$, the distance between $\lin g\lin B$ and $\lin G^+_\rep$ is controlled by Corollary \ref{anguloscasa} together with Lemma \ref{contraccionproximal} giving the first term in the inequality.

To control the second term, recall we have defined $\umargulis(g)=\pi^0_g(v_g)$, so $\o_g+\umargulis(g)\in G^+.$ By definition $b=v+\o_g$ for some $v\in\lin g^-_\rep$, so $gb=\lin g v+\o_g+\umargulis(g)$. Whence $$\inf\big\{\|gw-u\|:w\in B, u\in G^+\big\}\leq\big\|gb-\big(\o_g+\umargulis(g))\big)\big\|\leq\|\lin gv\|,$$ completing the proof.\end{proof}

\begin{obs} We remark that, as $g$ is $(\rep,X_0)$-compatible, the spectral radius of $\lin g|\lin g^-_\rep$ is strictly smaller than $1$, so $\|\lin g^n|\lin g^-_\rep\|\to0$ as $n\to\infty$. Moreover,  Smilga \cite[Proposition 7.26]{SmilgaAnnalen} states that given $C\geq1$ there exists $c$ such that every $(\rep,X_0)$-compatible $g$ with $ \|(\lin g^+|\lin g^-)\|\leq C$ it holds $$e^{-\min\{\sroot(\cartan(\lin g)):\sroot\in\vt\}}\leq c\|\lin g|\lin g^-_\rep\|.$$
\end{obs}

\begin{defi}\label{affinestrength} The \emph{affine-contraction} of a $(\rep,X_0)$-compatible $g\in\sf G\ltimes V$ is $$\varsigma(g):=\|\lin g|\lin g^-_\rep\|\cdot\|\lin g^{-1}|\lin G^+_\rep\|\cdot e^{\|\o_g\|}.$$ \end{defi}

We conclude the section with the following perturbation Lemma.

\begin{lemma}\label{gecuadrado} Let $h\in\sf G\ltimes V$ be $(\rep,X_0)$-compatible and loxodromic. Then, there exists a non-empty Zariski-open subset $\EuScript G_h$ of $\sf G^2$ such that for every pair $f,q\in\sf G\ltimes V$ with $(\lin f,\lin q)\in\EuScript G_h$ the following holds.
\begin{enumerate}\item The sequence $\R_+\cdot\jordan(\lin f\lin h^n\lin q)\to\R_+\cdot\jordan(\lin h)$ as $n\to\infty$. In particular $fh^nq$ is $(\rep,X_0)$-compatible for all large enough $n$. \item The attracting affine flag of $fh^nq$ converges, as $n\to\infty$, to $f(H^+)$ and the repelling affine flag converges to $q^{-1}(H^-)$.
\end{enumerate}
\end{lemma}

\begin{proof} We commence by considering the Zariski-open subset $\mathrm G_{h}\subset\sf G^2$ of Lemma \ref{modif}. For every $(\lin f, \lin q)\in\mathrm G_{h}$ the Lemma implies $$\frac{\jordan(\lin f\lin h^n\lin q)-n\jordan(\lin h)}{n}\xrightarrow[n\to\infty]{} 0,$$ giving the first statement in item (i). Since $\lin h$ is loxodromic, $\jordan(h)$ lies in the interior of $\X_\rep$ and thus item (i) follows readily.

Now, if $\lin f,\lin q$ moreover verify \begin{equation}\label{ab1}\lin f(\lin H^+_\rep)\cap  \lin q^{-1}(\lin h^-_\rep)=\{0\}\end{equation}
then, there exists a neighborhood $\EuScript U$ of $f(H^+)$ such that for all $B\in\EuScript U$ one has $q(B)\transverse \lin h^-_\rep$. Lemma \ref{atraco} implies that $h^nq(B)\to H^+$, so $fh^nq(B)\to f(H^+)$ as $n\to\infty$; i.e. if $n$ is large enough, there is a small neighborhood of $f(H^+)$ that is sent to itself by $fh^nq$, entailing that the attracting affine flag of $fh^nq$ is arbitrarily close to $f(H^+)$ as $n\to\infty$. To deal with the repelling flag, the argument follows by \begin{equation}\label{ab2} \lin q^{-1}(\lin H^-_\rep)\cap  \lin f(\lin h^+_\rep)=\{0\}.\end{equation}

The lemma is settled by observing that, since all the flags $\lin h^-_\rep,\lin h^+_\rep,\lin H^-_\rep$ and $\lin H^+_\rep$ are fixed, Equations \eqref{ab1} and \eqref{ab2} determine Zariski-open non-empty subsets of $\sf G^2$, which we further intersect with $\mathrm G_{hh}$ to conclude the proof.\end{proof}

\section{The Affine Ratio and Affine Limit Cone}

\subsection{The Af{}fine Ratio}

Let us consider a real vector space $V$ of dimension $2k+l$ and the incomplete flag space $$\FF_{k,l}(V)=\big\{(\ell,W):\ell\in\grass_k(V),W\in\grass_{k+l}(V),\,\ell\subset W\big\}.$$ It is a self-dual flag space of $V$ and two such flags, $(a,A),(b,B),$ are in general position if $a\cap B=b\cap A=\{0\},$ in this case $A\cap B$ has dimension $l.$ For $v\in V,$ let us denote by $(a,A)+v$ the affine flag $(a+v,A+v).$

Consider four flags $(a^+,A^+),(a^-,A^-),(b^+,B^+),(b^-,B^-)$ in $\FF_{k,l}(V),$ pairwise in general position, and consider also two arbitrary vectors $v,w\in V.$ We will define an invariant of the four-tuple of affine flags \begin{alignat*}{4}\sf a^+ & =(a^+,A^+)+w,\qquad &\sf a^- & =(a^-,A^-)+w,\\ \sf b^+ & =(b^+,B^+)+v,\qquad &\sf b^- & =(b^-,B^-)+v,\end{alignat*} which we call, by analogy with the cross ratio of four lines in the plane, the \emph{affine ratio}. Let us denote by $A^0=A^+\cap A^-$ (and $B^0=B^+\cap B^-$). This invariant will be an affine map from $A^0+w$ to itself, $$\raff(\sf a^-,\sf b^+,\sf b^-,\sf a^+):A^0\to A^0,$$ defined by following procedure.

Consider $u\in A^0$ so that $u+w\in A^0+w\subset A^-+w,$ translate $u+w$ parallel to $a^-+w$ until one reaches $A^-+w\cap(B^++v),$ call this vector $u_1.$ Translate $u_1$ parallel to $b^++v$ until reaching $u_2\in(B^++v)\cap (B^-+v),$ translate $u_2$ parallel to $b^-+v$ until reaching $u_3\in(B^-+v)\cap A^+,$ translate $u_3$ parallel to $a^++w$ until reaching again $u_4\in( A^++w)\cap (A^-+w)=A^0+w.$ Then we let (see Figure \ref{AffRatio}) 
 $$\raff(\sf a^-,\sf b^+,\sf b^-,\sf a^+)(u):=u_4-w\in A^0.$$

\begin{figure}[h]\centering
\begin{tikzpicture}\begin{scope}[yslant=0.1,scale=0.7]

\begin{scope}[xslant=0.2,yslant=0.1,shift={(2.2,2.5)}]
\filldraw[fill=gray!10, draw=black] (0,0) rectangle (5.4,-2.5);
\node[below right] at (5.4,-2.5) {$B^++v$};
\draw[dashed] (4.8,0) -- (4.8,-2.5);
\end{scope}

\begin{scope}[xslant=1,shift={(0,0)}]
\filldraw[fill=gray!45,  draw=black] (0,0) rectangle (5,4.5);
\node[below right] at (5,0) {$A^0+w$};
\node[right] at (5,4.5) {$A^-+w$};
\draw[dashed] (4.75,0) -- (4.75,4.5);
\draw[color=black] (4,0) -- (4,1.5);
\draw[color=black,>-] (4,1.5) -- (4,3.15);
\node[below] at (4,0) {$u+w$};
\end{scope}

\begin{scope}[xslant=0.2,yslant=0.1,shift={(2.2,2.5)}]
\filldraw[fill=gray!10, draw=black] (0,0) rectangle (5.4,3.5);
\draw[dashed] (4.8,0) -- (4.8,3.5);
\draw[color=black] (4.32,0) -- (4.32,1.5);
\draw[color=black,>-] (4.32,1.5) -- (4.32,3.15);
\end{scope}

\begin{scope}[xslant=1.6,yslant=0.3,shift={(-5.4,5)}]
\filldraw[fill=gray!25,  draw=black] (0,0) rectangle (3.4,2);
\node[left] at (0,0) {$B^-+v$};
\draw[dashed] (0.5,0)--(0.5,2);
\draw[color=black] (3.09,1) -- (3.09,2);
\draw[color=black,-<] (3.09,0) -- (3.09,1);
\end{scope}

\draw (0,0) rectangle (5,5);
\draw[dashed] (0.2,0)--(0.2,5);
\node[above] at (0.2,5) {$A^++w$};
\draw[-<] (4.58,0)--(4.58,2);
\draw[] (4.58,2)--(4.58,4.31);

\end{scope}
\end{tikzpicture}
\caption{Definition of the Affine Ratio}\label{AffRatio}
\end{figure}
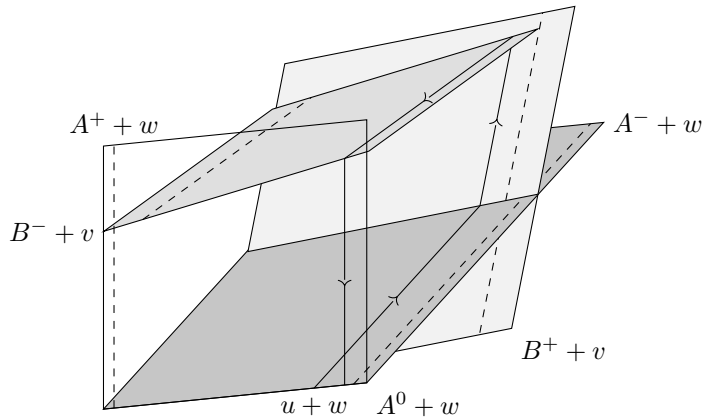

\begin{obs}\label{affR-invariante}The affine ratio is invariant under the action of $\GL(V)\ltimes V$ in the following sense: if $f\in\GL(V)\ltimes V$ then $$f\raff(\sf a^-,\sf b^+,\sf b^-,\sf a^+)f^{-1}=\raff(f\sf a^+,f\sf a^-,f\sf b^+,f\sf b^-).$$\end{obs}

Given a decomposition $V=U\oplus W$ let us denote by $\pi^{U,W}:V\to U$ the projection along it, so $\id=\pi^{U,W}\oplus\pi^{W,U}.$ The next lemma computes the translation part of $\raff$.

\begin{lemma}\label{formulaAffR}One has $\raff(\sf a^-,\sf b^+,\sf b^-,\sf a^+)(0)=\pi^{A^0,a^+\oplus b^-}\pi^{b^+\oplus a^-,B^0}(v-w).$
\end{lemma}

\begin{proof}Without loss of generality we may assume that $w=0.$ Using the definition of $\raff$ and the fact that $u=0$ one directly observes $u_1=\pi^{a^-B^+}v.$ From the decomposition $u_1-v=\pi^{b^+B^-}(u_1-v)+\pi^{B^-b^+}(u_1-v),$ one obtains that $$u_2=u_1-\pi^{b^+B^-}(u_1-v)=\pi^{b^+B^-}v+\pi^{B^-b^+}\pi^{a^-B^+}v.$$

\noindent
Now one has $u_3=\pi^{A^+b^-}u_2$; so finally we get \begin{alignat*}{2}\raff(x^+,x^-,y^+,y^-)(0)=u_4 & =\pi^{A^-a^+}u_3=\pi^{A^-a^+}\pi^{A^+b^-}\big(\pi^{b^+B^-}v+\pi^{B^-b^+}\pi^{a^-B^+}v\big)\\ & = \pi^{A^0,a^+\oplus b^-}\big(v-\pi^{B^-b^+}v+\pi^{B^-b^+}\pi^{a^-B^+}v\big)\\ & = \pi^{A^0,a^+\oplus b^-}\big(v+\pi^{B^-b^+}(\pi^{a^-B^+}v-v)\big) \\ & = \pi^{A^0,a^+\oplus b^-}\big(v-\pi^{B^-b^+}(\pi^{B^+a^-}v)\big)\\ & = \pi^{A^0,a^+\oplus b^-}\big(v-\pi^{B^-\cap B^+,b^+\oplus a^-}v\big)\\ & = \pi^{A^0,a^+\oplus b^-}\pi^{b^+\oplus a^-,B^0}(v),
\end{alignat*}as desired.\end{proof}

In particular, if all affine flags have one point in common $p=a^+\cap a^-\cap b^+\cap b^-$, then the translation part of $\raff$ is $0$ since we can then then take $v=w=p$.

An \emph{affine $\vt$-flag} of $\sf G\ltimes  V$ is an affine partial flag $(x_\rep,X_\rep)+v,$ for some $x\in\EuScript F_\vt(\sf G)$ and $v\in V.$ We now consider the affine ratio as an invariant of four affine $\vt$-flags. We will also normalize it so as to consider it as a map of $\Vlongest$ to itself. 

\begin{defi}\label{AffV0}Let $X^+,X^-,Y^+,Y^-$ be a four-tuple of pairwise transverse affine $\vt$-flags of $\sf G\ltimes  V,$ and define $$\raffN(X^-,Y^+,Y^-,X^+)  :\Vlongest\to \Vlongest$$ by conjugating $\raff\big(X^-,Y^+,Y^-,X^+\big)$ with the map $$\rep\big(\psi_{(x^+,x^-)}\big)\circ \trs_{-w}$$ that sends the pairs of affine flags $X^+=(x^+_\rep,X^+_\rep)+w$ and $X^-=(x^-_\rep,X^-_\rep)+w$ to $(a^+,A^+)$ and $(a^-,A^-).$ Finally, let us define the \emph{Translation Affine Ratio} $$\raffi(X^-,Y^+,Y^-,X^+)\in\trivia$$ as the translation part of $\raffN(X^-,Y^+,Y^-,X^+):\Vlongest\to \Vlongest$ along the neutralizing space $\trivia,$ this is to say $$\raffi(X^-,Y^+,Y^-,X^+):=\pi^\trivia\big(\raffN(X^-,Y^+,Y^-,X^+)(0)\big).$$\end{defi}

For a subset $\GA<\sf G\ltimes V$ we let $ \lin \GA<\sf G$ be the subset consisting on its linear parts. The purpose of this section is the following result.

\begin{prop}\label{hiperplanos}\label{interiorneu} Assume $\sf G$ is Zariski-connected. Let $\GA<\sf G\ltimes V$ be a Zariski-dense sub-semi-group such that $\Bcone_{\lin \GA}\cap\X_\rep$ has non-empty interior. Then the set $$\big\{\raffN(G^-,F^+,F^-,G^+)(0): f,g\in\GA\textrm{ are $(\rep,X_0)$-compatible and transverse}\big\}$$ is not contained in a hyperplane of $\Vlongest.$ Consequently, the set $$\big\{\raffi(G^-,F^+,F^-,G^+): f,g\in\GA\textrm{ are $(\rep,X_0)$-compatible and transverse}\big\}$$ is not contained in a hyperplane of $\trivia.$\end{prop}

\begin{proof} By Zariski-density, and since $\Bcone_{\lin \GA}\cap\X_\rep$ has non-empty interior, there exists a pair $g,h\in\GA$ such that the corresponding linear parts $\lin g$ and $\lin h$ are $(\rep,X_0)$-compatible, loxodromic and  transverse.

By conjugating $\GA$ we may, and will, assume that (recall Definition \ref{origen}) $$\o_{g}  =0,\  \lin g^+  =[\sf P^\vt],\textrm{ and } \lin g^-  =[\wk{\sf P}^\vt],$$ so that $(\lin g^+_\rep,\lin G^+_\rep)=(a^+,A^+)$, $(\lin g^-_\rep,\lin G^-_\rep)=(a^-,A^-)$, $A^0=\Vlongest$, $G^+=A^+$ and $G^-=A^-$. Applying Lemma \ref{formulaAffR} we see that we have to understand the vector 
\begin{alignat}{2}\label{inter}\nu(g,h) & :=\pi^{\Vlongest,a^+\oplus \lin h^-_\rep}\pi^{\lin h^+_\rep\oplus a^-,\lin H^0}(\o_h)\nonumber\\
&=\pi^{\Vlongest,a^+\oplus \lin h^-_\rep}\big((a^-\oplus \lin h^+_\rep)\cap H^0\big),\end{alignat}  for different choices of $h$.  More precisely, we fix $g$ and $h$ as above and a hyperplane $U\subset \Vlongest$, we will find $f,q\in\GA$ and $n\in\N$ such that $\nu(g,fh^nq)\not\in U.$

We apply Lemma \ref{gecuadrado} to find a non-empty Zariski-open subset $\EuScript G_h\subset \sf G^2$ such that if $f,q\in\sf G\ltimes V$ verify $(\lin f,\lin q)\subset\EuScript G_h$ then, for all large enough $n$ one has $fh^nq$ is $(\rep,X_0)$-compatible with attracting flag arbitrary close to $f(H^+)$ and repelling flag arbitrary close to $q^{-1}(H^-)$ (as $n\to\infty$).

To find elements $f,q\in\GA$ such that $\nu(g,fh^nq)\not\in U$ for big enough $n$, we use Equation \eqref{inter} and the above descriptions of the invariant affine flags of $fh^nq$ to see that we must find $f,q,\in\GA$ such that $(\lin f,\lin q)\in\EuScript G_h$ and \begin{equation}\label{abierto3}\big(U\oplus a^+\oplus \lin q^{-1}(\lin h^-_\rep)\big) \cap \Big(\big(a^-\oplus \lin f(\lin h^+_\rep)\big)\cap\big( f (H^+)\cap q^{-1} (H^-)\big)\Big)=\vacio. \end{equation}

This last equation defines a Zariski-open subset of $(\sf G\ltimes  V)^2$, the variables being $(f,q)$. Since $(\sf G\ltimes V)^2$ is Zariski-connected, it is Zariski-irreducible and thus a finite collection of non-empty Zariski-open subsets has non-trivial intersection. We already now that $\EuScript G_h$ is non-empty, so provided \eqref{abierto3} defines a non-empty set we have: since $\GA$ is Zariski-dense in $\sf G\ltimes  V,$ the product semi-group $\GA\times\GA$ is Zariski-dense in $(\sf G\ltimes V)^2$ and thus a pair $(f,q)\in\GA^2$ with $(\lin f,\lin q)\in\EuScript G_h$ and satisfying Equation \eqref{abierto3} will exist.

To complete the proof it remains thus to show that there exist $f,q\in\sf G\ltimes  V$ that verify \eqref{abierto3}. Consider then $\lin f=\lin q=\id.$ It suffices to find $v\in V$ such that \begin{equation}\label{vac}\big(U\oplus a^+\oplus \lin h^-_\rep\big) \cap \big( \lin h^+_\rep\oplus a^-\big)\cap\Big( \trs_v\big( H^+\cap H^-\big)\Big)=\vacio.\end{equation}

\noindent
In order to do so, we observe that \begin{alignat*}{2}&\dim\Big(\big(U\oplus a^+\oplus \lin h^-_\rep\big) \cap \big( \lin h^+_\rep\oplus a^-\big)\Big)=2k-1,\\ &\dim H^+\cap H^-=l.\end{alignat*} The second equality is obvious, the first one follows since $U\oplus a^+\oplus \lin h^-_\rep$ has co-dimension $1$ and $\lin h^+_\rep\oplus a^-$ is $2k$-dimensional (by transversality of $g$ and $h$) and not contained in the former. 

Thus, as the dimensions do not add up to $\dim V=2k+l$, we can translate the latter as to not intersect the former. i.e. there exists $v\in V$ such that Equation \eqref{vac} holds, as desired. \end{proof}

\subsection{The affine ratio and additivity default of Marguils' invariants}

Given $C>0$, we say that two $(\rep,X_0)$-compatible elements $g_0,g_1\in\sf G\ltimes  V$ are a $C$-\emph{transverse pair} if for all $i,j\in\{0,1\}$ one has $\|(\lin g^+_i|\lin g^-_j)\|\leq C$.

\begin{thm}[Smilga]\label{SM} Given $\eps>0$ there exists $C>0$ such that if $f,q\in\sf G\ltimes  V$ are a $(\rep,X_0)$-compatible $C$-transverse pair with affine contraction $\leq1/C$, then $$\big\|\margulis(fq)-\big(\margulis(f)+\margulis(q)\big)-\raffi(Q^-,Q^+,F^-,F^+)\big\|<\eps.$$
\end{thm}

\begin{proof}[Quick sketch of proof] The proof is essentially contained in Smilga \cite[Proposition 9.3]{SmilgaAnnalen}, however the statement is not exactly the same us ours so we just explain the main ideas.

We want to understand the Margulis invariant of $fq$ by decomposing it into its factors $f$ and $q$. The former is, up to normalizing its invariant flags, the translation part of $fq$ restricted to the neutralizing space $$\trivia^{\big((\lin f\lin q)^+,(\lin f\lin q)^-\big)}\subset (FQ)^0.$$ If the angles between the invariant flags of $\lin f$ and $\lin q$ are controlled, and moreover the affine contraction $\varsigma(f)$ is small enough (Definition \ref{affinestrength}), Lemma \ref{atraco} implies that, the attracting flag of $fq$ is close to $F^+$. An analogous control for quantities related to $q^{-1}$ implies that the repelling flag of $fq$ is arbitrarily close to $Q^-$, so the ideally neutral space $(FQ)^0$ of $fq$ is close to $F^+\cap Q^-$. Similarly $qf$ has attracting flag close to $Q^+$, repelling flag close to $F^-$, and ideally neutral space close to $Q^+\cap F^-$. See Figure \ref{Aditivity}.

Now, the map $q$ conjugates $fq$ and $qf$, so $q$ sends $(FQ)^0$ so $(QF)^0$, and analogously, $f$ sends $(QF)^0$ to $(FQ)^0$.

The idea is then to decompose the map $q:(FQ)^0\to(QF)^0$ as the projection from $(FQ)^0$ to $Q^0$ parallel to $q^-$, followed by the projection $Q^0$ to $(QF)^0$ parallel to $q^+$ and some quasi-translation of $(FQ)^0$. Some errors have to be taken into account here as $(FQ)^0$ is not exactly $F^+\cap Q^-$ but only close to it, these errors tend to be negligible. However, for the approximation of $q:(FQ)^0\to(QF)^0$ with this composition of projections and a quasi-translation to be good, a control on the total separation of the spaces, namely, for example a control on $\|\o_q\|$ and $\|\o_f\|$ is sufficient (together with the previous control on angles/size of $q$ and $f$). This is resolved in Lemmas 9.8 and 9.10 and Corollary 9.9 of Smilga \cite{SmilgaAnnalen}.

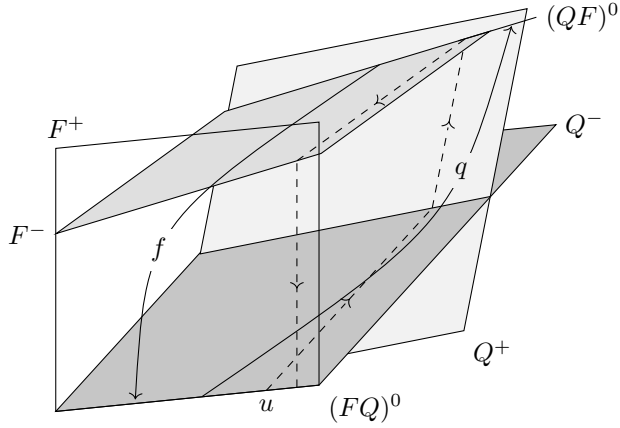
\begin{figure}[h]\centering
\begin{tikzpicture}\begin{scope}[yslant=0.1,scale=0.7]

\begin{scope}[xslant=0.2,yslant=0.1,shift={(2.2,2.5)}]
\filldraw[fill=gray!10, draw=black] (0,0) rectangle (5.4,-2.5);
\node[below right] at (5.4,-2.5) {$Q^+$};
\end{scope}

\begin{scope}[xslant=1,shift={(0,0)}]
\filldraw[fill=gray!45,  draw=black] (0,0) rectangle (5,4.5);
\node[below right] at (5,0) {$(FQ)^0$};
\node[right] at (5,4.5) {$Q^-$};
\draw[color=black,dashed] (4,0) -- (4,1.5);
\draw[color=black,>-,dashed] (4,1.5) -- (4,3.15);
\node[below]  at (4,0) {$u$};
\node (u) at (2.7,-0.1) {};
\node (v) at (1.6,-0.1) {};
\end{scope}

\begin{scope}[xslant=0.2,yslant=0.1,shift={(2.2,2.5)}]
\filldraw[fill=gray!10, draw=black] (0,0) rectangle (5.4,3.5);
\draw[color=black,dashed] (4.32,0) -- (4.32,1.5);
\draw[color=black,>-,dashed] (4.32,1.5) -- (4.32,3.15);
\end{scope}

\begin{scope}[xslant=1.6,yslant=0.3,shift={(-5.4,5)}]
\filldraw[fill=gray!25,  draw=black] (0,0) rectangle (3.4,2);
\node[left] at (0,0) {$F^-$};
\draw[color=black,dashed] (3.09,1) -- (3.09,2);
\draw[color=black,-<,dashed] (3.09,0) -- (3.09,1);
\node (f) at (3.5,2.2) {};
\node (h) at (2,2.1) {};
\node[right] at (4,2) {$(QF)^0$};
\draw (3,2)--(4,2);

\end{scope}

\draw (0,0) rectangle (5,5);
\node[above] at (0.2,5) {$F^+$};
\draw[-<,dashed] (4.58,0)--(4.58,2);
\draw[dashed] (4.58,2)--(4.58,4.31);
\end{scope}

\draw[->,thin] (u) ..  controls (5.3,2.5) ..   (f);
\node[left,fill=gray!10] at (5.6,3.2) {$q$};

\draw[->,thin] (h) ..  controls (1.2,2.5) ..   (v);
\node[left,fill=white] at (1.6,2.15) {$f$};

\end{tikzpicture}
\caption{Schematic situation in Theorem \ref{SM}, the ideally neutral spaces $(FQ)^0$ and $(QF)^0$ are very close (but do not coincide with) $F^+\cap Q^-$ and $Q^+\cap F^-$ respectively.}\label{Aditivity}
\end{figure}

On readily sees then the Affine Ratio of Figure \ref{AffRatio} appearing as the corresponding default, giving the result.\end{proof}

\subsection{The affine limit cone is convex and has non-empty interior}We define the \emph{affine limit cone} $\affinelim_{\GA}$ of a Zariski dense sub-semigroup $\GA<\sf G\ltimes  V$ as the smallest closed cone of $\trivia$ that contains the set of Margulis invariants $$\affinelim_\GA=\overline{\big\{\R_+\cdot\margulis(f):f\in\GA \textrm{ is $(\rep,X_0)$-compatible}\big\}}.$$

In light of Proposition \ref{hiperplanos} and Theorem \ref{SM} one has the following result. Similar versions will also appear in Kassel-Smilga \cite{KS} and in Ghosh \cite{Souravpersonalcomm}.

\begin{cor}\label{interiorAL}Let $\rep:\sf G\to \GL(V)$ be an irreducible representation with non-trivial neutralizing space and let $\GA<\sf G\ltimes V$ be a Zariski-dense sub-semi-group such that $\inte\Bcone_{\lin \GA}\cap\X_\rep\neq\vacio$. Then $\affinelim_\GA$ is convex and has non-empty interior.\end{cor}

In contrast with Benoist's limit cone, it is fairly common for $\affinelim_\GA$ to be the whole space $\trivia.$ Indeed, this is the case when $\longest$ acts trivially on $\trivia.$

\begin{proof} We first establish convexity. To that end, we fix a pair $g,h\in\GA$ of $(\rep,X_0)$-compatible, loxodromic elements, we have to find an element in $\GA$ whose Margulis invariant lies about $\R_+(\margulis(g)+\margulis(h))$. We consider the Zariski open sets $\EuScript G_g$ and $\EuScript G_h$ given by Lemma \ref{gecuadrado} for $g$ and $h$. Since $g$ and $h$ are fixed, each of the equations \begin{alignat*}{4} \lin f(\lin H^+_\rep)\cap\lin h^-_\rep & =\{0\},
 \qquad &\lin H^+_\rep\cap q^{-1}(\lin h^-_\rep) &=\{0\},
 \\ \lin G^-_\rep\cap \lin q_0^{-1}(\lin g^-_\rep)& =\{0\},
 \qquad &\lin f_0(\lin G^+_\rep)\cap\lin g^-_\rep & =\{0\},
 \\ \lin H^+_\rep\cap \lin q_0^{-1}(\lin g^-_\rep)&=\{0\},
 \qquad &\lin G^+_\rep\cap \lin q^{-1}(\lin h^-)&=\{0\},\end{alignat*} defines a Zariski-open non-empty subset of $\sf G$. Since $\sf G$ is Zariski-irreducible finite collections of non-empty open sets intersect and thus there exist $f,q,f_0,q_0\in\GA$ such that $(\lin f,\lin q)\in\EuScript G_h$, $(\lin f_0,\lin q_0)\in\EuScript G_g$, that moreover verify the corresponding equations above. For large enough $n$ we have thus, by means of Lemma \ref{gecuadrado}, that

\begin{itemize} \item[-] $fh^nq$ and $f_0g^nq_0$ are $(\rep,X_0)$-compatible, \item[-] $fh^nq$ and $h$ are transverse, \item[-] $f_0g^nq_0$ and $g$ are transverse, \item[-] for all $m\geq1$ the elements $h^mfh^nq$ and $g^mf_0g^nq_0$ are also transverse (and $(\rep,X_0)$-compatible by Lemma \ref{pares}).\end{itemize}

Applying now Theorem \ref{SM} we see that, for big enough $n$, \begin{alignat*}{2}\lim_{k\to\infty}\frac{\margulis(g^kf_0g^nq_0h^kfh^nq)}k & =\lim_{k\to\infty}\frac{\margulis(g^kf_0g^nq_0)+\margulis(h^kfh^nq)}k\\ &=\lim_{k\to\infty}\frac{\margulis(g^k)+\margulis(f_0g^nq_0)+\margulis(h^k)+\margulis(fh^nq)}k \\&=\margulis(g)+\margulis(h),\end{alignat*} proving convexity of $\affinelim_{\GA}$.

To prove that $\affinelim_{\GA}$ has non-empty interior we use Theorem \ref{SM} together with Proposition \ref{interiorneu}. Indeed, if there exists a hyperplane $U\subset\trivia$ such that $\margulis(g)\in U$ for every $(\rep,X_0)$-compatible $g\in\GA$, then for any transverse pair of $(\rep,X_0)$-compatible elements $f,q\in\GA$ we have $$\raffi(Q^-,F^+,F^-,Q^+)=\lim_{n\to\infty} \margulis(f^nq^n)-n(\margulis(f)+\margulis(q))\in U,$$ contradicting Proposition \ref{interiorneu}. Finally, a convex cone that is not contained in a hyperplane has non-empty interior, completing the proof.\end{proof}

\section{The case of reducible representations}

We now consider a finite collection of non-trivial irreducible representations $\{\rep_i:\sf G\to \SL(V_i)\}_{i\in I}$ with $0\in\poids_{\rep_i}$ and study the representation $$\rep:=\bigoplus_i\rep_i:\sf G\to V:=\bigoplus_iV_i.$$ Fix $X_0\in\a^+\cap\Fix(\ii)$ which does not belong to the finite union of kernels $\ker\chi,$ $\chi\in\poids_{\rep_i}\setminus\poids^{\longest}_{\rep_i}$. Let $\vt_i\subset\simple$ be the set of simple roots stabilizing $\poids_{\rep_i}^+$ and $\vt:=\bigcup_{i\in I}\vt_i.$

It we let $$a_i^+=\bigcup_{\chi\in\poids_{\rep_i}^+}V^\chi\textrm{ and }A_i^+=\bigcup_{\chi\in\poids_{\rep_i}^+\cup\poids_{\rep_i}^\longest}V^\chi$$ then the parabolic group $\sf P_\vt$ is the stabilizer in $\sf G$ of $a^+:=\bigoplus_ia_i^+$ and of $A^+:=\bigoplus_i A_i^+$ and, if we denote by $p_i:\EuScript F_\vt\to\EuScript F_{\vt_i} $ the natural projection, then we have a map $\EuScript F_\vt\to\grass_{\sum_i k_i,\sum_i n_i}(V)$ defined by $$x\mapsto(x_\rep,X_\rep):=\Big(\bigoplus_i (p_i(x))_{\rep_i},\bigoplus_i (P_i(X))_{\rep_i}\Big),$$ (recall Equation \eqref{equiv}). We let $\encono_\rep:=\bigcap_i\encono_{\rep_i}$, where $\encono_{\rep_i}$ is defined as in Equation \eqref{encono}, and we say that $g\in\sf G\ltimes_\rep V$ is $(\rep,X_0)$-\emph{compatible} if $\jordan(\lin g)\in\encono_\rep$.

For each $i$ let us denote by $\trivia_i$ the associated neutralizing space and define the \emph{trivializing space} of $\rep$ by $$\trivia=\bigoplus_i\trivia_i.$$ Define also, for a $(\rep,X_0)$-compatible $g=(\lin g,v)\in\sf G\ltimes_\rep V $ its \emph{Margulis invariant} by $$\margulis(g):=\sum_i\margulis(\lin g,v_i),$$ where $v=\sum_i v_i$ in the decomposition $V=\bigoplus_i V_i.$ We extend the definitions of affine ratio and affine limit cone and the proof of Corollary \ref{interiorAL} gives the following:

\begin{cor}\label{rojo}Let $\rep:\sf G\to \SL(V)$ be as above and let $\GA<\sf G\ltimes V$ be a Zariski-dense sub-semi-group such that $\inte\Bcone_{\lin \GA}\cap\X_\rep\neq\vacio$. Then the affine limit cone $\affinelim_\GA\subset\trivia$ is convex and has non-empty interior.\end{cor}

\section{The cocycle viewpoint: Zariski density}\label{cocycleviewpoint}

Let us fix a (possibly reducible) representation $\rep:\sf G\to \SL(V)$. If $\G<\sf G$ is a semi-group, a \emph{cocycle} over $\rep$ is a map $\co:\G\to V$ such that for every $\lin g,\lin h\in\G$ $$\co(\lin g\lin h)=\co(\lin g)+\rep(\lin g)\co(\lin h);$$ it is a \emph{co-boundary} if there exists $v\in V$ with $\co(\lin g)= v-\rep(\lin g)v$.  The vector space $$H^1_\rep(\G,V)=\frac{\{\textrm{cocycles over $\rep$}\}}{\{\textrm{co-boundaries}\}}$$ is called \emph{the first twisted cohomology group}.  Semi-groups of $\sf G\ltimes_\rep V$ whose linear part is $\G$, are in bijective correspondence with cocycles over $\rep$ via $$\co\mapsto\G_\co:=\{(\lin g,\co(\lin g))\in\sf G\ltimes_\rep V:\lin g \in\G\}.$$ Two such groups are conjugated by a pure translation if and only if the associated cocycles differ by a co-boundary, i.e. are cohomologous.

One has the following result of Ghosh \cite{SouravIso} that crucially uses Whitehead's Lemma on the vanishing of the first cohomology of semi-simple Lie algebra representations.

\begin{prop}[{Ghosh \cite[Proposition 6.2]{SouravIso}}]\label{Zarirre} Let $\rep:\sf G\to \SL(V)$ be irreducible and non-trivial. Let $\G<\sf G$ be Zariski-dense and $\co\in H^1_\rep(\G,V)$ be non-trivial. Then the group $\G_\co$ is Zariski-dense in $\sf G\ltimes_\rep V$.
\end{prop}

\begin{proof} We add some details for later use. Let $\sf X$ be the Zariski closure of $\G_\co$ and consider  'the linear part' morphism ${\tt L}:\sf X\to\sf G$. Since $\G$ is Zariski-dense in $\sf G$ we have that ${\tt L}(\sf X)$ is surjective. Thus $\forall g\in\sf G$ there exists $u_g\in V$ with $(g,u_g)\in\sf X$.

If ${\tt L}$ were injective then such $u_g$ would be unique giving a well defined cocycle, $g\mapsto u_g\in V$ of $\sf G$ over $\rep$, extending $\co$. Whitehead's Lemma (see Raghunathan's book \cite{RaghunathanBook}) implies that $u_g$ is cohomologicaly trivial, and so is $\co$ contradicting our assumption. Thus ${\tt L}$ is not injective, which gives a pure translation $(\id,u)\in\sf X$. Since $(g,u_g)(\id,u)(g,u_g)^{-1}=(\id,\rep(g)u)$, the Proposition will be proved once we have stablished that the additive group spanned by $\rep(\sf G)u$ is $V$.

Lemma \ref{esquiva} provides a $g\in\sf G$ such that for all $\chi\in\poids_\rep$ one has $(\rep(g)u)_\chi\neq0$, where, for a vector $w\in V$, we have denoted by $w_\chi$ its componenent in the restricted weight space $V^\chi$, following the restricted weight decomposition of $V$.

Since the additive group spanned by $\rep(\sf G)u$ and that of $\rep(\sf G)\rep(g)u$ coincide, we assume from now that the pure translation $(\id, u)\in\sf X$ is such that $u_\chi\neq0$ for all $\chi\in\poids_{\rep}.$

Since $\poids_{\rep}$ is finite we can consider $z\in\a$ such that the values $\chi(z)$, for $\chi\in\poids_{\rep}$, are pairwise distinct. In particular $\chi(z)\neq0$ for any non-vanishing $\chi$.

Let us fix a weight $\mu\in\poids_\rep\setminus\{0\}$ and consider the linear map $${\tt R}^{\mu}=(\rep(\exp(z))-\id)\prod_{\chi\in\poids_{\rep}-\{0,\mu\}}\left(2\rep(\exp(z))-\rep\Big(\exp\Big(\big(1+\frac{\log 2}{\chi(z)}\big)z\Big)\Big)\right)\in\gl(V).$$ One readily sees that \begin{itemize}\item[-] ${\tt R}^{\mu} u$ belongs to the additive group spanned by $\rep(\sf G)u$, \item[-]the order on the above product is irrelevant as we are only considering  elements of $\exp(\a)$, which commute; whence ${\tt R}^{\mu }u_\chi=0$ for all $\chi\neq\mu$.
\end{itemize}

\noindent
In particular, \begin{alignat*}{2}{\tt R}^{\mu} u & ={\tt R}^{\mu}(u_{\mu}) \\&=\Big((e^{\mu(z)}-1)\prod_{\chi\in\poids_{\rep}-\{0,\mu\}}\big(2e^{\mu(z)}-e^{\mu(z)+\log2\frac{\mu(z)}{\chi(z)}}\big)\Big)(u_{\mu})= c\cdot u_{\mu}\neq0,
\end{alignat*} as the coefficient $c$ is non-zero since $\chi(z)\neq\mu(z)$ for all $\chi\neq\mu$.

One concludes that $\rep(\exp(a)){\tt R}^{\mu} u=\exp(\mu(a))cu_{\mu}$ for all $a\in\a$, so by also considering differences, we get that the line $\R(u_\mu)$ is contained in the additive group spanned by $\rep(\sf G)u$. Now, the additive group spanned by $\rep(\sf G)\R (u_{\mu})$ coincides with the vector space spanned by $\rep(\sf G) (u_{\mu})$, which is $V$ by irreducibility of $\rep$.\end{proof}

\begin{lemma}\label{parcial} Let $\rep:\sf G\to\SL(V)$ and $\uppsi:\sf G\to\SL(W)$ be two representations with $\rep$ irreducible and such that there exists $0\neq\mu\in\poids_\rep\setminus \poids_{\uppsi}$. Let $\G<\sf G$ be a Zariski-dense semigroup, $\co_V:\G\to V$ a non-coboundary cocycle and $\co_W:\G\to W$ a cocycle. Denote by $$\co=\co_V+\co_W:\G\to V\oplus W.$$ Then the Zariski closure $\sf X$ of $\G_\co$ contains $\sf G\ltimes V$. If moreover $\G_{\co_W}$ is Zariski-dense in $\sf G\ltimes W$, then $\sf X=\sf G\ltimes(V\oplus W)$.
\end{lemma}

\begin{proof} Let $\sf X$ be the Zariski closure of $\G_{\co}$. Consider the projection $\pi:\sf G\ltimes(V\oplus W)\to\sf G\ltimes V$ given by $\pi(g,v+w)=(g,v)$. Since $\coclase_V\neq0$, Proposition \ref{Zarirre} implies that $\pi|\sf X$ is surjective. Whence, for every $v\in V$ there exists $w\in W$ such that $(\id,v+w)\in\sf X$. Fix $\mu\in\poids_\rep\setminus\poids_\uppsi$ and choose $v$ so that $v_\mu\neq0$.

We consider now the restricted weight space decomposition $$V\oplus W=\sum_{\chi\in\poids_\rep\cup\poids_{\uppsi}}(V\oplus W)^\chi,$$ and denote, for $u\in V\oplus W$ by $u_\chi\in (V\oplus W)^\chi$ the associated component. Observe that $(V\oplus W)^\mu= V^\mu$ and thus $(v+w)_\mu=v_\mu$.

We now proceed again as in the proof of Proposition \ref{Zarirre} by considering the (modified) operator $${\tt R}^{\mu}=(\rep(\exp(z))-\id)\prod_{\chi\in\poids_{\rep}\cup\poids_\uppsi-\{0,\mu\}}\left(2\rep(\exp(z))-\rep\Big(\exp\Big(\big(1+\frac{\log 2}{\chi(z)}\big)z\Big)\Big)\right),$$ and applying it to $u=v+w$. The same arguments lead to the desired inclusion.

To prove the last item, we consider the projection $\pi^W:\sf G\ltimes(V\oplus W)\to\sf G\ltimes W$ given by $(g,v+w)\mapsto(g,w)$. By assumption $\pi^W(\sf X)=\sf G\ltimes W$ whence for each $w\in W$ and $g\in\sf G$ there exists $v\in V$ with $(g,v+w)\in\sf X$. However, as we have established that $\sf G\ltimes V\subset\sf X$, we obtain that $(g,w)=(g,v+w)\cdot(\id,-v)\in\sf X$. Thus, $\sf X$ contains both $\sf G\ltimes V$ and $\sf G\ltimes W$, giving the result.\end{proof}

Let us introduce the following definition.

\begin{defi}\label{disjoined} A finite collection of irreducible representations $\{\rep_i:\sf G\to V_i\}_{i\in I}$ is \emph{disjoined} if we can order $I=\lb1,k\rb$ such that $\rep_1$ is non-trivial and for each $i\geq 2$ there exists $0\neq\mu_i\in\poids_{\rep_i}\setminus\bigcup_{l=1}^{i-1}\poids_{\rep_{l}}.$ A representation $\rep$ is \emph{disjoined} if the collection of its factors is.
\end{defi}

\begin{obs}\label{adjuntadisjunda} Observe that the Adjoint representation is always disjoined (regardless that $\sf G$ has isomorphic factors) since the restricted weights of each irreducible factor of this representation lie on different factors of $\a^*$.
\end{obs}

\begin{cor}\label{souravzariski} Let $\{\rep_i:\sf G\to \SL(V_i)\}_{1}^k$ be a disjoined collection. Let $\G<\sf G$ be a Zariski-dense semigroup. Consider for each $i$ a non-coboundary cocycle $\co_i:\G\to V_i$ and define $\co=\sum_i\co_i:\G\to\bigoplus_iV_i$. Then $\G_\co$ is Zariski-dense in $\sf G\ltimes\big(\bigoplus_1^k V_i\big)$.
\end{cor}

\begin{proof} Follows by induction applying Proposition \ref{Zarirre} and Lemma \ref{parcial}.\end{proof}

We fix from now on a disjoined representation $\rep:\sf G\to V$ and a cocycle $\co:\G\to V$ over $\rep$. Remark \ref{Invariancia} implies that if $\lin g\in\G$ is a ($\rep,X_0)$-compatible element then $\margulis(\lin g,\co(\lin g))$ only depends on the class $[\co]\in H^1_\rep(\G,V)$, so we consider the map $\margulis:\G^\rep\times H^1_\rep(\G,V)\to\trivia$ defined by $$\margulis_\co(\lin g)=\margulis^\lin g(\co):=\margulis(\lin g,\co(\lin g)),$$ where we have denoted by $\G^{\rep}=\{(\rep,X_0)\textrm{-compatible elements of }\G\}$. By definition, the map is linear on the second variable and $\margulis_\co$ identically vanishes when $\co$ is a co-boundary. So from Corollary \ref{interiorAL} we conclude the following:

\begin{cor}\label{conococyclos}Let $\G\subset\sf G$ be a Zariski-dense sub-semi-group. Let $\{\rep_i:\sf G\to \SL(V_i)\}_{i\in I}$ be a disjoined collection and for each $i$ let $\co_i:\G\to V_i$ be a cocycle over $\rep_i$. Let $\rep=\bigoplus_i\rep_i$, and assume there exists, for each $i$, a $(\rep,X_0)$-compatible $\lin g_i\in\G$ such that $\margulis(\lin g_i,\co_i(\lin g_i))\neq0$. Denote by $V=\bigoplus_iV_i$ and by $\co=\sum_i\co_i:\G\to V$. Then the affine limit cone $\affinelim_{\G_\co}$ has non-empty interior.
\end{cor}

\begin{proof} By Remark \ref{Invariancia},  $[\co]_i\neq0$ as $\margulis(\lin g_i,\co_i(\lin g_i))\neq0.$ Corollary \ref{souravzariski} then implies that $\G_\co$ is Zariski-dense in $\sf G\ltimes_\rep V$ and the statement is reduced to Corollary \ref{rojo}.\end{proof}

\begin{obs}\label{nomade1} Under the assumptions of Corollary \ref{conococyclos} for $\co$, if $\G$ is a group and $\longest$ acts trivially on $\trivia$ then $\affinelim_{\G_\co}=\trivia.$ Indeed, since $\margulis(\lin g,\co(\lin g))=-\longest\margulis\big((\lin g,\co(\lin g))^{-1}\big)$ for every $\lin g,$ convexity gives $0\in\inte\affinelim_{\G_\co}.$ \end{obs}

\section{Compatible and $\vt$-Anosov linear part, normalized Margulis spectra} \label{normarde}


Consider $\rho\in\Anosov_\vt(\gh,\sf G)$ and a disjoined $\rep:\sf G\to\SL(V).$ A cocycle $\co\in H^1_{\rep\rho}(\gh,V)$ induces a $\trivia$-valued translation cocycle $c:\gh\times\bord^2\gh\to\trivia$ as in S. \cite{dichotomy} (see also Ledrappier \cite{ledrappier}), defined by $$c\big(\g,(x,y)\big)=\pi^\trivia\Big(\rep\big(\psi_{(x,y)}^{-1}\big)\pi^{X_\rep\cap Y_\rep,x_\rep\oplus y_\rep}\big(\co(\g^{-1})\big)\Big).$$ Indeed we have to check that for every pair $\g,h\in\gh$ one has $c\big(\g h,(x,y)\big)=c\big(h,(x,y)\big)+c\big(\g,h(x,y)\big),$ which follows from a straightforward computation. By S. \cite[Proposition 3.1.1]{dichotomy} there exists a H\"older-continuous $\ledrappier_\co:\sf U\gh\to\trivia$ such that for every hyperbolic $\g\in\gh$  $$\ell_{[\g]}(\ledrappier_\co)=\margulis(\g,\co(\g)).$$ Similar constructions are considered in Goldman-Labourie-Margulis \cite{GLM}. Thus, for every $\length\in\inte(\Bcone_{\vt,\rho})^*$, the set of \emph{normalized Margulis spectra} is convex and compact: $$\M^\length\big(\coclase\big)=\overline{\left\{\frac{\margulis\big(\g,\co(\g)\big)}{\length^\g(\rho)}:\g\in\gh_{\mathrm{h}}\right\}}.$$


\begin{obs}\label{nomade} Under the assumptions of Remark \ref{nomade1} for $\co$ (replacing $\G$ with $\rho(\gh)$), if $\longest$ acts trivially on $\trivia$ then convexity gives $\inte\M^\psi(\coclase)\neq\vacio$.
\end{obs}

Similar versions of the following, for split $\sf G,$ can also be found in Ghosh \cite{ghoshJordan}.


\begin{prop}[Kassel-Smilga \cite{KS}]\label{proper} Let $\rep:\sf G\to\SL(V)$ be irreducible and such that $\poids_\rep^\longest=\{0\}$, and consider $\rho\in\Anosov_\vt(\gh,\sf G)$ so that $\Bcone_{\rho}\subset\encono_\rep.$ If $0\in\inte\M^\length(\coclase)$ then the action of $\morfi_\co$ on $V$ is not proper. However, if $\length\in\inte(\a^+)^*$ and  $0\notin\M^\length(\coclase)$ then the corresponding action is proper.
\end{prop}

%
%

In particular, if $0\notin\M^{\psi}(\coclase)$ then these Margulis space-times lie on the context of S. \cite[\S\,3.4,\,3.5,\,3.6]{dichotomy} and several results there apply directly.

\part{The cone of Jordan variations, normalizations, pressure}

A faithful morphism $\rho:\G\hookrightarrow\sf G$ is fixed, so we identify $\G$ with $\rho(\G)$ and say that $\g\in\G$ is loxodromic if $\rho(\g)$ is. We also fix an integrable vector $v\in\sf T_\rho\caracG,$ which defines a cocycle over $\Ad\rho,$ $\co_{v}:\G\to\ge$ given by $$\g\mapsto\co_{v}(\g):=\deriva t0\rho_t(\g)\rho(\g)^{-1}.$$ 

\noindent
It is a co-boundary iff there exists $(s_t)\in\sf G$ with $s_0=\id$ such that $\forall\g\in\G$ the curves $\rho_t(\g)$ and $s_t\rho(\g)s_t^{-1}$ have the same derivative.  Equivalently, the curve $\rho_t$ is tangent at $0$ to the conjugacy class of $\rho$, which is also equivalent to the fact that the curve of characters $\rho_t\in\caracG$ has zero derivative.

For $\g\in\G$ we let $\jordan^\g:\caracG\to\a$ be the map $\jordan^\g(\eta)=\jordan(\eta(\g))$ and $\mathrm d\jordan^\g$ its differential. If $\varphi\in\a^*$ we let $\varphi^\g:\caracG\to\R$ be the composition $$\varphi^\g=\varphi\circ\jordan^\g: \rho\mapsto\varphi\big(\jordan\big(\rho(\g)\big)\big).$$

We introduce two concepts which are the main object of this part.

\begin{defi}\label{defsBase} \item \begin{itemize}\item[-]The \emph{cone of Jordan variations of $v$} is the cone associated to variations of Jordan projections:$$\jordanlim_{v}:=\overline{\Big\{\R_+\cdot\varjor{\g}{v}:\g\in\G\textrm{ loxodromic}\Big\}}\subset\a.$$

\item[-] Let $\length\in\conodual{\rho}$, then the \emph{set of $\length$-normalized variations} is $$\VV{\length}_{v}=\overline{\Big\{\frac{\varjor\g v}{\length^\g(\rho)}:\g\in\G\Big\}}\subset\a.$$ \end{itemize}
\end{defi}

Let $\ge=\bigoplus_i \ge_i$ be the decomposition of $\ge$ on simple ideals and assume we've chosen the Cartan subspaces $\a_i$ of $\ge_i$ so that $\a=\bigoplus_i \a_i$. Let $p_i:\ge\to\ge_i$ be the associated projections, choices have been made so that $p(\a)=\a_i$. Assume also the Weyl chambers $\a_i^+$ where chosen so that $\a^+=\bigoplus_i \a_i^+$.

The vector $v$ has \emph{full variation} if for every $i\in I$ the cocycle $p_i(\co_{v})$ is not a coboundary, and has \emph{full loxodromic variation} if $p_i(\jordanlim_{v})\neq\{0\}$ for every $i$.

\section{Variation of eigenvalues and some consequences of Part \ref{affineactions}}

We now apply Part \ref{affineactions} to the adjoint representation $\rep=\Ad:\sf G\to\SL(\ge).$ The set of weights is $\root$ and no root is $\longest$-invariant so the ideally neutral space $$\Vlongest=\ge^0=\m\oplus\a.$$ Writing $\m=[\m,\m]\oplus\centro(\m)$ one has that $\sf M$ preserves each factor $\ge^0=[\m,\m]\oplus\centro(\m)\oplus\a$ and acts trivially on $\centro(\m)\oplus\a$ so the neutralizing space is $$\trivia=\centro(\m)\oplus\a.$$ 

Moreover, picking any $X_0\in\a^+\cap\Fix (\ii)$ one readily sees that the $(\Ad,X_0)$-compatible cone is nothing but the whole Weyl chamber $\a^+.$ The Margulis invariant is thus well defined for any $\adg\in\sf G\ltimes\ge$ with loxodromic $\lin \adg\in \sf G$ and one has $$\margulis(\adg)\in\centro(\m)\oplus\a.$$

Finally, observe that $\Ad$ is a disjoined representation (recall Definition \ref{disjoined}).

\subsection{The variation of the Kostant-Lyapunov-Jordan projection}

Let $(g_t)_{t\in(-\eps,\eps)}\subset\sf G$ be a differentiable curve with loxodromic $g=g_0.$ We denote by $\vec g\in\sf T_{g_0}\sf G$ its derivative. Consider $\adg\in\sf G\ltimes\ge$ with linear part $g$ and translation vector $$d_gL_{g^{-1}}(\vec g)=\deriva t0g_tg^{-1}\in\ge.$$ Then one has the following.

\begin{prop}\label{margulisderivada}The $\a$-factor of $\margulis(\adg)$ is ${\displaystyle\deriva t0\jordan\big(g_t).}$ 
\end{prop}

The proof requires the following Lemma.

\begin{lemma}\label{curva} Consider a differentiable curve $(s_t)_{(-\eps,\eps)}\subset \sf G$ with $s_0=\id$, let $h_t=s_t^{-1}g_ts_t$ and $\sf h\in\sf G\ltimes\ge$ be defined as above, then $\margulis(\sf h)=\margulis(\sf g).$
\end{lemma}

\begin{proof} Indeed, explicit computation yields that the translation part of $\sf h$ is $$\deriva t0h_th^{-1}=\deriva t0g_tg^{-1}+\Ad(g_0)(\vec s)-\vec s,$$ so the lemma follows by Remark \ref{inv2}.
\end{proof}

\begin{proof}[Proof of Proposition \ref{margulisderivada}] Let $x_t,y_t\in\EuScript F_\simple$ be the repelling and attracting flags of $g_t$, for small $t.$ Since $\margulis(\adg)$ is invariant under conjugation, we can assume that $x_0=[\wk{\sf P}^\simple]$ and $y_0=[\sf P^\simple],$ so that $g=g_0=ma $ for some $m\in\sf M$ and $a\in\sf A.$ Consider also a differentiable curve $s_t\in \sf G$ with $s_0=\id,$ that sends the pair $(x_t,y_t)$ to $([\wk{\sf P}^\simple],[{\sf P}^\simple]),$ then by Lemma \ref{curva} one has $\margulis(\sf h)=\margulis(\sf g),$ for $h_t:=s_t^{-1}g_ts_t.$ 

We compute now $\margulis(\sf h).$ Since $h_t$ fixes $([\wk{\sf P}^\simple],[{\sf P}^\simple])$ we have $h_t=m_ta_t$ for $m_t\in\sf M$ and $a_t\in\sf A$, and $\jordan(g_t)=\jordan(a_t).$ The derivative of $h_th^{-1}$ is, since $\sf M$ and $\sf A$ commute,  \begin{equation}\label{cta}\deriva t0h_th^{-1}  =\deriva t0 m_ta_ta^{-1}m^{-1}= \vec mm^{-1}+\vec aa^{-1}.\end{equation}

\noindent
The Margulis invariant of $\sf h$ is then computed by considering the eigenspace decomposition of $\Ad(ma),$ which is nothing but the root space decomposition $$\ge=\ge^0\oplus\bigoplus_{\aa\in\root}\ge^\aa=[\m,\m]\oplus\centro(\m)\oplus\a\oplus\n\oplus\wk\n,$$ and projecting the vector (\ref{cta}) onto $\trivia=\centro(\m)\oplus\a$ parallel to this decomposition. The $\a$-factor of $\margulis(\adg)$ is then $\vec aa^{-1},$ as desired.\end{proof}

The same proof above actually gives the following:

\begin{cor} Let $\sf G_\C$ be a complex semi-simple algebraic group. Let $\a_\C$ be a Cartan subalgebra of $\sf G_\C$ and let $\mu:\sf G_\C\to \exp(\a_\C)$ be the eigenvalue projection. Let $(g_t)_{t\in(-\eps,\eps)}\subset\sf G_\C$ be a differentiable curve with loxodromic $g_0.$ Then $$\margulis(\adg)=\big(\mathrm{d}\mu(\vec g)\big)\mu(g)^{-1}\in\a_\C.$$
\end{cor}

\begin{proof} Considering $\sf G_\C$ as a real-algebraic group one has that $\m$ is abelian, so $\m=\centro(\m)$ and $\a_\C=\centro(\m)\oplus\a$. In the course of the proof of Proposition \ref{margulisderivada} one may observe that, the $\centro(\m)$-factor of $\margulis(\adg)$ is the projection of $\vec mm^{-1}\in\m=[\m,\m]\oplus\centro(\m)$ parallel to $[\m,\m],$ which readily gives the result.\end{proof}

\subsection{Every functional sees eigenvalue variations}  

We now prove Theorem \ref{A}.

\begin{cor}\label{phi-cotangent} If $\rho(\G)$ is Zariski-dense and $v$ has full loxodromic variation, then $\jordanlim_{v}$ is convex with non-empty interior. In particular, if $\rho$ is a regular point of $\caracG$ and $\varphi\in\a^*$ does not annihilate any of the $\a_i$'s, then $\spa\big\{\mathrm d\varphi^{\g}:\g\in\G\big\}=\sf T^*_\rho\caracG.$
\end{cor}

\begin{proof} Proposition \ref{margulisderivada} places the statement in the conditions of Corollary \ref{conococyclos} where $V_i=\ge_i$. The Corollary applies since the Adjoint representation is disjoined (Remark \ref{adjuntadisjunda}). Thus, the affine limit cone $$\affinelim_{\rho(\G)_{\co_{v}}}\subset\centro(\m)\oplus\a$$ is convex and has non-empty interior, whence it's projection onto the second factor also, giving the conclusion.\end{proof}

We introduce for convenience the following definition.

\begin{defi}If $\sf H$ is a reductive subgroup of $\sf G$ then an \emph{adjoint factor of $\sf H$} is a collection of irreducible factors of the representation $\ad_\ge|\h:\h\to\gl(\ge).$ An adjoint factor is \emph{disjoined} if the associated representation is.
\end{defi}

If $\rho(\G)$ has semi-simple Zariski closure $\sf H,$ then $H^1_{\Ad}(\G,\ge)$ splits as $$H^1_{\Ad\rho}(\G,\ge)=\bigoplus_{\factor{\sf H}\textrm{ irreducible adjoint factor}}H^1_{\Ad\rho}(\G,\factor{\sf H}).$$ 

We have the following refinement of Corollary \ref{phi-cotangent} whose proof is identical. 

\begin{cor}\label{factor-cotangente} Assume $\rho(\G)$ has semi-simple Zariski closure $\sf H$ and assume that $\Bcone_{\rho}\cap\inte\a^+\neq\vacio$. Let $\factor{\sf H}$ be a disjoined adjoint factor and $\rep=\Ad_\sf G(\sf H)|\factor{\sf H}.$ Assume $\co\in H^1_{\rep\rho} (\G,\factor{\sf H})\setminus\{0\}$ is integrable, then for every loxodromic $\g\in\G$ one has $$\varjor\g{\co}\in \factor{\sf H}\cap\a.$$ Moreover, if $[\co]$ projects non-trivially to the twisted cohomology associated to each irreducible factor of $\rep$, then  $\jordanlim_{v}$ is convex and has non-empty interior in $\factor{\sf H}\cap\a$. Consequently, for every $\varphi\in\a^*$ such that $\factor{\sf H}\cap\a\subsetneq\ker\varphi$ there exists $\g\in\G$ such that $$\mathrm{d}\varphi^{\g}(\co)\neq0.$$\end{cor}

\begin{proof} Since $\sf H$ contains a $\sf G$-loxodromic element the $0$-weight space of the representation $\rep:\sf H\to\GL(V_\sf H)$ verifies $$\a\cap V_\sf H\subset(V_\sf H)^0\subset (\m\oplus\a)\cap V_{\sf H}.$$ It follows that the $\a$-factor of the Margulis invariant of $(\eta(\g),\co(\g))$, as an element of $\sf H\ltimes_\rep V_\sf H $, coincides with $\varjor\g \co$, so non-empty interior of $\jordanlim_{\co}$ in $\factor{\sf H}\cap \a$ follows from Corollary \ref{conococyclos}.\end{proof}

\section{Zariski-density of elements with full variation}

In this section we establish the following.


\begin{prop}\label{fullZariskidenso} Assume $v\in\sf T_\rho\caracG$ has full loxodromic variation and Zariski-dense base-point. Fix a finite collection of hyperplanes $\EuScript P$ of $\a$. Then the set $$\G_\EuScript P=\{\textrm{loxodromic $g\in\G$ with}\ \varjor\g v\notin\bigcup_{U\in\EuScript P} U\}$$ is Zariski-dense in $\sf G$. Moreover, the set $\{\jordan(g):g\in\G_\EuScript P\}$ intersects every open cone contained in $\Bcone_\rho$.\end{prop}

We place ourselves in the assumptions of the Proposition, whose proof follows the same lines as Benoist \cite{limite} (see \cite[Theorem 6.36]{Benoist-QuintLibro}) for loxodromic elements.

\begin{lemma}\label{abc} Let $\K$ be a field and consider $w,g,h\in\GL(d,\K)$, then for every $N\in\N$ the Zariski closure of $\{wg^nh^n:n\in\lb N,\infty)\}$ contains the product $wgh$. Analogously, the Zariski closure of $\{g^nh^nw:n\in\lb N,\infty)\}$ contains $ghw$.
\end{lemma}

\begin{proof}
Let $I=\{p\in\R[x_{ij}]:p(wg^{n}h^{n})=0\ \forall n\geq N\}$ be the associated ideal. We must show that for every $p\in I$ it holds $p(wgh)=0$.

Consider the map $T:\R[x_{ij}]\to\R[x_{ij}]$ defined, for $X=(x_{ij})_{ij}$, by $$Tp(X)=p(wg^Nw^{-1}Xh^N).$$ It is an isomorphism that preserves $I$. Moreover, the finite-dimensional vector space $$I^m=\{p\in I \textrm{ of total degree }\leq m\}$$ verifies $T(I^m)\subset I^m$ which yields, since $T$ is invertible, that $T(I^m)=I^m$, and thus $T(I)=I$. Consider then $q\in I$ and let $p\in I$ be such that $Tp=q$, then $q(wgh)=p(wg^{N+1}h^{N+1})=0,$ as desired.\end{proof}

Recall the projections $p_i:\a\to\a_i$ and let, for $U\in\EuScript P,$ $p_U\in\a^*$ be such that $\ker p_U=U.$ We treat indistinctively the $p_i$'s and the $p_U$'s and say throughout this section that $g\in\G$ has \emph{full variation} if forall $j\in I\cup\EuScript P$ one has $p_j(\varjor{g}{v})\neq0.$

\begin{lemma}\label{potencia}Let $g,h\in\G$ be loxodromic and transverse, if $g$ has full variation and $k\in\N$ is such that $$\forall i\ p_i\big(k\varjor gv+\varjor h v\big)\neq0,$$ then for all large enough $n$ (depending on $k$) the element $(g^k)^nh^n$ has full variation.
\end{lemma}

\begin{proof} By Theorem \ref{CR} we have $$\frac{\jordan\big((g^k)^nh^n\big)}n-\big(\jordan(g^k)+\jordan(h)\big)\xrightarrow{n\to\infty}0$$ and the convergence is uniform about $g^k$ and $h$. Since we're considering an analytic variation $u\mapsto\rho_u$ and $\jordan$ is an analytic function when restricted to loxodromic elements of $\sf G$, we can differentiate both sides in the convergence to obtain
$$\frac{\varjor{{g^k}^nh^n}v}n-k\varjor g v-\varjor h v\xrightarrow{n\to\infty}0.$$
By assumption we have, for every $i$, $p_i(k\varjor g v+\varjor h v)\neq0$ so the above convergence implies the lemma.\end{proof}

\begin{lemma}\label{ealverre}Let $\g\in\G$ be loxodromic, then there exists $ g\in\G_\EuScript P$ transverse to $\g$.
\end{lemma}

\begin{proof} Consider a loxodromic $g\in\G$ with full variation, the existence of such $g$ is guaranteed by Theorem \ref{A}. By means of Zariski density of $\rho(\G)$, we can find a loxodromic $h$ that is transverse to both $\g$ and $g$. By Lemma \ref{potencia} elements of the form $g^nh^m$, for arbitrary large $n$ and $m$, have full variation, and analogously for $h^mg^n$. Again for large enough $m,n$, the pairs $h^mg^n$ and $g^nh^m$ are transverse, so we can find, by Lemma \ref{potencia} again, large enough $k,l$ so that $$(h^mg^n)^k(g^nh^m)^l $$ has full variation. Moreover, the attracting flag of the latter element is arbitrarily close to $h^+$ and the repelling flag is close to $h^-$, thus $(h^mg^n)^k(g^nh^m)^l$ has full variation and is transverse to $\g$.\end{proof}

\begin{lemma}\label{semigrupo} Let $g,h$ be loxodromic and transverse, assume moreover that $g\in\G_\EuScript P$ then, the products $hg$ and $gh$ belong to the Zariski closure $\overline{\G_\EuScript P}{}^\mathrm Z$. Moreover, the semi-group spanned by $\{gh,g\}$ is also contained in $\overline{\G_\EuScript P}{}^{\mathrm Z}$.
\end{lemma}

\begin{proof} Since $g$ has full variation, there exist $K$ such that for all $k\geq K$ one has $$\forall i\ p_i\big(\varjor{g^k}v+\varjor h v\big)=p_i\big(k\varjor gv+\varjor h v\big)\neq0.$$ For every such $k$,
Lemma \ref{potencia} implies that for all large enough $n$ one has $(g^k)^nh^n\in\G_\EuScript P$. Lemma \ref{abc} yields $g^kh\in\overline{\G_\EuScript P}{}^{\mathrm Z}$ for all $k\geq K$ and thus $gh\in\overline{\G_\EuScript P}{}^{\mathrm Z}$.

We now show that the semi-group spanned by $\{gh,g\}$ is also contained in $\overline{\G_\EuScript P}{}^{\mathrm Z}$. To this end, consider an arbitrary word $w$ on the letters $gh$ and $g$, and assume by induction that $w\in\overline{\G_\EuScript P}{}^{\mathrm Z}$. We will show that the words $$ghw,wgh,gw,wg$$ are all contained in $\overline{\G_\EuScript P}{}^{\mathrm Z}$. Since $g$ and $h$ are transverse and we are only considering positive powers in $w$, the word $w$ is transverse to $g$, and to any element of the form $g^mh^n$ for positive $n,m\in\N$.

Since $g\in\G_\EuScript P$, then the first statement of this lemma implies that $wg\in\overline{\G_\EuScript P}{}^{\mathrm Z}$ as desired. Moreover, again since $g\in\G_\EuScript P$, the paragraph above establishes that, for all $k\geq K$ and all $n\geq N(k)$, one has $(g^k)^nh^n\in\G_\EuScript P$. Thus, since $w$ is transverse to $(g^k)^nh^n$, the first item of this Lemma gives that $\forall k\geq K,\ \forall n\geq N(k)$ $$ w (g^k)^nh^n\in\overline{\G_\EuScript P}{}^{\mathrm Z},$$ which implies, by applying twice Lemma \ref{abc} that $wgh\in\overline{\G_\EuScript P}{}^{\mathrm Z}$ as desired. The other inclusions are analogous.\end{proof}

\begin{proof}[Proof of Proposition \ref{fullZariskidenso}] We will show that the Zariski closure of $\G_\EuScript P$ contains all loxodromic elements of $\rho(\G)$, which are in turn Zariski-dense by Benoist \cite{limite}.

Consider then $\g\in\G$ loxodromic. By Lemma \ref{ealverre} there exists $g\in\G_\EuScript P$ transverse to $\g$. Lemma \ref{semigrupo} establishes that the semigroup spanned by $\{g\g,g\}$ is contained in $\overline{\G_\EuScript P}{}^{\mathrm Z}$, but the Zariski closure of a semi-group is a group (c.f. \cite[Lemma 6.15]{Benoist-QuintLibro}) thus the group spanned by $\{g\g,g\}$ is contained in $\overline{\G_\EuScript P}{}^{\mathrm Z}$ and in particular so does $\g=g^{-1}(g\g)$, as desired.

We finally establish the last statement in the Proposition. Consider an open cone $\scr C\subset\Bcone_\rho,$ $\g\in\G$ loxodromic with $\jordan(\g)\in\scr C$ and $g\in\G_\EuScript P$  transverse to $\g$. Consider then $t\in\R_+\setminus\Q$ so that $t\jordan(\g)+\jordan(g)\in\scr C$, by the abundance of such $t$'s we may further assume that for all $i$ $p_i\big(t\varjor \g v+\varjor g v\big)\neq0$. Consider then a sequence of rationals in lowest terms $m_n/q_n\to t$, since $t\in\R\setminus \Q$ we have $\min\{m_n,q_n\}\to\infty$. Lemma \ref{pares} implies then that $$\lim_{n\to\infty} \frac{\jordan(\g^{m_n}g^{q_n})}{q_n}=t\jordan(\g)+\jordan(g)\in\scr C,$$ so $\jordan(\g^{m_n}g^{q_n})\in\scr C$ for big enough $n$. Moreover, again by analyticity of our curve and uniform convergence on the above limit, we get by differentiating both sided of the limit that $\g^{m_n}g^{q_n}\in\G_\EuScript P$ for all large enough $n$.\end{proof}

\subsection{Convexity of normalized variations}

\begin{prop}\label{normconvex} Let $v\in\sf T_\rho\caracG$ have full loxodromic variation and Zariski-dense base-point then, the set of normalized variations is convex.
\end{prop}

\begin{proof} Consider $\g,h\in\full$. We first assume that $\g$ and $h$ are transversally loxodromic. Using the argument from the proof of Proposition \ref{fullZariskidenso}, we obtain that for every irrational $t\in\R_+$ one has $\jordan(\g^{m_n}h^{q_n})/q_n\to t\jordan(\g)+\jordan(h)$ which in turn gives:\begin{alignat*}{2}\frac{\length^{\g^{m_n}h^{q_n}}(\rho)}{q_n}&\xrightarrow[n\to\infty]{} t\length^\g(\rho)+\length^h(\rho),\\ \frac{\varjor{\g^{m_n}h^{q_n}}v}{q_n}&\xrightarrow[n\to\infty]{} t\varjor\g v+\varjor hv.\end{alignat*} Combining both equations and since $\VV \length_v$ is closed, we obtain that for every $t\in\R_+$ $$\frac{t\varjor\g v+\varjor hv}{ t\length^\g(\rho)+\length^h(\rho)}\in\VV\length_v.$$ So letting now $t=\length^h(\rho)/\length^\g(\rho)$ we obtain ${\displaystyle\frac12\left(\frac{\varjor\g v}{\length^\g(\rho)}+\frac{\varjor hv}{\length^h(\rho)}\right)\in\VV\length_v},$ as desired.

If $\g$ and $h$ are not transversally loxodromic then by Lemma \ref{modif} we can replace $h$ by some element of the form $fh^nq$ and apply the above.\end{proof}

\section{Theorem \ref{teoB}: Base-point independence}\label{SecTB}

We introduce, for convenience, the following definitions.

\begin{defi}\item \begin{itemize}\item[-] The \emph{support} of $\varphi\in\a^*$ is $\supp\varphi=\big\{\sroot\in\simple:\<\varphi,\sroot\>\neq0\big\}.$ Equivalently, upon writing $\varphi=\sum_{\sroot\in\simple}\varphi_\sroot\peso_\sroot$, one has $\sroot\in\supp\varphi$ if and only if $\varphi_\sroot\neq0$. In particular,  $\varphi\in(\a_\t)^*$ if and only if $\supp\varphi\subset\t$.

\item[-]If $g\in\sf G$ we say that \emph{$\aa$ strictly minimizes $g$ among $\supp\varphi$} if $\aa\in\supp\varphi$ and \begin{equation}\label{min}\aa(\jordan(g))<\sroot(\jordan(g))\ \forall\sroot\in\supp\varphi-\{\aa\}.\end{equation}\end{itemize}\end{defi}

 The main result of this section is the following.

\begin{thm}\label{LedIndepA} Let $\rho:\G\hookrightarrow\sf G$ be a Zariski-dense sub-semigroup and $v\in\sf T_\rho\caracG$ have full loxodromic variation. Consider $\varphi\in\a^*$ and assume there exist $\aa\in\supp\varphi$ with $\dim\ge_\aa=1$ and $g\in\G$ such that $\aa$ strictly minimizes $g$ among $\supp\varphi$. Then given $\length\in\a^*$ there exists a loxodromic $\g\in\G$ such that $\mathrm d\varphi^\g(v)-\length^\g(\rho)\notin\Z$. \end{thm}

\begin{cor}\label{r2}
In the assumptions of Theorem \ref{LedIndepA}, the additive group spanned by the pairs $\{\big(\mathrm d\varphi^\g(v),\jordan^\g(\rho)\big):\g\in\G\textrm{ loxodromic}\}$ is dense in $\R\times\a$.
\end{cor}

\begin{proof} Otherwise, there exist $(a,\psi)\in\R\times\a^*$ s.t. for all loxodromic $\g\in\G$ it holds  $a\varphi(\varjor\g v)+\length(\jordan^\g(\rho))\in\Z$. If $a\neq0$ then $a\varphi$  verifies the assumptions of Theorem \ref{LedIndepA} giving a contradiction. If $a=0$ this is contained in Benoist's Theorem \ref{interiorsi}.
\end{proof}

\begin{ex} For $g\in\PSL(2,\C)$ denote by $|\g|$ its translation length on the hyperbolic $3$-space $\H^3$. Let $\cal T(S)$ be the Teichm\"uller space of $S$ as above. Consider then a Zariski-dense quasi-Fuchsian representation $\eta:\pi_1S\to\PSL(2,\C)$ and a non-zero $v=\vec\rho\in\sf T_\rho\cal T(S)$. Assume there exists $g\in\pi_1S$ with $|\rho g|<|\eta g|$. Then by Corollary \ref{r2} the group spanned by the pairs $$\Big\{\big((\partial/\partial t)_{t=0}|\rho_t\g|, |\eta \g|\big):\g\in\pi_1S\Big\}$$ is dense in $\R^2$.\end{ex}

Recall from \S\,\ref{s.Levi} that we have a projection $\pi_\t:\a\to\a_\t.$

\begin{cor}\label{double-density} Let $\t\subset\simple$ be such that $\dim\ge_\sroot=1$ for all $\sroot\in\t$. Let $\rho:\G\hookrightarrow\sf G$ be a Zariski-dense sub-semi-group  and consider an integrable, full loxodromic variation $v\in\sf T_\rho\caracG$.  Then the group spanned by $$\big\{\big(\varjort\g v,\jordan^\g(\rho)\big):\g\in\G\textrm{ loxodromic}\big\}$$ is dense in $\a_\t\times\a$.  Whence, if $\length\in\conodual\rho$ the set $\VV\length_{\t,v}$ has non-empty interior.\end{cor}

\begin{proof}Otherwise there exist $\varphi\in(\a_\t)^*$ and $\length\in\a^*$ such that for all loxodromic $\g\in\G$ one has $\mathrm d\varphi^\g(v)-\length^\g(\rho)\in\Z$. Since $\rho(\G)$ is Zariski-dense its limit cone has non-empty interior, there exists then $g\in\G$ such that values $\sroot^g(\rho)$ for $\sroot\in\t$ are pairwise distinct. Since $\supp\varphi\subset\t$ and all roots in $\t$ have $1$-dimensional root space, there exists $\aa\in\supp\varphi$ with 1-dimensional root-space that strictly minimizes $g$ among $\supp\varphi$. This contradicts Theorem \ref{LedIndepA}. The last statement now follows since by Proposition \ref{normconvex} $\VV\length_{\t,v}$ is convex. \end{proof}

\subsection{Strongly transversally proximal elements}
We use freely notation from \S\,\ref{proxBenoist} and let, for proximal $g$,  $\mu_1(g)$ be the eigenvalue with $|\mu_1(g)|=\exp(\jordan_1(g)),$ $\mu_2(g)=\exp(\jordan_2(g))$ and $$\pi_g(w)=\beta_g(w)v_g.$$ 


Recall that an element of $\End(V)$ is \emph{semi-simple} if it is diagonalizable over $\C$. Finally, recall Eq. \eqref{B1} defining the multiplicative cross ratio $\crossm_1$.

\begin{lemma}[{Benoist-Quint \cite[Lemma 7.15]{Benoist-QuintLibro}}]\label{lemmacm} Let $g,h\in\End(V)$ be transversally proximal, then for $m,n$ big enough $g^nh^m$ is proximal and  \begin{equation}\label{cm}c_n(g,h):=\lim_{m\to\infty} \traza(\pi_g\pi_{g^nh^m})=\frac{\traza(\pi_gg^n\pi_h)}{\traza(g^n\pi_h)}=\crossm_1(g^+,g^-,g^nh^+,h^-).\end{equation} If they are strongly transversally proximal and $g$ is semi-simple, then the sequence $$\Big(\frac{\vp_1(g)}{\vp_2(g)}\Big)^{n}\log |c_{n}(g,h)|$$ is bounded. Moreover, let $g_e$ be the elliptic component of $g$ in Jordan's decomposition and $n_k$ a sequence such that $g_e^{n_k}|V_2(g)\xrightarrow[k\to\infty]{}\id|V_2(g).$  Then, \begin{equation}\label{cmraiz}\lim_{k\to\infty}\Big(\frac{\vp_1(g)}{\vp_2(g)}\Big)^{n_k}\log |c_{n_k}(g,h)|=\frac{-\beta_h(\tau_gv_h)}{\crossm_1(g^+,g^-,h^+,h^-)}\neq0,\end{equation} and the convergence is moreover uniform on $h$.
\end{lemma}

\begin{proof} We focus on the second statement which is slightly different from what is found in \cite{Benoist-QuintLibro}. Using Equation \eqref{cm} we compute

$$\log|c_n(g,h)|\underset{n\to\infty} {\sim}c_n(g,h)-1=\frac{\traza\big((\pi_g-1)g^n\pi_h\big)}{\traza(g^n\pi_h)}.$$ The denominator is easily controlled, indeed\begin{equation}\label{crprox}\frac{\traza(g^n\pi_h)}{\vp_1(g)^n}=\frac{\beta_h(g^nv_h)}{\vp_1(g)^n}\xrightarrow[n\to\infty]{}\beta_h\big(\pi_g(v_h)\big)=\crossm_1(g^+,g^-,h^+,h^-).
\end{equation}

We now study the numerator. Observe that $\traza\big((1-\pi_g)g^n\pi_h\big)=\beta_h(g^n\tau_g v_h)+o(\vp_2(g)^n).$ Since $g$ is semi-simple,  $V_2(g)$ decomposes as $\bigoplus_{i} W_i$ where each $W_i$ is $g$-invariant and $g|W_i=\vp_2(g)K_i,$ where $K_i:W_i\to W_i$ lies in an abelian compact group. Thus, the sequence $$\frac{\beta_h(g^n\tau_g v_h)+o(\vp_2(g)^n)}{\mu_2(g)^n}$$ is bounded. Moreover, considering the sequence $n_k\to\infty$ as in the statement, for all $i$ one has $K_i^{n_k}\to\id.$ Since $g$ is semi-simple we deduce that $$\frac{g^{n_k}|V_2(g)}{\vp_2(g)^{n_k}}\xrightarrow[k\to\infty]{}\id,$$ so one concludes  $$\frac{\traza\big((1-\pi_g)g^{n_k}\pi_h\big)}{\vp_2(g)^{n_k}}=\frac{\beta_h(g^{n_k}\tau_g v_h)}{\vp_2(g)^{n_k}}+\frac{o(\vp_2(g)^{n_k})}{\vp_2(g)^{n_k}}\xrightarrow[k\to\infty]{}\beta_h(\tau_gv_h),$$ as desired. \end{proof}

Suppose now that $g,h\in\sf G$ are $\t$-proximal, then we say they are \emph{strongly transversally $\t$-proximal} if for every $\sroot\in\t$ the maps $\rep_\sroot g$ and $\rep_\sroot h$ are strongly transversally proximal. For such a pair and $n\in\N$ we define the vector \begin{equation}\label{definu}\nu^\t_n(g,h)\in\a_\t\ \textrm{ so that }\ \forall\sroot\in\t,\ \peso_\sroot\big(\nu^\t_n(g,h)\big):=\log\big|c_n\big(\rep_\sroot g,\rep_\sroot h)\big|.\end{equation}

\begin{prop}\label{cmgeneral} Let $g\in\sf G$ be loxodromic and consider $h\in \sf G$ so that $g$ and $h$ are strongly transversally $\t$-proximal. Consider $\varphi\in(\a_\t)^*$ and let $\aa\in\supp\varphi$ be such that \begin{equation}\label{minprop}\aa(\jordan(g))=\min\{\sroot(\jordan(g)):\sroot\in\supp\varphi\}.\end{equation} If $\aa$ has multiplicity $1$ and is the only root in $\supp\varphi$ realizing the above minimum then there exists $\kappa^\varphi(g,h)\in\R$ such that
\begin{equation}\label{convkappa}\varphi\big(\nu^\t_{2n}(g,h)\big)e^{2n\aa(\jordan(g))}\xrightarrow[n\to\infty]{}\kappa^\varphi(g,h)\neq0.\end{equation}
In the latter case, the map $(\g,\eta)\mapsto\kappa^\varphi(\g,\eta)$ is analytic on both variables on a neighborhood of $(g,h)\in\sf G^2$ and the above convergence is uniform on a neighborhood of $g$ and $h$.

\end{prop}

\begin{proof}By definition of $\nu^\t_n(g,h)$, upon writing $\varphi=\sum_{\sroot\in\t}\varphi_\sroot\peso_{\sroot}$ one has \begin{alignat*}{2}\varphi\big(\nu^\t_{n}(g,h)\big) & = \sum_{\sroot\in\t}\varphi_\sroot\peso_\sroot\big(\nu^\t_{n}(g,h)\big) \\ & =\sum_{\sroot\in\t}\varphi_\sroot \log\big|c_{n}(\rep_\sroot g,\rep_\sroot h)\big|. \end{alignat*}

Considering, for each $\sroot\in\t$, the fundamental representation associated to $\peso_\sroot$ we see by Lemma  \ref{lemmacm} that $e^{n\sroot(\jordan(g))}\peso_\sroot\big(\nu^\t_n(g,h)\big)$ is bounded, thus if $\aa$ is the only root realizing the minimum in Equation \eqref{minprop} we have, for every $\sroot\in\supp\varphi-\{\aa\}$ that $$e^{n\aa(\jordan(g))}\peso_\sroot\big(\nu^\t_n(g,h)\big)\xrightarrow[n\to\infty]{}0.$$

Moreover, since $\dim\ge_\aa=1$, $V_2(\rep_\aa(g))$ is $1$-dimensional. Thus $$\Big(\frac{\vp_1(\rep_\aa(g))}{\vp_2(\rep_\aa(g))}\Big)^2=e^{2\aa(\jordan(g))}$$ and combining with the last statement of Lemma \ref{lemmacm} we obtain

\begin{alignat}{2}\label{kafi}\varphi\big(\nu^\t_{2n}(g,h)\big)e^{2n\aa(\jordan(g))}\xrightarrow[n\to\infty]{}\varphi_\aa\frac{-\beta_{\rep_\aa h}(\tau_{\rep_\aa g}v_{\rep_\aa h})}{\peso_\aa(\cross_\t(g^+,g^-,h^+,h^-))}:=\kappa^\varphi(g,h).\end{alignat}

The convergence is uniform on a neighborhood of $h$ so we now treat uniform convergence in $g$. If $g'$ is close to $g$ then, by one-dimensionality of $\ge_\aa$, $g'$ also acts as a homothety with ratio $\vp_2(\rep_\aa g')$ on $V_2(\rep_\aa g')$ and thus uniform convergence follows. Analyticity of $\kappa^\varphi(\cdot,h)$ follows as, since $g$ is loxodromic, the map $\tau_{\rep_\aa g}$ varies analytically about $g$. \end{proof}

For a real-analytic curve $\big(t\mapsto(g_t,h_t)\big)_{t\in(-\eps,\eps)}$ with $g=g_0$ and $h=h_0$ strongly transversally $\t$-proximal, with $g$ loxodromic, we denote for every $ n\in\N$ and $\varphi\in(\a_\t)^*$ by $\nu_n^{(g,h)}(t):=\nu_n(g_t,h_t)$ and $\kappa^{\varphi,(g,h)}(t)=\kappa^\varphi(g_t,h_t).$ We also let $${\dot\nu}_n^{(g,h)}=\frac{\partial}{\partial t}\Big|_{t=0} \nu_{n}^{(g,h)}(t)\ \textrm{ and }\ {\dot\kappa}^{\varphi,(g,h)}=\frac{\partial}{\partial t}\Big|_{t=0}\kappa^{\varphi,(g, h)}(t).$$

\begin{cor}\label{parest}Consider a real-analytic curve $t\mapsto(g_t,h_t)$ for $t\in(-\eps,\eps)$ with $g=g_0$ and $h=h_0$ strongly transversally $\t$-proximal, and assume $g$ is loxodromic. Consider $\varphi\in(\a_\t)^*$ and assume there exists $\aa\in\supp\varphi$ with $\dim\ge_\aa=1$ and so that \begin{equation}\label{aamin}\aa(\jordan(g))<\sroot(\jordan(g))\ \forall\sroot\in\supp\varphi-\{\aa\}.\end{equation} Then, $$\varphi\big({\dot\nu}_{2n}^{(g,h)}\big)e^{2n\aa(\jordan(g))}+ \varphi\big({\nu}_{2n}^{(g,h)}\big)2n \Big(\frac{\partial}{\partial t}\Big|_{t=0}\aa\big(\jordan(g_t)\big)\Big)e^{2n\aa(\jordan(g))}\xrightarrow[n\to\infty]{} {\dot\kappa}^{\varphi,(g,h)}.$$
\end{cor}

\begin{proof}

Assumptions are made so that Proposition \ref{cmgeneral} applies. By definition, $\nu_n^{(g,h)}(t)$ and $\kappa^{\varphi,(g,h)}(t)$ are real-analytic, and since the convergence in Eq. \eqref{convkappa} is uniform, we can intertwine limit and derivative to obtain the desired result.\end{proof}

\subsection{Proof of Theorem \ref{LedIndepA}}

We place ourselves under the assumptions of Theorem \ref{LedIndepA} and begin with the following lemma that does not assume Zariski-density of the base point, it will be also needed later on.

\begin{lemma}\label{intermedioModif} Consider an analytic curve $(\rho_t:\G\to\sf G)_{t\in(-\eps,\eps)}$ with speed $v$ and a loxodromic $g\in\G$. Consider $\varphi\in(\a_\t)^*$ and assume there exists $\aa\in\supp\varphi$ with $\dim \ge_\aa=1$ that strictly minimizes $g$ among $\supp\varphi$. Let $h\in\G$ be such that the pair $(g,h)$ is strongly transversally $\t$-proximal.
Consider $\length\in(\a_\t)^*$ and assume that the values $\sroot^g(\rho)$, for $\sroot\in\supp\length$, are all distinct. If for all loxodromic $\g\in\G$ it holds $\mathrm d\varphi^\g(v)-\length^\g(\rho)\in\Z$, then $$\mathrm d\aa^{g}(v)=0.$$ \end{lemma}

\begin{proof}By definition (Eq. \eqref{definu}), the vector $\nu_n( g,h)$ is a uniform double limit of sums of vectors of the form $\pm\jordan(\rho(g^n(g^kh^l)^m))$. Moreover, for every $n$, $\nu_n(g,h)$ is an analytic function on $g$ and $h$, which is also a uniform limit of analytic functions. Since the curve $\rho_t$ is analytic we can intertwine limit and derivative in the definition of $\nu_n$ so for every $n\in\N$ it holds, as $g$ and $h$ are transversally $\t$-proximal, that \begin{equation}\label{varcmblu}m_n:=\mathrm d\varphi\big(\nu^{\t,(g, h)}_n\big)(v) -\psi\big({\nu}^{\t,(g, h)}_n(\rho)\big)\in\Z,\end{equation} where we have denoted by $\nu_n^{\t,(g, h)}(\rho_t)=\nu^\t_n(\rho_t g,\rho_t h).$ Since $\dim\ge_\aa=1$ Corollary \ref{parest} states that $$\mathrm d\varphi\big(\nu^{\t,(g, h)}_{2n}\big)(v)e^{2n\aa^{g}(\rho)}+ \varphi\big({\nu}^{\t,(g, h)}_{2n}(\rho)\big)2n \big(\mathrm d\aa^{g}(v)\big)e^{2n\aa^{g}(\rho)}\xrightarrow[n\to\infty]{} \mathrm d\kappa^{\varphi,(g, h)}(v).$$

Pairing with Equation \eqref{varcmblu} gives \begin{equation}\label{piblu}\Big(\length\big({\nu}^{\t,(g, h)}_{2n}(\rho)\big)+m_{2n}\Big) e^{2n\aa^{g}(\rho)}+ \varphi\big({\nu}^{\t,(g, h)}_{2n}(\rho)\big)2n \big(\mathrm d\aa^{g}(v)\big)e^{2n\alpha^{g}(\rho)}\xrightarrow[n\to\infty]{} \mathrm d\kappa^{\varphi,(g, h)}(v).\end{equation}

Dividing by $n$ and considering the limit we obtain by Equation \eqref{convkappa}

\begin{equation}\label{piblu2}\Big(\length\big({\nu}^{\t,(g, h)}_{2n}(\rho)\big)+m_{2n}\Big) \frac{e^{2n\aa^{g}(\rho)}}n\xrightarrow[n\to\infty]{}  2\mathrm d\aa^{g}(v)\kappa^{\varphi,(g,h)}(\rho)\end{equation}

\noindent One has that $\nu_{2n}^{\t,(g,h)}(\rho)\to0$ and thus also does $\length(\nu_{2n}^{\t,(g,h)}(\rho)).$ Since $e^{2n\aa^g(\rho)}/n$ is divergent, we obtain that $m_{2n}=0$ for all big enough $n$, giving in turn that \begin{equation}\label{psiblu3}\length\big({\nu}^{\t,(g, h)}_{2n}(\rho)\big)\frac{e^{2n\aa^{g}(\rho)}}n\xrightarrow[n\to\infty]{}  2\mathrm d\aa^{g}(v)\kappa^{\varphi,(g,h)}(\rho).\end{equation}

Using the definition of $\nu^{\t,(g,h)}_n$ we obtain $$e^{2n\aa^g(\rho)}\length\big(\nu^\t_{2n}(g,h)\big)  =\sum_{\sroot\in\t}\length_\sroot e^{2n\big(\aa^g(\rho)-\sroot^g(\rho)\big)} e^{2n\sroot^g(\rho)}\log\big|c_{2n}(\rep_\sroot g,\rep_\sroot h)\big|.$$ 

If there exists $\sroot\in\supp\length$ so that $\aa^g(\rho)-\sroot^g(\rho)>0$ then we let $\sroot$ be the root that maximizes this value. Applying Lemma  \ref{lemmacm} to the representation $\rep_\sroot$ we obtain a subsequence $n_k$ such that $$e^{2n_k\sroot^g(\rho)}\log\big|c_{2n_k}(\rep_\sroot g,\rep_\sroot h)\big|\xrightarrow{k\to\infty}{}K\neq0.$$ Since the terms $e^{2n_k\delta^g(\rho)}\log\big|c_{2n_k}(\rep_\delta g,\rep_\delta h)\big|$ are bounded for every $\delta\in\supp\length$, we deduce that
$e^{2n_k\aa^g(\rho)}\length\big(\nu^\t_{2n_k}(g,h)\big)$ is diverging to infinity at an exponential rate $\mu=\aa^g(\rho)-\sroot^g(\rho)>0$, which combined with Equation \eqref{psiblu3} gives $e^{2n_k\mu}/n_k$ is convergent as $k\to\infty$, a contradiction.

We conclude that $\forall$ $\sroot\in\supp\length$ one has $\aa^g(\rho)\leq\sroot^g(\rho)$ and applying Lemma \ref{lemmacm} we obtain, since $e^{2n\sroot^g(\rho)}\log\big|c_{2n}(\rep_\sroot g,\rep_\sroot h)\big|$ is bounded for every $\sroot$, that 
$e^{2n\aa^g(\rho)}\length\big(\nu^\t_{2n}(g,h)\big)$ is converging to a constant $C$, which is possibly zero. Thus $$\length\big({\nu}^{\t,(g, h)}_{2n}(\rho)\big)\frac{e^{2n\aa^{g}(\rho)}}n\xrightarrow[n\to\infty]{} 0,$$ giving, since $\kappa^\varphi(g,h)\neq0$ by Proposition \ref{cmgeneral}, that $\mathrm d\aa^g(v)=0$ as desired.\end{proof}

\begin{proof}[Proof of Theorem \ref{LedIndepA}] Let us assume by contradiction that for all $\g\in\G$ one has $$\mathrm d\varphi^\g(v)-\length^\g(\rho)\in\Z.$$ By hypothesis, there exists $\aa\in\supp\varphi$ with $\dim\ge_\aa=1$ and a loxodromic $\g$ verifying Eq. \eqref{min}. The Zariski-density assumption gives (c.f. Benoist-Quint \cite[Lemma 7.20]{Benoist-QuintLibro}), for each loxodromic $\g$, an $h\in\G$ such that $(\g,h)$ are strongly transversally loxodromic. Thus, we can apply Lemma \ref{intermedioModif} to every loxodromic $g\in\G$ such that the values $\sroot^g(\rho)$, for $\sroot\in\t$, are all distinct and verifies Equation \eqref{min} to obtain:

\begin{obs}\label{obstu} For every $g\in\G$ loxodromic  such that the values $\sroot^g(\rho)$, for $\sroot\in\t$, are all distinct and verifying Eq. \eqref{min} one has $\mathrm d\aa^g (v)=0.$
\end{obs}

Proposition \ref{fullZariskidenso} states we can choose $\g\in\G_{\mathrm{full}}$ verifying the assumptions of the above remark. Moreover, we use Proposition \ref{subsemi} by Benoist and more specifically Remark \ref{G'schottky} to choose a Zariski-dense sub-semi-group $\G'<\G$ that contains $\g^k$ for some large power $k$ and whose limit cone is a convex cone about $\R_+\jordan(\rho \g),$ chosen so that for all $ h\in\G'$ Equation \eqref{min} holds (for $h$ instead of $\g$) and such that the values $\sroot^h(\rho)$, for $\sroot\in\t$, are all distinct.

Since $\G'$ is chosen with $\g^k\in\G'$, the curve $\eta_t:=\rho_t|\G'$ has full loxodromic variation. However, Remark \ref{obstu} gives that $$\forall h\in\G'\textrm{ it holds }\mathrm d\aa^{h}(\vec{\eta})=0,$$ contradicting Theorem \ref{A}. This completes the proof.\end{proof}

\section{The case of $\t$-Anosov representations: cohomological independence and other consequences}
We assume throughout that $\gh$ is word-hyperbolic and $\rho\in\Anosov_\t(\gh,\sf G).$


\subsection{Full loxodromic variation for $\t$-Anosov representations}\label{fullerton}We prove:

\begin{cor}\label{corfullerton} Assume that $\t\cap\simple_i\neq\vacio$ for every simple factor $\ge_i$ of $\ge$. If $v$ has full variation then it has (full) loxodromic variation. 
\end{cor}

The proof is essentially contained in Bridgeman-Canary-Labourie-S. \cite[\S\,10]{pressure} so we only give the minor required modifications.

\begin{lemma}\label{presslox}Let $\rho:\gh\to\sf G$ be $\t$-Anosov and have Zariski-dense image, then $$\big\{\big(h^-,h^+,\g^-,\g^+\big):\textrm{$\rho(\g)$ and $\rho(h)$ are loxodromic}\big\}$$ is dense in $\bord^{(4)}\gh$.
\end{lemma}

\begin{proof} The Lemma is certainly true if we remove the 'loxodromic' condition, we show how we reduce the question to this situation. Since $\rho(\gh)$ is Zariski-dense $\xi^\t(\bord\gh)$ is the limit set of $\rho(\gh)$ on the flag space $\EuScript F_\t(\sf G)$. By Benoist \cite[Remarque 3.6 2)]{limite} it is the image, under the natural projection, of the limit set $\Lambda_{\rho(\gh)}$ in the full flag space $\EuScript F_\simple(\sf G)$. Again by Zariski-density, the latter is the closure of attracting full flags of loxodromic elements in $\rho(\gh)$, and the lemma follows.\end{proof}

\begin{proof}[Proof of Corollary \ref{corfullerton}]
Consider $\aa_i\in\t\cap\simple_i$. We will show that for each $i$ there exist a loxodromic $\g\in\gh$ such that $\mathrm{d}\peso_{\aa_i}^\g(v)\neq0.$ Otherwise, for every loxodromic $\g\in\gh$ one has $$\frac{\partial}{\partial t}\Big|_{t=0}\jordan_1\big(\rep_{\aa_i}(\rho_t(\g))\big)=0.$$

Using now \cite[Prop. 9.4]{pressure} one has that for every co-prime pair of loxodromic $\g,h\in\gh$ it holds $\partial^{\log{}}\mathrm b_{\rep_{\aa_i}\rho_t}(\g^-,\g^+,h^-,h^+)=0$ (recall \S\ref{crog}). By Lemma \ref{presslox}, this implies that for every $(x,y,z,t)\in\partial^{(4)}\gh$ one has $\partial^{\log{}}\mathrm b_{\rep_{\aa_i}\rho_t}(x,y,z,t)=0$. Moreover $\rep_{\aa_i}\rho$ is irreducible and projective Anosov,  so from this point on the proof of \cite[Lemma 10.3]{pressure} woks verbatim to give that the cocycle $p_i(\co_v)$ is cohomologically trivial, contradicting our assumptions.\end{proof}

\subsection{Cohomological independence of $\ledrappier$ and $\vec\ledrappier$}\label{Livsiccohom}
Consider the Ledrappier potential $\ledrappier_t=\ledrappier_{\rho_t}:\UG\to\a_\t$ and denote by $\vec \ledrappier=\vec{\ledrappier}{}:\UG\to\a_\t$ its differentiation w.r.t. $t$ at $0$. Corollary \ref{corfullerton} and Theorem \ref{LedIndepA} readily give:
 
\begin{cor}\label{Livindependent} Let $v\in\sf T_\rho\Anosov_\t(\gh,\sf G )$ have full variation with Zariski-dense $\rho.$ \begin{enumerate}\item\label{noZ} Consider $\varphi,\psi\in(\a_\t)^*$ and assume there exist $\g\in\gh$ and a multiplicity-$1$ root $\aa$ that strictly minimizes $\g$ among $\supp\varphi$. Then, $\varphi(\vec\ledrappier{})-\length(\ledrappier)$ is not Liv\v sic-cohomologous to a function with periods in $\Z$.
\item If moreover every root in $\t$ has multiplicity one then $\ledrappier$ and ${\vec\ledrappier}{}$ are Liv\v sic-cohomologically independent (thus Corollary \ref{masa} applies).
\end{enumerate}
\end{cor}

We conclude then the following.

\begin{cor}\label{nondeg} Consider $\length\in\conodual{\t,\rho}$ and assume there exist $\g\in\gh$ and $\aa$ with $\dim\ge_\aa=1$ that strictly minimizes $\g$ among $\supp\length$. If $v\in\sf T_\rho\caracteres$ has full variation then  $\PP^\length_\rho(v)>0$.\end{cor}

\begin{proof} By Theorem \ref{JJ}, $\PP^\length_\rho$ degenerates at $v$ iff $(\mathrm d \entropy\length{}v)\length(\ledrappier)$ and $\entropy{\length}{}\length\big({\vec\ledrappier}\big)$ are Liv\v sic-cohomologous. However this does not hold by Corollary \ref{Livindependent}\eqref{noZ}.\end{proof}

\subsection{Variations along level sets of $\scr h$ give non-proper actions}\label{leve} This principle is used in Labourie \cite{Labourie-FuchsAffine,LabourieAffine}, we give new situations where it applies.

Recall from Sullivan \cite{sullivan} (see Yue \cite{yue}), that if $\sf X$ has rank $1$ and $\rho:\gh\to\isom\sf X$ is convex-co-compact, then $\scr h_\rho$ is the Hausdorff dimension of the limit set of $\rho$ on $\partial_\infty\sf X$ for a visual metric. Also, by Bridgeman-Canary-Labourie-S. \cite{pressure} (see \S\,\ref{thermoSS}), the function $\rho\mapsto\scr h_\rho$ is real-analytic on the space of convex-co-compact representations.

The adjoint representation of a rank-one simple $\ge$ has neutralizing dimension $1$ as long as $\centro(\m)=\{0\}$. Thus, the rank 1 simple groups with $\neudim(\Ad)\neq1$ have Lie algebras equal to $\so_{1,3}$ or $\su_{1,n}$ for $n\geq2$ (see Knapp \cite[Appendix C]{knapp}), whence:

\begin{cor}[Rank 1]\label{levelsets} Let $\sf G$ be the identity component of the isometry group of $\H^n_\R$ for $n\neq3$, $\H^n_\H$ $n\geq2$, or the Cayley hyperbolic plane, and let $\rho:\gh\to \sf G$ be convex co-compact and Zariski-dense. Let $\co\in H^1_{\Ad\rho}(\gh,\ge)$ be an integrable cocycle. If $\mathrm d_\rho\scr h(\co)=0$ then the affine action of $\morfi_\co$ on $\ge$ is not proper.
\end{cor}

\begin{proof} If $\mathrm d\scr h(\co)=0$ Eq. \eqref{masaentropia} implies that $\peso_\aa(\mass{\peso_\aa}{\co})=0$. Lemma \ref{masainte} gives then $0\in\inte \VV{\peso_\aa}_{\co}$. Since $\centro(\m)=\{0\}$ Proposition \ref{margulisderivada} implies that $\VV{\peso_\aa}_{\co}=\M^{\peso_\aa}(\co)$. Kassel-Smilga's Proposition \ref{proper} gives then non-properness of the affine action.\end{proof}

\begin{cor} We let $\sf G$ be as in Corollary \ref{levelsets} and $\mathbb F$ be a non-abelian free group, then there exists $C>0$ such that if $\rho:\mathbb F\hookrightarrow\sf G$ is a Schottky group with contraction greater than $C$, then $\scr h$ is not critical at $\rho$.
\end{cor}

\begin{proof} Follows from Corollary \ref{levelsets} together with Smilga's construction \cite{smilga1}.
\end{proof}

Let us consider now $\sf G=\SL(3,\R)$ and the functional $\sf H\in\a^*$ given by $\sf H(a)=(a_1-a_3)/2$, the \emph{Hilbert entropy} $\entropy{\sf H}{}:\Anosov_\simple(\gh,\SL(3,\R))\to\R_+.$

\begin{cor}[$\SL(3,\R)$]\label{HilbertCritica}Consider  $0\neq v\in\sf T_\rho\Anosov_\simple(\gh,\SL(3,\R))$ with Zariski-dense base-point. If $\mathrm d\entropy{\sf H}{}(v)=0$ then the affine action on $\sl(3,\R)$ via $\co_v$ is not proper.
\end{cor}

\begin{proof} Since $\SL(3,\R)$ is simple $v$ has full variation. Applying Corollary \ref{Livindependent} we obtain that $\VV{\sf H}_v$ has non-empty interior and by Corollary \ref{masa} $$\mass{\sf H}v\in\inte\VV{\sf H}_v.$$

By Equation \eqref{masaentropia} $\sf H(\mass{\sf H}v)=-\mathrm d\log \entropy{\sf H}{}(v)=0$, so $\mass{\sf H}v\in\ker\sf H$. Since $\sf H$ is $\ii$-invariant, the set of normalized variations $\VV{\sf H}_v$ is also $\ii$-invariant so we obtain $\ii(\mass{\sf H}v)=-\mass{\sf H}v\in\inte\VV{\sf H}_v$,  thus by convexity, $0\in\inte\VV{\sf H}_v.$ Since $\SL(3,\R)$ is split, $ \m=0$ and whence $\centro(\m)=\{0\}$, so by Proposition \ref{margulisderivada} we conclude that $0\in\inte\M^{\sf H}(\co_v)$. Kassel-Smilga's Proposition \ref{proper} gives then non-properness of the action.\end{proof} 

As before, Corollary \ref{HilbertCritica} together with Smilga's construction \cite{smilga1} gives:

\begin{cor} Let $\mathbb F$ be a non-abelian free group, then there exists $C>0$ such that if $\rho:\mathbb F\hookrightarrow\SL(3,\R)$ is a Schottky subgroup with contraction greater than $C$, then $\entropy{\sf H}{}$ is not critical at $\rho$.
\end{cor}

\section{The case of $\T$-positive representations}\label{Theta-positive}

Guichard-Wienhard \cite{GW-positive} have introduced, for a subset $\T\subset\simple$, the notion of a \emph{$\T$-positive representation}  $\pi_1S\to\sf G$. We refer to their work for the definition and instead use the following result, which states that these representations are $\T$-Anosov, and verify a stronger form called \emph{hyperconvexity}. We will whence begin by studying the latter in \S\ref{112}. The following collects results which can be found in Beyrer-Pozzetti \cite{BP}, Guichard-Labourie-Wienhard \cite{GLW}, Beyrer-Pozzetti-Guichard-Labourie-Wienhard \cite{BGLPW} in different generality and in Pozzetti-S.-Wienhard \cite{PSW1,PSW2}, S. \cite{clausurasPos} and \cite{BGLPW} for the second item.


\begin{thm}\label{tpos}\item\begin{enumerate}\item\label{anoso}\emph{{\bf{(}} \cite{BP,GLW,BGLPW}{\bf{)}}} Every $\T$-positive representation is $\T$-Anosov, moreover, $\tpos\subset\frak X\big(\pi_1S,\sf G\big)$ is open and closed. \item\label{hyper}\emph{{\bf{(}} \cite{PSW1,PSW2,clausurasPos,BGLPW}{\bf{)}}} For every $\upalpha\in\inte\T,$ and $\rho\in\tpos,$ the representation $\rep_\upalpha\rho$ is $(1,1,2)$-hyperconvex. 
\end{enumerate}
\end{thm}

\subsection{$(1,1,2)$-hyperconvex representations and Hausdorff dimension}\label{112} The main purpose of this section is to establish Corollary \ref{C1hff} below. We recall a definition from Pozzetti-S.-Wienhard \cite{PSW1}.

\begin{defi}\label{hyperdefi} Let $\K=\R$ or $\C.$ A representation $\rho:\gh\to\SL(d,\K)$ is $(1,1,2)$-\emph{hyperconvex} if it is $\{\slroot_1,\slroot_2\}$-Anosov and for every triple $x,y,z\in\bord\gh$ of pairwise distinct points one has $$\big(\xi^1(x)\oplus\xi^1(y)\big)\cap\xi^{d-2}(z)=\{0\}.$$ If $\sf G$ is real-algebraic  and $\aa\in\simple,$ then a representation $\rho:\gh\to\sf G$ is $\aa$-\emph{hyperconvex} if $\rep_\aa\rho$ is $(1,1,2)$-hyperconvex. 
We let $\Anosov^{\pitchfork}_{\aa}(\gh,\sf G)$ denote the space of such representations, it is an open subset of the character variety (\cite[Proposition 6.2]{PSW1}).
\end{defi}


\begin{thm}\label{simpleroot1}Let $\rho:\pi_1S\to\SL(d,\R)$ be $(1,1,2)$-hyperconvex, then
\begin{enumerate}\item \label{h==1}\emph{{\bf{(}}Pozzetti-S.-Wienhard \cite{PSW1}{\bf{)}}} $\entropy{\slroot_1}\rho=1$ and
\item\label{Gsimple}\emph{{\bf{(}}Pozzetti-S. \cite[Theorem C]{PS}{\bf{)}}} if $\rho(\pi_1S)$ acts furthermore irreducibly on $\R^d$ then the Zariski closure $\sf G$ of $\rho(\pi_1S)$ is simple and the highest restricted weight of the representation $\sf G\hookrightarrow\SL(d,\R)$ is a multiple of a fundamental weight for a root $\aa\in\simple$ with $\dim\ge_\aa=1$.  \end{enumerate}
\end{thm}

%

We let $\barra{\aa}$ be the set consisting of $\aa$ together with the simple roots neighbouring $\aa$ in the Dynking diagram of $\sf G.$ Observe that $\aa\in\inte\barra\aa$ and whence $\aa\in(\a_{\barra\aa})^*$.

\begin{cor}\label{roothyperconvex} Let $\rho\in\Anosov^\pitchfork_\aa(\gh,\sf G)$ have Zariski-dense image. Then $\rho$ is $\barra\aa$-Anosov. Moreover, $\PP^\aa_\rho$ is (well defined and) Riemannian.\end{cor}

\begin{proof} The representation $\rep_\aa$ has highest restricted weight $n\peso_\aa$, whence $\forall\,g\in\sf G$
\begin{alignat*}{2}\slroot_1\big(\jordan(\rep g)\big)&=\aa\big(\jordan_{\sf G}(g)\big),\\ \slroot_2\big(\jordan(\rep g)\big)&=\min_{\sroot\in\simple_\sf G:\<\aa,\sroot\>\neq0}\sroot\big(\jordan_{\sf G}(g)\big).\end{alignat*}

\noindent
Since $\rep_\aa\rho$ is $\{\slroot_2\}$-Anosov, this last equation gives that $\rho$ is $\barra\aa$-Anosov. Since $\aa\in(\a_{\barra\aa})^*$ the pressure $\PP^\aa$ is well defined and $\PP^{\slroot_1}_\rho(v)=\PP^\aa_\rho(v)$. By Theorem \ref{simpleroot1} one has $\mathrm{d}\entropy\aa{}v=0$ and thus $\PP^\aa(v)=0$ if and only if $\forall\g$ $\mathrm{d}\alpha^\g(v)=0$ (Eq. \eqref{P=0}).  However, since $\sf G$ is simple $v$ has full variation and, since $\rho$ is Anosov, by Corollary \ref{fullerton} it has full loxodromic variation, so Corollary \ref{phi-cotangent} concludes.\end{proof}


\subsubsection{Root system of the complexification} We recall here some facts needed in \ref{hyperC}. Let $\ge$ be a simple real Lie algebra with Cartan decomposition $\ge=\k\oplus\p$.

Let $\ge_\C$ be the complexification of $\ge$, then $\u=\k\oplus i\p$ is a compact subalgebra and $\s=\p\oplus i\k$ is an $\ad_\u$-module. Thus, the involution $\tau:\ge_\C\to\ge_\C$ defined as $\tau|\u=\id$ and $\tau|\s=-\id$ is a Cartan involution of $\ge_\C$.

Let $\b\subset\m$ be a maximal Abelian subalgebra. Then $\h=\a\oplus\b$ is a Cartan subalgebra of $\ge$, meaning that $\h_\C$ is a Cartan subalgebra of $\ge_\C$. As such, we have a root-space decomposition $$\ge_\C=\h_\C\oplus\bigoplus_{\aa\in\sf{\Sigma}}(\ge_\C)_\aa,$$ where $(\ge_\C)_\aa=\big\{x\in\ge_\C:\forall h\in\h_\C\textrm{ one has }[h,x]=\aa(h)x\big\}$ and $\sf{\Sigma}=\big\{\aa\in(\h_\C)^*:(\ge_\C)_\aa\neq\{0\}\big\}.$


Corollary 6.49 from Knapp's book \cite{knapp} states that every $\aa\in\sf\Sigma$ verifies $\aa|\a\oplus i\b$ is real-valued, so since $\a\oplus i\b$ is a real form form of $\h_\C$, $\aa$ is uniquely determined by $\aa|\a\oplus i\b$. Moreover $\a\oplus i\b$ is a maximal abelian subspace of $\s$, so $\sf{\Sigma}'=\{\aa|\a\oplus i\b:\aa\in\sf{\Sigma}\}$ is the restricted root system of $\ge_\C$ as a real Lie algebra of non-compact type.

One has also that (\cite[Eq. 6.48b]{knapp}) if $\sroot\in\root$ then $$\ge_\sroot=\ge\cap\Big(\bigoplus_{\aa\in\sf{\Sigma}:\aa|\a=\sroot}(\ge_\C)_\aa\Big).$$ We obtain thus the following Lemma.

\begin{lemma}\label{rootspaceC}Let $\ge$ be simple and assume there exists $\sroot\in\root$ such that $\dim\ge_\sroot=1$, then $\ge_\C$ is simple and there exists a unique $\aa_\sroot\in\sf{\Sigma}$ such that $\aa_\sroot|\a=\sigma.$ Consequently, if $\sroot\in\simple$, we have a natural embedding of flag spaces $\EuScript F_{\{\sroot\}}(\sf G)\subset\EuScript F_{\{\aa_\sroot\}}(\sf G_\C).$
\end{lemma}

\begin{proof} The real algebra $\ge$ cannot in itself be complex (otherwise every root space has dimension $\dim_\R=2$), so the first statement follows from \cite[Theorem 9.4(b)]{knapp}.

Concerning the second statement, for every $\aa\in\sf{\Sigma}$ one has $(\ge_\C)_\aa\cap\ge\neq0,$ indeed if $x+iy\in(\ge_\C)_\aa$ with $x,y\in\ge$, then $$[a,x+iy]=[a,x]+i[a,y]=\aa(a)x+\aa(a)iy.$$ So since $\aa|\a$ is real valued one has, for all $a\in\a$, $[a,x]=\alpha(a)x$ and $[a,y]=\alpha(a)y$. The last assertion follows. \end{proof}

\subsubsection{Variation of Hausdorff dimension on complex groups}\label{hyperC} Throughout this subsection we let $\sf G_\C$ be the group of $\C$-points of $\sf G$. Consider the embedding $\sf G\subset\sf G_\C$.

If $\aa$ is a restricted root of $\sf G_\C$ (as a real group), then we fix a Riemannian metric on the flag space $\EuScript F_{\{\aa\}}(\sf G_\C)$, denote by $\Hff(X)$ the associated Hausdorff dimension of a subset $X$ of $\EuScript F_{\{\aa\}}(\sf G_\C)$ and consider the function \begin{alignat*}{2}\Hff_\aa &:\Anosov_{\{\aa\}}(\gh, \sf G_\C)\to\R_{>0}\\\rho&\mapsto\Hff\big(\xi^\aa(\bord\gh)\big).\end{alignat*}  It follows from Pozzetti-S.-Wienhard \cite{PSW1} and Bridgeman-Canary-Labourie-S. \cite{pressure} that $\Hff_\aa$ is an analytic function on $\Anosov^{\pitchfork}_\aa(\gh,\sf G_\C).$


By Lemma \ref{rootspaceC}, if $\sroot\in\simple$ is a restricted root of $\sf G$ so that $\dim\ge_\sroot=1$, then there is a unique restricted root $\aa_\sroot$ of $\sf G_\C$ so that $\aa_\sroot|\a=\sroot$, so we have an inclusion $$\Anosov_{\sroot}^\pitchfork(\gh,\sf G)\subset\Anosov_{\aa_\sroot}^\pitchfork(\gh,\sf G_\C),$$ and an embedding of the flag spaces $\EuScript F_{\{\sroot\}}(\sf G)\subset\EuScript F_{\{\aa_\sroot\}}(\sf G_\C)$.

Let $\sf J$ denote the almost-complex structure of $\frak X(\pi_1S,\sf G_\C)$ induced by the complex structure of $\sf G_\C$. Let us also consider the irreducible representation $\rep_\aa:\sf G\to\SL(d,\R)$ which extends by complexifying to a representation $\rep_\aa:\sf G_\C\to\SL(d,\C)$. Since $\rep_\aa\sf G$ contains a proximal element, the complexified representation is also irreducible (over $ \C$). We recall a needed ingredient.

\begin{thm}[{Bridgeman-Pozzetti-S.-Wienhard \cite{HessianHff}}]\label{BPSWHyper} Let $\rho:\pi_1S\to\PSL(d,\R)$ be $(1,1,2)$-hyperconvex. Then for every integrable $v\in\sf T_\rho\frak X(\pi_1S,\PSL(d,\R))$ with integrable ${\sf J}v,$ one has $\hess_\rho\Hff_{\slroot_1}(\sf{J} v)=\PP^{\slroot_1}_\rho(v).$\end{thm}

We establish then the purpose of this section.

\begin{cor}\label{C1hff} Let $\rho\in\Anosov_\aa^\pitchfork(\gh,\sf G)$ have Zariski-dense image. Then for every $v\in\sf T_\rho\frak X(\pi_1S,\sf G_\C)$ that is not tangent to the real characters one has $$\hess_\rho\Hff_\aa(v)>0.$$ In particular, there exists a neighborhood $\EuScript V$ of $\rho$ in $\frak X(\pi_1S,\sf G_\C)$ such that if $\eta\in\EuScript V$ verifies $\Hff\big(\xi^\aa_\eta(\bord\pi_1S)\big) =1$ then the Zariski closure of $\eta(\pi_1S)$ is (conjugate to) $\sf G$.
\end{cor}

\begin{proof} Theorem \ref{regularsurface} implies that $\rho$ is a regular point of $\frak X(\pi_1S,\sf G_\C)$ and thus, since $\rho$ has values in $\sf G$, one has \begin{equation}\label{tsplits}\sf T_\rho\frak X(\pi_1S,\sf G_\C)=\sf T_\rho\frak X(\pi_1S,\sf G)\oplus\sf J\big(\sf T_\rho\frak X(\pi_1S,\sf G)\big).\end{equation}

Theorem \ref{BPSWHyper} gives that for $v\in\sf T_\rho\frak X(\pi_1S,\sf G)$, $\hess_\rho\Hff_{\slroot_1}(\sf Jv)=\PP^{\slroot_1}_\rho(v).$ Thus combining with Corollary \ref{roothyperconvex} we obtain that, for every $v\in\sf T_\rho\frak X(\pi_1S,\sf G)$ $$\hess_\rho\Hff_\aa(\sf J v)>0.$$ The result then follows.\end{proof}

\begin{cor}\label{HffTpos} For every non-zero $v\in\sf T\tpos$ with Zariski-dense basepoint, and $\ua\in\inte\T$ one has $\hess_\rho\Hff_{\upalpha}(\sf Jv)>0.$\end{cor}

\begin{proof} Follows from Theorem \ref{tpos}\eqref{hyper} and Corollary \ref{C1hff}.\end{proof}

\subsection{Length functions and pressure} We establish the following:


\begin{cor}\label{inteT}Let $\rho:\pi_1S\to\sf G$ be $\T$-positive and $\sf H$ be its Zariski closure.
\begin{enumerate}
\item If $\sf H=\sf G$ then  $\forall\,\length\in\conodual{\T,\rho}$ the pressure form $\PP_\rho^\length$ is Riemannian.

\item\label{estrata} If $\sf H$ is simple, then for every $\length\in\conodual{\inte\T,\rho}$ one has $\PP^\length$ is Riemannian when restricted to characters with values in $\sf H$.
\end{enumerate}
\end{cor}

\begin{proof} When $\rho(\pi_1S)$ is Zariski-dense then, since all roots in $\T$ have one-dimensional root spaces and $\sf G$ is simple, the result readily follows from Corollary \ref{nondeg}.

The second item is a bit more involved. The combination of Theorem \ref{tpos}\eqref{hyper} and S. \cite[Lemma 4.8]{clausurasPos} imply that for every $\upalpha\in\inte\T$ there exists (a unique) $\sroot_\upalpha\in\simple_\sf H$ such that for all $\g\in\pi_1S$ it holds $$\sroot_\upalpha(\jordan_{\sf H}\big(\iota\g)\big)=\upalpha_k\big(\jordan_\sf G(\rho\g)\big).$$ In particular $\iota:\pi_1S\to\sf H$ is $\{\sroot_\upalpha\}$-Anosov. Moreover \cite[Lemma 4.8]{clausurasPos} states that $\dim \ge_{\sroot_\upalpha}=1$. Theorems \ref{simpleroot1} and \ref{tpos}\eqref{hyper} imply moreover that $$\peso_\upalpha(\jordan_{\sf G}\rho \g)=n_\upalpha\peso_{\sroot_\upalpha}(\jordan_\sf H(\iota\g)).$$ It follows that, since $\length\in\<\{\peso_\upalpha:\upalpha\in\inte\T\}\>$ there exists $\bar\length\in(\a_\sf H)^*$ such that for every $\g\in\pi_1S$ one has $$\length(\jordan_{\sf G}(\rho\g))=\bar\length(\jordan_\sf H(\iota\g)).$$ Moreover $\bar\length\in\conodual{\{\sroot_\upalpha\},\iota}$ so  Corollary \ref{nondeg} yelds the desired non-degeneracy.
\end{proof}


We introduce then the following definition.

\begin{defi} A \emph{length function} on $\tpos$ is a $\Mod(S)$-invariant map $\uppsi:\tpos\to(\a_\T)^*$ so that for every $\rho\in\tpos$ one has $\uppsi(\rho)\in\conodual{\T,\rho}$.
\end{defi}

\begin{cor}\label{tposlengthfunction} Let $\uppsi:\tpos\to(\a_{\inte\T})^*$ be a length function. Then the semi-definite form $\rho\mapsto\PP^{\uppsi(\rho)}$ induces a $\Mod(S)$-invariant path metric on the space of $\T-$positive representations with simple Zariski-closure.
\end{cor}

We emphasize that our length function has values in $(\a_{\inte\T})^*\subset(\a_\T)^*$.

\begin{proof}

The set of pairs $(\sf H,\rep)$, where $\sf H$ is a simple Lie group and $\rep:\sf H\to\sf G$ is an irreducible representation up to conjugation, is finite. We further restrict the class of such pairs by only considering $(\sf H,\rep)$ if there exists $\rho\in\tpos$ with Zariski closure conjugate to $\rep(\sf H)$. By assumption, the finitely many submanifolds of $$W_{(\sf H,\rep)}=\big\{\rho\in\tpos:\overline{\rho(\pi_1S)}{}^{\mathrm Z}\subset\textrm{ a conjugate of }\rep(\sf H)\big\}$$ exhaust the space we want to understand. Moreover, by Labourie \cite[Theorem 5.2.6]{LabourieLectures} the set $$W_{(\sf H,\rep)}^\mathrm{Z}=\big\{\rho\in W_{(\sf H,\rep)}:\overline{\rho(\pi_1S)}{}^\mathrm{Z}\textrm{ is conjugate to }\rep(\sf H)\big\} $$ is open on $W_{(\sf H,\rep)}$ and $W_{(\sf H,\rep)} \setminus W_{(\sf H,\rep)}^\mathrm{Z}$ has dimension strictly smaller than $\dim W_{(\sf H,\rep)}$.

Since we have chosen the length function $\uppsi$ to have values in $(\a_{\inte\T})^*$, Corollary \ref{inteT}\eqref{estrata} implies that $\rho\mapsto\PP^{\uppsi(\rho)}$ is Riemannian on every $W_{(\sf H,\rep)}^\mathrm{Z}$, so  Lemma \ref{nested} below gives the desired conclusion.\end{proof}

\begin{lemma}[{Bray-Canary-Kao-Martone \cite[Lemma 5.2]{pressurecusped}}]\label{nested}Let $W_0$ be a smooth manifold and let $W_n \subset W_{n-1} \subset\cdots\subset W_1 \subset W_0$ be a nested collection of submanifolds of $W_0$ so that $W_i$ has non-zero codimension in $W_{i-1}$ for all $i$. Set $W_{n+1} = \vacio$. Suppose that $g$ is a smooth non-negative symmetric $2$-tensor on $W_0$ such that for every $i \in\lb0,n\rb$, the restriction of $g$ to $\sf T_xW_i$ is positive definite if $x \in W_i \setminus W_{i+1}$. Then, the path pseudo-metric defined by $g$ is a metric.
\end{lemma}

\part{Hitchin components}

Let $\ge$ be a simple split real Lie algebra and $\Int\ge$ its group of inner automorphisms. Let  also $\s\subset\ge$ be a principal $\sl_2$ as in \S\,\ref{xalpha}. This Lie-algebra morphism comes from a Lie-group morphism $\tau_\ge=\PSL(2,\R)\to\Int\ge$ also called \emph{principal}.

Let also $S$ be a closed orientable connected surface of Euler characteristic $\chi(S)<0$. A representation $\rho:\pi_1S\to\Int\ge$ is \emph{Fuchsian} if it factors as $$\pi_1S\to\PSL(2,\R)\xrightarrow{\tau_\ge}\Int\ge,$$ where the first arrow is discrete and faithful. A connected component of the character variety $\frak X(\pi_1S,\Int\ge)$ that contains a Fuchsian representation will be called a \emph{Hitchin component} of $\Int\ge$ and denoted by $\hitchin_\ge(S)$. Hitchin \cite{Hitchin} established that $\hitchin_\ge(S)$ is a contractible differentiable manifold of dimension $|\chi(S)|\dim\ge$. The \emph{Fuchsian locus} inside $\hitchin_\ge(S)$ is a natural copy of the Teichm\"uller space of $S$.

In this section we establish Theorem \ref{thmPrinc} describing degenerations of pressure forms on $\hitchin_\ge(S)$. Actually, by Labourie \cite{labourie} and Beyrer-Labourie-Guichard-Pozzeti-Wienhard \cite{BGLPW}, Hitchin representations are $\simple$-Anosov, so representations with Zariski-dense image are already dealt with by Corollary \ref{nondeg}. Moreover, since Hitchin representations are $\simple$-positive (Fock-Goncharov \cite{FG} and \emph{loc. cit.} \cite{BGLPW}) and have simple Zariski-closure by Theorem \ref{clausuras} below,  Corollary \ref{tposlengthfunction} establishes separation of the path-pseudo metric, for every length function $\uppsi:\hitchin_\ge(S)\to\a^*$.




\begin{cor}\label{pathmetricHitchin} For any length function $\uppsi:\hitchin_\ge(S)\to\a^*$ the associated pressure semi-norm $\rho\mapsto\PP^{\uppsi(\rho)}$ induces a $\Mod(S)$-invariant path metric on $\hitchin_\ge(S)$.
\end{cor}



However, understanding degenerations at non-Zariski-dense points is much more subtle and will require some work. This will be finally established in Theorem \ref{degLietheoretic}.

A key object to understand these degenerations is that of \emph{Kostant lines} of $\ge$, by definition these are the $0$-restricted weight space of an $\ad\s$-module. They appear in Theorem \ref{thmPrinc} but are also needed to understand Hausdorff dimension degenerations (\S\,\ref{Hffdeg}). This is why we will find explicit formulae for these lines (\S\,\ref{ktsec}). These computations play a role on giving an explicit description of the functional $\kahl\in\a^*$ whose pressure form $\PP^\kahl$ is compatible with Goldman's symplectic form at Fuchsian points (Corollary \ref{compa1}).

\section{Necessary facts}

The opposition involution $\ii$ of types $\A$, $\D$ and $\E_6$ is induced by a non-trivial external involution $\undi:\Int\ge\to\Int\ge$, unique up to conjugation, that induces in turn a non-trivial involution of the character variety that preserves each Hitchin component $\undiX:\hitchin_\ge(S)\to\hitchin_\ge(S)$. We have thus natural inclusions \begin{alignat}{2}\label{ptfixes}\Fix\undiX=\hitchin_{\B_k}(S) & \subset\hitchin_{\A_{2k}}(S),\nonumber\\ \Fix\undiX=\hitchin_{\Ce_k}(S) & \subset\hitchin_{\A_{2k-1}}(S),\nonumber\\
\Fix\undiX=\hitchin_{\B_k}(S)& \subset\hitchin_{\D_{k+1}}(S),\nonumber\\ \Fix\undiX=\hitchin_{\F_4}(S)& \subset\hitchin_{\E_6}(S).\end{alignat} There is also another natural inclusion $\hitchin_{\Ge}(S)\subset\hitchin_{\B_3}(S)$ given by the fact that the fundamental representation for the short root of $\rep:\gedossplit\to\so(3,4)$ sends a principal $\sl_2$ of $\gedossplit$ to a principal $\sl_2$ of $\so(3,4)$.



Since a Hitchin representation is $\simple$-Anosov (Labourie \cite{labourie}) we may consider its \emph{critical hypersurface} as in \S\,\ref{anosovpre}. The following was first stablished by Potrie-S. \cite{exponentecritico} for types $\A$, $\B$, $\Ce$ and $\Ge$ and the work from Pozzetti-S.-Wienhard \cite{PSW1} together with S. \cite{clausurasPos} gives a unified approach for all types:

\begin{thm}[\cite{exponentecritico,PSW1,clausurasPos}]\label{PS} For every 
 $\rho\in\hitchin_\ge(S)$ one has $\simple\subset\EuScript Q_\rho$. \end{thm}

Convexity of the critical hyper-surface together with the above gives then:

\begin{cor}\label{fuchscritica} If $\prin$ is Fuchsian then, for every $\length\in\conodual{\prin},$ $\entropy\length\prin$ is critical at $\prin$.
\end{cor}

The following is a consequence of Luzstig's positivity from Fock-Goncharov \cite{FG}:

\begin{prop}\label{todopar} For every $\rho\in\hitchin_\ge(S)$ and every pair of transverse $\g,h\in\pi_1S$, the pair $\rho(\g)$ and $\rho(h)$ is strongly transversally $\simple$-proximal.
\end{prop}

The following recovers a result by Guichard \cite{clausura} for types $\A$, $\B$, $\Ce$ and $\Ge$.

\begin{thm}[S. \cite{clausurasPos}]\label{clausuras}Let $\rho\in\hitchin_\ge(S)$ have Zariski closure $\sf H.$ Then $\h_{ss}$ is either $\ge,$ a principal $\sl_2(\R),$ or $\Int\ge$-conjugated to one of the possibilities in Table \ref{unesco}.
\end{thm}

\begin{table}[h!]
  \begin{center}
    \begin{tabular}{r|r|r} 
      $\ge$ & $\h_{ss}$ & $\phi:\h_{ss}\to\ge$  \\
      \hline
     $\sl_{2n}(\R)$ & $\sp(2n,\R)$ & defining representation \\\hline
      \multirow{2}{*}{$\sl_{2n+1}(\R)$} & $\so(n,n+1)$ $\forall n$ & defining representation \\
       & $\gedossplit$ if $n=3$ & fundamental for the short root\\\hline
      $\so(3,4)$ & $\gedossplit$ & fundamental for the short root \\\hline
      \multirow{4}{*}{$\so(n,n)$} & $\so(n-1,n)$ $\forall n\geq3$ & stabilizer of a non-isotropic line \\
       & $\so(3,4)$ if $n=4$ & fundamental for the short root \\\cline{3-3}
       & \multirow{2}{*}{$\gedossplit$ if $n=4$} & stabilizes a non-isotropic line $L$ and is  \\ & & fundamental for the short root on $L^\perp$\\\hline 
      $\e_6$ & $\f_4$ & $\Fix\undi$ (see Eq. \eqref{ptfixes})
    \end{tabular}
    \caption{Theorem \ref{clausuras}. If a simple split algebra $\ge$ is not listed in the first column then $\h_{ss}$ is either $\ge$ or a principal $\sl_2(\R);$ $\e_6,\f_4$ and $\gedossplit$ denote the split real forms of the corresponding exceptional complex Lie algebras. 
} \label{unesco}
  \end{center}
\end{table}

We conclude with the proof of the following Corollary from the Introduction.

\begin{cor}[Curves with arbitrarily small root-variation]\label{rootsmall} Consider $\slroot\in\simple$ and let $0\neq v\in\sf T_\rho\hitchin_{\ge}(S)$ have Zariski-dense base-point. Then there exists $h$ such that for positive $\eps$ and $\delta$ there exists $C>0$ with $$\#\big\{[\g]\in[\pi_1S]\textrm{ primitive}:\peso_{\slroot}^\g(\rho)\in(t-\eps,t]\textrm{ and }|\mathrm d\slroot^\g(v)|\leq\delta\big\}\sim C \frac{e^{ht}}{t^{3/2}}.$$ 
\end{cor}

\begin{proof} Since $\rho$ is Hitchin, the flow $\phi^{\peso_\slroot(\ledrappier)}$ is H\"older-conjugated to a $\class^{1+\alpha}$-Anosov flow $\Phi=(\Phi_t:\sf US\to\sf US)_{t\in\R}$ (\cite{exponentecritico,PSW1}). Theorem \ref{teoB} implies that group spanned by $$\Big\{\big(\mathrm d\slroot^\g(v),\peso_\slroot^\g(\rho)\big):\g\in\pi_1S\Big\}$$ is dense in $\R^2.$ Finally, by Theorem \ref{PS} $0\in\inte\slroot\big(\VV{\peso_\slroot}_v)$. This places $\Phi$ together with the potential $\slroot(\vec\ledrappier)$ in the assumptions of Babillot-Ledrappier \cite[Theorem 1.2]{babled}, where we pick $\xi=0$.\end{proof}

\section{Kostant lines}\label{ktsec}

Recall from Kostant \cite{kostant} that there are $\rk\ge$ irreducible adjoint factors of $\s$ and they have odd dimensions $2e+1.$ The numbers $e$ are called \emph{the exponents} of $\ge$ and the associated factor is denoted by $V_e.$   Table \ref{exponentes} gives the exponents for each type.

\begin{table}[h]
\begin{center}
\begin{tabular}{r|r} 
$\simple_\ge$ & exponents  \\\hline     
$\A_d$& $1,2,\ldots,d$ \\\hline
$\B_d $ & $1,3,5,\ldots,2d-1$\\\hline
$\Ce_d$ &$1,3,5,\ldots,2d-1$\\\hline
$\D_d$ & $1,3,\ldots, 2d-3,d-1$\\\hline 
$\E_6$ & $1,4,5,7,8,11$\\\hline
$\E_7$ & $1,5,7,9,11,13,17$\\\hline
$\E_8$ & $1,7,11,13,17,19,23,29$\\\hline
$\F_4$ & $1,5,7,11$\\\hline
$\Ge$ & $1,5$
\end{tabular}
\caption{Exponents of irreducible reduced root systems} \label{exponentes}
\end{center}
\end{table}

If $e$ is an exponent of $\ge,$ then the 0-restricted-weight space of $V_e$ is a line of $\a$ that we will denote by $\kt^e=\kt^e_\ge=V_e\cap\a$ and call \emph{the Kostant line} of exponent $e.$ In this section we task on giving a rather explicit description of these lines.

\begin{obs}\label{perp}If $e\neq f$ then $\kt^e$ and $\kt^f$ are orthogonal for the Killing form.
\end{obs}

\begin{proof} Let $\s=\<E,H,F\>$ have the standard relations of an $\sl_2$-triple and let $0\neq v_e^+\in V_e$ belong to the highest restricted weight space, so $\ad E\cdot v^+_e=0$. By definition one has $0\neq k_e=\ad(F)^e(v_e^+)\in\kt^e$. By associativity of the Killing form one has $$(k_e,k_f)=(-1)^{e+f}\big((\ad F)^e(v_e^+),(\ad F)^f(v_f^+)\big)=(-1)^{f}\big(v_e^+,(\ad F)^{e+f}(v_f^+)\big).$$ However if $e>f$ then $(\ad F)^{e+f}(v^+_f)=0,$ so, $(k_e,k_f)=0.$\end{proof}

\begin{obs}\label{evenw0}Kostant lines $\kt^e$ are fixed by the longest element of the Weyl group of $\A_d$ for even exponent $e,$ and are anti-fixed for odd exponent $e.$
\end{obs}


\subsection{$\A,$ $\B$ and $\Ce$}

If we denote by $f(x)=x(d-x),$ then the triple $\{E,F,H\}$ below spans a principal $\sl_2$ of $\sl_d(\R),$ denoted by $\s:$

\begin{alignat*}{2} H & =  \diag(d-1,d-3,\ldots,1-d),\\ 
E &= \begin{pNiceMatrix}  & 1 &  & & \\
& & \Ddots  & & \\
 & & & & \\
 & & & &1 \\
 & & & & \end{pNiceMatrix}\textrm{ and } F  =  \begin{pNiceMatrix}  0&  &  & & \\
f(1) & \Ddots& & & \\
 & f(2) & & & \\
 & & \Ddots& & \\
 & & & f(d-1)&0 \end{pNiceMatrix}.\end{alignat*}


\noindent
Consider the matrix product $E^e=E\cdots E$  and define the \emph{Kostant vector} $$\ktv^e:=(-1)^e(\ad F)^e(E^e)=\Big[\cdots\big[[E^e,F],F\big]\cdots F\Big].$$

For every $e\in\lb1,d\rb$ the $E^e$ is annihilated by $\ad E$ and is an eigenvector of $\ad H$ of eigenvalue $2e.$ The space $\spa\{(\ad F)^l\cdot E^e:l\in\Z_{\geq0}\}$ is thus an $\ad\s$-module of dimension $2e+1.$ The $0$-restricted weight space $\R\cdot (\ad F)^e(E^e)$ is the Kostant line $\kt^e$ and thus $\ktv^e\in\kt^e-\{0\}$.

Denote by $\pi^{i,j},$ for $i,j\in\lb1,d\rb,$ the elementary matrix whose only non-vanishing entry is $(j,i),$ and this entry is $1,$ this is to say, $\pi^{i,j}$ is the operator sending $e_j\mapsto e_i$ and $e_k\mapsto 0$ for every $k\neq j.$ We simplify $\pi^{i,i}$ as $\pi^i.$ Elementary computation gives:

\begin{equation}\label{pj}[\pi^{i,j},\pi^{l,t}]=\left\{\begin{array}{cc}0 & \textrm{ if }i\neq t,l\neq j,\\ \pi^{i,t} & \textrm{ if }i\neq t,j=l,\\\pi^i-\pi^j& \textrm{ if }i=t,l=j,\\-\pi^{l,j}&\textrm{ if }i=t,l\neq j.\end{array}\right.\end{equation}

Also, with this notation one has \begin{equation}\label{EFpi}E^e=\sum_{j=1}^{d-e} \pi^{j,e+j}\textrm{ and }F=\sum_{i=1}^{d-1}f(i)\pi^{i+1,i}.\end{equation}

\begin{prop}\label{formulaexponentes} One has that \begin{alignat}{2}\label{j=1}\ktv^e & =(-1)^e\sum_{l=1}^{d-e}\Big(f(l)\cdots f(l+e-1)\cdot\sum_{t=0}^{e}(-1)^{t}\binom{e}{t}\pi^{l+e-t}\Big)\nonumber \\ & =\sum_{j=1}^d\pi^j\Big(\sum_{t=0}^e(-1)^{t}\binom{e}{t}f(j-t)\cdots f(j-t+e-1)\Big).\end{alignat} 
\end{prop}

For example one has \begin{alignat}{3}\label{ejsktv}\ktv^2 & =2\sum_{j=1}^d\big(d^2+3d(1-2j)+6j(j-1)+2\big)\pi^j, \nonumber\\  \ktv^3 & = 6\sum_{j=1}^d(-2j + 1 + d)(d^2 - 10dj + 10j^2 + 5d - 10j + 6)\pi^j,\nonumber\\ \ktv^{d-1} &  = f(1)\cdots f(d-1)\sum_{j=1}^d(-1)^{j-1}\binom{d-1}{d-j}\pi^j.\end{alignat} 

\begin{proof} We compute  $(\ad F)^e(E^e)$. Using Equation \eqref{EFpi} this translates to computing the brackets $$\Big[\sum_{i=1}^{d-1}f(i)\pi^{i+1,i},\big[\cdots,\big[\sum_{i=1}^{d-1}f(i)\pi^{i+1,i},\sum_{j=1}^{d-e} \pi^{j,e+j}\big]\cdots\big]\Big],$$ for which we use the elementary computations in Equation \eqref{pj}. To do so, we use a recursive argument, for which we compute the bracket $[F,a(l,t)\pi^{l,t}]$ for arbitrary $l,t\in\lb1,d\rb$ and some real-valued function $a.$ Direct computation gives then $$[F,a(l,t)\pi^{l,t}]=a(l,t)\big(f(l)\pi^{l+1,t}-f(t-1)\pi^{l,t-1}\big).$$ 

Applying again $[F,\cdot],$ the term $\pi^{l+1,t-1}$ will appear once for each factor $\pi^{l+1,t}$ and $\pi^{l,t-1},$ with coefficient $f(l)f(t-1)(-1)2.$ If one further applies $[F,\cdot]$ one readily sees the binomial coefficients with the alternating signs appearing as the coefficient of $\pi^{l+k-j,t-j}$ (together with the corresponding $f$'s), this is to say \begin{alignat}{2} (\ad F)^k(\pi^{l,t}) & =  f(l)\cdots f(l+k-1)\pi^{l+k,t}\nonumber\\ & +\sum_{i=1}^{k-1}f(l)\cdots f(l+k-i-1)\cdot f(t-1)\cdots f(t-i)(-1)^i\binom ki\pi^{l+k-i,t-i} \nonumber\\ & + (-1)^kf(t-1)\cdots f(t-k)\pi^{l,t-k}\end{alignat}

Thus, replacing $t=l+e$ and $k=e$ one has: \begin{alignat}{2} (\ad F)^e(\pi^{l,l+e}) & =  f(l)\cdots f(l+e-1)\pi^{l+e,l+e}\nonumber\\ & +\sum_{i=1}^{e-1}f(l)\cdots f(l+e-i-1)\cdot f(l+e-1)\cdots f(l+e-i)(-1)^i\binom ei\pi^{l+e-i,l+e-i} \nonumber\\ & + (-1)^ef(l+e-1)\cdots f(l+e-e)\pi^{l,l+e-e}\nonumber\\&= f(l)\cdots f(l+e-1)\sum_{i=0}^e(-1)^i\binom ei\pi^{l+e-i}.\end{alignat} Summing on $l$ from $1$ to $d-e$ gives the first required formula.

The second equality is not completely immediate from the first so we quickly explain how it is obtained. By standard reordering of the sum one gets: 

\begin{alignat*}{3}\ktv^e &= (-1)^e\Big(\sum_{l=1}^{d-e}f(l)\cdots f(l+e-1)\cdot\sum_{t=0}^{e}(-1)^{t}\binom{e}{t}\pi^{l+e-t}\Big) & \\ & =(-1)^e\sum_{s=0}^e(-1)^{e-s} \sum_{j=1}^{d-e}\big(\binom{e}{s}f(j)\cdots f(j+e-1)\big)\pi^{j+s} & (s=e-t)\\ &= \sum_{s=0}^e(-1)^{s} \sum_{i=s+1}^{d-e+s}\big(\binom{e}{s}f(i-s)\cdots f(i-s+e-1)\big)\pi^{i} &\quad (i=j+s).
\end{alignat*}

One observes then that for every $i\in\lb1,s\rb$ the number $f(i-s)\cdots f(i-s+e-1)=0,$ since $i-s\leq0$ and $i-s+e-1\geq0$ (recall $s\in\lb0,e\rb$), so one can extend the lower index of the sum in $i$ in the above formula to starting from $i=1$ and the sum will be unchanged. Analogous reasoning allows to extend the upper index of the sum (recall $f(x)=f(d-x)$) so the proof is complete.\end{proof}

We need the following to describe adjoint factors in the Hitchin component.

\begin{lemma}[Exponents are shifted]\label{permutaexponentes} Consider $e,k\in\lb2,d-1\rb$ then the vector $[[F,E^e],E^k] \in\R\cdot E^{e+k-1}$. Moreover, for $k\leq d-3,$ $[[F,E^3],E^k]\neq0$.
\end{lemma}

\begin{proof} The centralizer of $E$ has dimension $d-1$ (Kostant \cite[Corollary 5.3]{kostant}). It is thus spanned, as a vector space, by $\{E^l:l\in\lb1,d-1\rb\}.$ The first assertion of the lemma follows by the combination of two straightforward calculations: \begin{alignat*}{2} \ad_E\big(\big[[F,E^e],E^k\big]\big) & = [\ad_E([F,E^e]),E^k]=\big[[H,E^e],E^k\big]=2e[E^e,E^k]=0;\\ \ad_H\big(\big[[F,E^e],E^k\big]\big) & =\big[\ad_H([F,E^e]),E^k\big]+\big[[F,E^e],\ad_H(E^k)\big]\\ & =(2e-2+2k)\big[[F,E^e],E^k\big].\end{alignat*} Indeed, the first computation gives that the desired element belongs to the span of $\{E^l:l\in\lb1,d-1\rb\}$, and the second asserts that it is an eigenvector of $\ad_H$ of eigenvalue $2(e+k-1)$, giving the desired conclusion.

To show that $[[F,E^3],E^k]\neq0$ if $k\leq d-3$ we use Eq. \eqref{pj}. One has $$[F,E^3]= \sum_{j=1}^{d-3}f(j)\pi^{j+1,j+3}-f(j+2)\pi^{j,j+2}.$$ Since we intend to further bracket with $E^k=\sum_{l=1}^{d-k}\pi^{l,k+l}$ we observe that $$[\pi^{j+1,j+3},\pi^{l,k+l}]=\left\{\begin{array}{cc}\pi^{j+1,k+j+3} & \textrm{ if }l=j+3,\\ -\pi^{j+1-k,j+3}& \textrm{ if }k+l=j+1,\\ 0& \textrm{ otherwise, }\end{array}\right.$$ where both non-vanishing options cannot simultaneously occur (since $k\neq-2$). Similarly one has $$[\pi^{j,j+2},\pi^{l,k+l}]=\left\{\begin{array}{cc}\pi^{j,j+k+2}& \textrm{ if }l=l+2,\\ -\pi^{j-k,j+2} & \textrm{ if }k+l=j,\\ 0 & \textrm{ otherwise. }\end{array}\right.$$ Putting together the last three equations, one has $$\big[[F,E^3],E^k\big]=\sum_{j=1}^{d-3}f(j)\big(\pi^{j+1,k+j-3}-\pi^{j+1-k,j+3}\big)-f(j+2)\big(\pi^{j,j+k+2}-\pi^{j-k,j+2}\big).$$ We show then that the coefficient of $[[F,E^3],E^k]$ in  $\pi^{1,k+3}$ is non-zero. Indeed it is $$-f(3)+f(k+3)-f(k)=-6k\neq0$$ if $k\leq d-4$ or $-2f(3)\neq0$ if $k=d-3.$\end{proof}

\begin{prop}[Adjoint Factors]\label{adfactors} \item
\begin{itemize}\item[-] Let $\ge=\so(n,n+1)$ or $\sp(2n,\R)$ and $\phi:\ge\to\sl(d,\R)$ be the defining representation. Then as an $\ad(\phi\,\ge)$-module one has $$\sl(d,\R)=\rep_{2\peso_{\aa_1}}\oplus\rep_{\peso_{\aa_2}}.$$ These two factors also correspond to the decomposition of $\sl(d,\R)$ in odd versus even exponents. Moreover, for each factor one has $\vt=\simple$ and for any $X_0\in\a^+$ the cone of $(\rep,X_0)$-compatible elements is $\encono_{\rep}=\a^+$.

\item[-] Let now $\gedossplit$ be a real-split form of the exceptional complex Lie algebra of type $\Ge$ and let $\phi:\gedossplit\to\sl_7(\R)$ be the fundamental representation associated to the short root. Then $\phi(\gedossplit)$ has three adjoint factors given by  $V_1\oplus V_5$, $V_3$ and $V_2\oplus V_4\oplus V_6$ and for each factor the same conclusion as in the previous item holds.

\end{itemize}
\end{prop}

\begin{proof} We focus on the first item, the second following similarly but with more involved computations that we omit. Using the computation for the exponents of $\ge$ in Table \ref{exponentes} one sees that $\phi(\ge)=\sum_{\mathrm{odd}\ e}V_e$ is an irreducible factor. Moreover, as $[F,E^3]\in\phi(\ge)$, Lemma \ref{permutaexponentes} implies that all even exponents belong to the same irreducible factor, giving the result. The second statement follows readily as, by direct computation, the weights of $\ad(\phi)$ are integer multiples of simple roots of $\ge$ and every root appears as a weight in each of the factors.  The $\gedossplit$ case follows by explicit verification. See Figure \ref{hasseEx}. \end{proof}



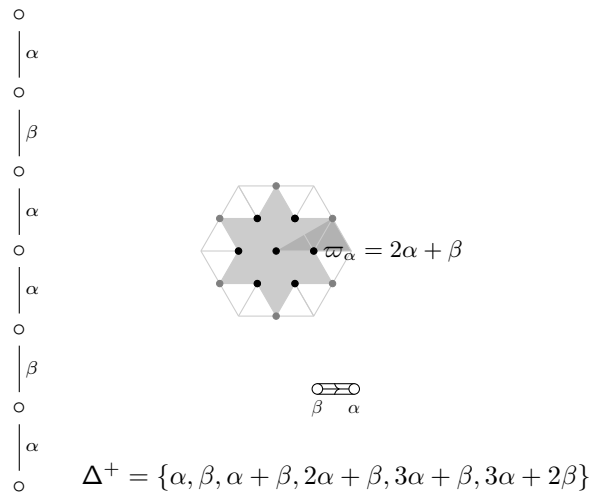
\begin{figure}
\begin{tikzpicture}
\node[scale=1] at (0,0){
\begin{tikzcd}[column sep=small]
 & \circ \arrow[d, "\aa"]& & \\ 
  & \circ \arrow[d, "\bb"]& & \\
 & \circ\arrow[d, "\aa"]&  & \\
 & \circ \arrow[d, "\aa"]& & \\ 
  & \circ \arrow[d, "\bb"]& & \\
 & \circ\arrow[d, "\aa"]&  & \\
 & \circ  & &
\end{tikzcd}};

\node[scale=1] at (4,-2){\dynkin[labels={\bb,\aa},scale=1.4] G2};
\node[scale=1] at (4,-3){$\simple^+=\{\aa,\bb,\aa+\bb,2\aa+\bb,3\aa+\bb,3\aa+2\bb\}$};

\node[scale=1] (b) at (4,0) {\begin{tikzpicture}\begin{rootSystem}{G}
\node[right] at (hex cs:x=1,y=0){\small\(\peso_\aa=2\aa+\bb\)};
\roots

\wt [black]{1}{0}
\wt [black]{-1}{1}
\wt [black]{-1}{0}
\wt [black]{0}{0}
\wt [black]{1}{-1}
\wt [black]{-1}{1}
\wt [black]{2}{-1}
\wt [black]{-2}{1}
\WeylChamber
\end{rootSystem}\end{tikzpicture}};
\end{tikzpicture}
\caption{Hasse diagram for the $7$-dimensional irreducible representation of $\Gedossplit$, which is the fundamental representation of the short root, together with the corresponding set of weights (in black).
}\label{hasseEx}
\end{figure}

\subsection{Cleaner formulae for $\kt^e$}

We proceed to a more explicit computation of $\ktv^e$. To this end, consider the \emph{(finite) difference operator} defined, for $g:\R\to\R$, by $$\diff g(x)=g(x+1)-g(x).$$ We also consider, for a real number $z\in\R$ (a slight modification of) the \emph{falling factorial} notation: for $k\in\N$ we let $$z^{\lf k}=z(z-1)\cdots (z-k+1),$$ with the convention that $z^{\lf 0}=1,$ in particular $0^{\lf 0}=1$, and for later use we define $z^{\lf{-k}}=0.$ For a function $g,$ we let $g(x)^{\lf k}$ be the $k$-th falling factorial applied to the real number $g(x).$ Straightforward computations yield the following rules.

\begin{lemma}\label{cuenta} For every $k\in\N$ one has\begin{itemize}\item[i)] ${\displaystyle\diff^k g(x)=\sum_{i=0}^k(-1)^i\binom k ig(x+k-i);}$\item[ii)] a Leibnitz rule ${\displaystyle\diff^k(gh)(x)=\sum_{i=0}^k\binom ki\diff^ig(x)\diff^{k-i}h(x+i)};$ 
\item[iii)] $\diff x^{\lf k}=kx^{\lf{k-1}}$ and if we let $r(x)=l-x$ for some $l\in\R,$ then $$\diff \big(r(x)\big)^{\lf k}=-kr(x+1)^{\lf{k-1}}.$$

\end{itemize}\end{lemma}

Considering the function $g_{d,e}(x)=d+e-x,$ together with \begin{equation}\label{Fde}F_{d,e}(x)=f(x-e)f(x-e+1)\cdots f(x-1)=(x-1)^{\lf e}\cdot\big(g_{d,e}(x)\big)^{\lf e},\end{equation} Proposition \ref{formulaexponentes} yields $$\ktv^e  =\sum_{j=1}^d\pi^j\Big(\sum_{t=0}^e(-1)^{t}\binom{e}{t}F_{d,e}(j-t+e)\Big) =\sum_{j=1}^d\pi^j\diff^eF_{d,e}(j),$$ where the last equality comes from Lemma \ref{cuenta}. We compute then $\diff^eF_{d,e}(x)$ using the Leibnitz rule applied to the product $F_{d,e}(x)=(x-1)^{\lf e}\cdot(g_{d,e}(x))^{\lf e}.$
\begin{alignat*}{2}\diff^eF_{d,e}(x) & =\sum_{i=0}^e\binom{e}{i} \diff^i(x-1)^{\lf e}\diff^{e-i}(g_{d,e}(x+i))^{\lf{e}}\\ &= \sum_{i=0}^e(-1)^{e-i}\binom{e}{i} \frac{e!}{(e-i)!}(x-1)^{\lf {e-i}}\frac{e!}{i!}\big(g_{d,e}(x+e)\big)^{\lf i}\\ &= (-1)^ee!\sum_{i=0}^e(-1)^{i}\binom{e}{i}^2 (x-1)^{\lf {e-i}}\big(g_{d,e}(x+e)\big)^{\lf i}.\end{alignat*}

In order to decide whether $\ktv^e$ belongs to the kernel of a simple root $\slroot_j\in\simple$ we compute $-\slroot_j(\ktv^e)=\diff^{e+1}F_{d,e}(j),$ which we write, by Lemma \ref{cuenta}, as 
\begin{alignat*}{2}\diff^{e+1}F_{d,e}(x) & =\sum_{i=0}^{e+1}\binom{e+1}{i} \diff^i(x-1)^{\lf e}\diff^{e+1-i}(g_{d,e}(x+i))^{\lf{e}}\\ &= \sum_{i=0}^{e+1}(-1)^{e+1-i}\binom{e+1}{i} \frac{e!}{(e-i)!}(x-1)^{\lf {e-i}}\frac{e!}{i-1!}\big(g_{d,e}(x+e+1)\big)^{\lf {i-1}}\\ &= (-1)^{e+1}(e+1)!\sum_{i=1}^e(-1)^{i}\binom{e}{i}\binom{e}{i-1} (x-1)^{\lf {e-i}}\big(d-1-x\big)^{\lf {i-1}}.\end{alignat*}

We record the above computations in the following lemma.

\begin{lemma}\label{calculo} One has \begin{alignat}{2}\ktv^e & =(-1)^ee!\sum_{j=1}^d\pi^j\Big(\sum_{t=0}^e(-1)^t\binom{e}{t}^2(j-1)^{\lf{e-t}}(d-j)^{\lf t}\Big);\nonumber\\ \label{diofanto}\slroot_j(\ktv^e) & =(-1)^e(e+1)!\sum_{t=1}^e(-1)^{t}\binom{e}{t}\binom{e}{t-1} (j-1)^{\lf {e-t}}\big(d-1-j\big)^{\lf {t-1}}.\end{alignat}
\end{lemma}

\begin{obs}\label{peso1} In particular one has $\peso_1(\ktv^e)=e!(d-1)^{\lf e}>0$.
\end{obs}

\begingroup
\renewcommand{\arraystretch}{1.13}

\begin{table}[h!]
\begin{center}
\begin{tabular}{r|r|r|r|r} 
& $\sl_3(\R)$ & $\sl_4(\R)$ & $\sl_5(\R)$ & $\sl_6(\R)$  \\
\hline
$\ktv^1$ & $(2,0,-2)$ & $(3,1,-1,-3)$ & $(4,2,0,-2,-4)$ & $(5,3,1,-1,-3,-5)$\\\hline
$\ktv^2$ & $(4,-8,4)$ &  $12\cdot(1,-1,-1,1)$ &$12\cdot(2,-1,-2,-1,2)$ & $8\cdot(5,-1,-4,-4,-1,5)$\\\hline
$\ktv^3$ &  &  $36\cdot(1,-3,3,-1)$ & $144\cdot(1,-2,0,2,-1)$ & $72\cdot(5,-7,-4,4,7,-5)$\\\hline
$\ktv^4$ & & & $576\cdot(1,-4,6,-4,1)$ & $2880\cdot(1,-3,2,2,-3,1)$\\\hline
$\ktv^5$ & & & & $14400\cdot(1,-5,10,-10,5,-1)$\\\hline\hline
\end{tabular}

\ \ \\
\begin{tabular}{r|r|r} 
& $\sl_7(\R)$ & $\sl_8(\R)$   \\
\hline
$\ktv^1$ & $(6,4,2,0,-2,-4,-6)$ & $(7,5,3,1,-1,-3,-5,-7)$ \\\hline
$\ktv^2$ & $12\cdot(5,0,-3,-4,-3,0,5)$ &  $12\cdot(7,1,-3,-5,-5,-3,1,7)$ \\\hline
$\ktv^3$ & $720\cdot (1,-1,-1,0,1,1,-1) $&  $180\cdot(7,-5,-7,-3,3,7,5,-7)$ \\\hline
$\ktv^4$ &$2880\cdot (3,-7,1,6,1,-7,3) $& $2880\cdot(7, -13, -3, 9, 9, -3, -13, 7)$ \\\hline
$\ktv^5$ & $86400\cdot (1,-4,5,0,-5,4,-1)$& $43200\cdot(7, -23, 17, 15, -15, -17, 23, -7)$\\\hline
$\ktv^6$ & $518400\cdot (1,-6,15,-20,15,-6,1)$ & $3628800\cdot(1, -5, 9, -5, -5, 9, -5, 1)$\\\hline
$\ktv^7$ & & $25401600\cdot(1, -7, 21, -35, 35, -21, 7, -1)$
\end{tabular}

\caption{The Kostant vectors of $\sl_d(\R)$ for $d\in\lb3,8\rb$.} \label{ejemplosKV}
\end{center}
\end{table}
\endgroup

\subsection{Type $\D$}\label{typeD}

Consider a $2n$-dimensional real vector space equipped with a bilinear form $\upomega$ of signature $(n,n)$ and let $\SO_{n,n}$ be the volume preserving automorphisms of $\upomega.$ Let also $x\mapsto x^*$ be the adjoint operator defined by $\upomega$, then $$\so(n,n)=\{x\in\sl(2n,\R):x+x^*=0\}.$$

Consider a non-isotropic line $\ell$ and its orthogonal complement $\ell^\perp$ for $\upomega$.

We have then an $\upomega$-preserving involution $i$ with $i|\ell=-\id$ and $i|\ell^\perp=\id$, which gives an involution $\undi:\SO_{n,n}\to\SO_{n,n}$ defined by $g\mapsto igi$. The group of fixed points of $\undi$ is the subgroup of $\SO_{n,n}$ that stabilizes $\ell$. For $g\in\Fix\undi$, the restriction $g\mapsto g|\ell^\perp$ gives an isomorphism of $(\Fix\undi)_0$ with a special orthogonal group of signature $(n-1,n)$.

The differential $d_{\scr e}\undi:\so_{n,n}\to\so_{n,n}$ coincides with $x\mapsto ixi$ and is a Lie-algebra involution giving a decomposition \begin{equation}\label{adjnn}\so_{n,n}=\Lie(\Fix\undi)\oplus\{x\in\so_{n,n}:\mathrm d_\scr e\undi(x)=-x\}.\end{equation} If $x\in\sl(2n,\R)$ is anti-fixed by $d_e\undi$ then one readily observes that $x(\ell)\subset\ell^\perp$ and $x(\ell^\perp)\subset\ell$. We can easily describe then the anti-fixed subspace of $\so(n,n)$ as  $$\{x\in\so_{n,n}:\mathrm d_\scr e\undi(x)=-x\}=\big\{x-x^*:x\in\hom(\ell,\ell^\perp)\big\}.$$ It is a $2n-1$-dimensional vector space and an irreducible $\Lie(\Fix\undi)$-module.

If $\s$ is a principal $\sl_2$ of $\so_{n,n}$ then it stabilizes a non-isotropic line, which we can assume to be $\ell$, and acts irreducibly on $\ell^\perp$, so we conclude that $$V_{n-1,\mathrm a}:=\{x:\mathrm d_e\undi(x)=-x\}$$ is an irreducible $\s$-factor of dimension $2(n-1)+1$. 

The $\mathrm a$ in the notation solves an ambiguity issue when $n$ is even. Indeed, observe from Table \ref{exponentes} that two situations occur for $\so_{n,n}.$ If $n$ is odd, the exponent $n-1$ occurs with multiplicity one and is the only even exponent of $\so_{n,n}.$ However if $n$ is even, there are two $\s$-adjoint factors of dimension $2(n-1)+1.$ One of these factors is contained in $\Fix\mathrm d_\scr e\undi,$ and the other one is $V_{n-1,\mathrm a}$.

Let $\a_{\so_{n,n}}=\R^n$ be a Cartan subspace of $\so_{n,n}$ and consider the set of simple roots $\simple=\{\slroot_1,\ldots,\slroot_{n-1},\upalpha_n\}$ where $\slroot_i(a)=a_i-a_{i+1}$ and  $\upalpha_n(a)=a_{n-1}+a_{n}.$ We can choose a Cartan subspace $\a_{\so_{n-1,n}}$ of $\Fix\undi$ that is embedded in $\a_{\so_{n,n}}$ as $\{a\in\R^n:a_n=0\}.$ The involution $\mathrm d_\scr e\undi$ acts on $\a_{\so_{n,n}}$ as $$\ii:=(a_1,\ldots,a_n)\mapsto(a_1,\ldots,-a_n).$$ It sends $\slroot_{n-1}$ to $\upalpha_n$ and fixes the other roots so it is the opposition involution $\ii$ of $\a_{\so_{n,n}}$. Moreover, the Kostant line associated to the anti-fixed factor is
$$\kt^{n-1,\mathrm a}:=\R\cdot(0,\ldots,0,1).$$ The other Kostant lines are those of $\so_{n-1,n}$ inside $\a_{\so_{n,n}}$ via the above inclusion.

\subsubsection{Triality}\label{tricota} 

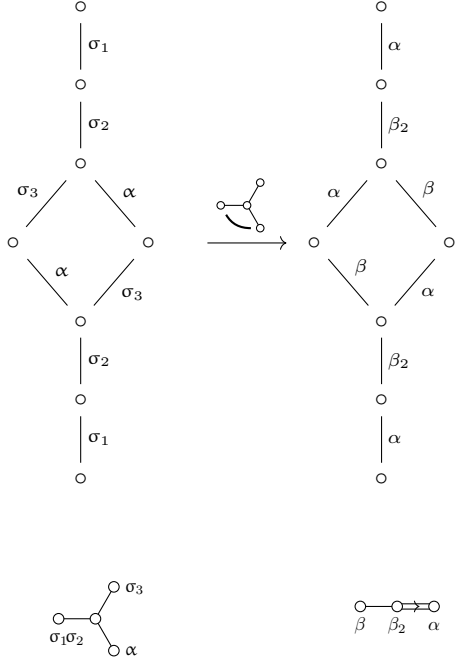
\begin{figure}
\begin{tikzpicture}[scale=.9]
\node[scale=1] (c) at (2,0){
\begin{tikzcd}[column sep=small]
 & \circ \arrow[d, dash, "\slroot_1"]& & \\ 
 & \circ\arrow[d, dash, "\slroot_2"]&  & \\
 & \circ \arrow[dl, dash,swap,"\slroot_3"] \arrow[dr, dash,"\upalpha"]  & &\\
  \circ  \arrow[dr, dash,"\upalpha"] &  & \arrow[dl, dash,"\slroot_3"] \circ  & \\
  & \circ  \arrow[d, dash,"\slroot_{2}"]&   &   \\
  & \circ\arrow[d, dash,"\slroot_1"]& & \\
 & \circ &  &
\end{tikzcd}};

\node[scale=1] at (2,-5){\dynkin[labels={\slroot_1,\slroot_2,\slroot_3,\upalpha},scale=1.4] D4};

\node[scale=1] (d) at (6,0){
\begin{tikzcd}[column sep=small]
 & \circ \arrow[d, dash, "\aa"]& & \\ 
 & \circ\arrow[d, dash, "\bb_2"]&  & \\
 & \circ \arrow[dl, dash,swap,"\aa"] \arrow[dr, dash,"\bb"]  & &\\
  \circ  \arrow[dr, dash,"\bb"] &  & \arrow[dl, dash,"\aa"] \circ  & \\
  & \circ  \arrow[d, dash,"\bb_{2}"]&   &   \\
  & \circ\arrow[d, dash,"\aa"]& & \\
 & \circ &  &
\end{tikzcd}};

\node[scale=1] at (6,-5){\dynkin[labels={\bb,\bb_2,\aa},scale=1.4] B3};

\node[scale=1] at (4,0.5) {$\begin{dynkinDiagram}D{4}\draw[thick] (root 1) to [out=-60, in=180] (root 4);\end{dynkinDiagram}$}; 
\draw[scale=1, ->] (c) -- (d);

\end{tikzpicture}
\caption{The irreducible representation $\so(3,4)\to\so(4,4)$.}\label{d4b3}
\end{figure}

We now deal with the special case $\D_4$. In this special case the Dynkin diagram has an order three automorphism $\uptau$ that fixes $\slroot_2$ and $\slroot_1\mapsto\slroot_3$, $\slroot_3\mapsto\upalpha_4$ and $\upalpha_4\mapsto\slroot_1$, see Equation \eqref{uptau}.

\begin{equation}\label{uptau}\begin{dynkinDiagram}[scale=1.5]D{4}
\node[left,/Dynkin diagram/text style] at (root 1){\(\slroot_{1}\)};
\node[above right,/Dynkin diagram/text style] at (root 3){\(\slroot_{3}\)};
\node[below right,/Dynkin diagram/text style] at (root 4){\(\upalpha_{4}\)};
\draw[thick,->] (root 4) to [out=180, in=-60] (root 1);
\draw[thick,->] (root 1) to [out=60, in=180] (root 3);
\draw[thick,->] (root 3) to [out=-60, in=60] (root 4);
\end{dynkinDiagram}\end{equation}

This automorphism can be realized as the orthogonal transformation $$\uptau:=\frac12\left(\begin{smallmatrix} 1 & 1 & 1 &-1\\1& 1& -1& 1\\1&-1&1&1\\1&-1&-1&-1\end{smallmatrix}\right):\a_{\so_{4,4}}\to\a_{\so_{4,4}}.$$ This automorphism can also be extended to an external automorphism of $\underline{\uptau}:\so_{4,4}\to\so_{4,4}$ whose fixed point set $\Fix\uptau=\gedosl$. Moreover, the involution $\uptau\ii\uptau^{-1}$ of $\a_{\so_{4,4}}$ has fixed-point set the image $\rep_\aa(\so_{3,4})$ of the fundamental representation for the short root of $\so_{3,4}$, see Figure \ref{d4b3}. The adjoint factors of $\rep_\aa(\so_{3,4})$ are then $$\so_{4,4}=\big(\underline\uptau(V_1\oplus V_3\oplus V_5)\big)\oplus(\underline\uptau(V_{3,\mathrm a})).$$ Using the explicit formula for $\uptau:\a_{\so_{4,4}}\to\a_{\so_{4,4}}$ one has \begin{equation}\label{ktresa} \underline\uptau(V_{3,\mathrm a})\cap\a_{\so_{4,4}}=\uptau(\kt^{3,\mathrm a})=\R\cdot(-1,1,1,-1).\end{equation}
This last equation will be needed in the proof of Theorem \ref{degLietheoretic}.

\section{Pressure degenerations are Lie-theoretic}\label{seccionPresion}

In this section we prove the following. Recall that for Fuchsian $\prin\in\hitchin_g(S)$ we let $$\tangcero\prin e=H^1_{\Ad\prin}(\pi_1S,V_e).$$

\begin{thm}\label{degLietheoretic} Let $\ge$ be simple split  of type $\A$, $\B$, $\Ce$, $\D$ or $\Ge$. Consider $\rho\in\hitchin_\ge(S)$ and a length functional $\length\in\conodual\rho$ then, the pressure form $\PP^\length_\rho$ is degenerate at $\tangente\in\sf T_\rho\hitchin_\ge(S)$ if and only if either of the following situations hold:\begin{itemize}

\item[-] $\rho$ is Fuchsian and $${\displaystyle \tangente\in\bigoplus_{e:\length(\kt^e)=0}\tangcero\prin e},$$

\item[-] $\rho$ is self dual, $\length$ is $\ii$-invariant and $\tangente$ is $\overline\ii$-anti-invariant. 

\item[-] $\ge$ is of type $\sf A_6$ or $\Ce_3$, the Zariski closure of $\rho(\pi_1S)$ is $\gedosl$, $\length(\kt^3)=0$ and $\tangente\in H^1_{\Ad\rho}(\pi_1S,V_3)$.

\item[-] $\ge$ is of type $\sf D_4$, the Zariski closure of $\rho(\pi_1S)$ is conjugate to $\rep_{\aa}(\SO_{3,4})$, $\tangente\in H^1_{\Ad\rho}\big(\pi_1S,\underline\uptau (V_{3,\mathrm a})\big)$ and $\length(-1,1,1,-1)=0.$

\end{itemize}
\end{thm}

We begin the proof of Theorem \ref{degLietheoretic} with some preparation lemmas.

\subsection{Preparation Lemmas of independent interest I}

\begin{lemma}\label{interiorV} Let $\prin\in\hitchin_\ge(S)$ be Fuchsian and $e$ an exponent of $\ge$. For any non-zero $\co\in \tangcero\prin e$ and every $\length\in\conodual{\prin}$ such that $\length(\kt^e)\neq0$ the set of normalized variations $\VV\length_{\co}\subset\kt^e$ has non-empty interior.
\end{lemma}

\begin{proof} Corollary \ref{fuchscritica} and Equation \eqref{masaentropia} imply that $$0=\partial^{\log}\entropy\length{}=\length(\mass\length{\co})\in\length\big(\VV\length_{\co}\big).$$  Since $\ker\length\cap\kt^e=\{0\}$ and $\VV\length_{\co}\subset\kt^e$ by Corollary \ref{factor-cotangente}, the above equation yields then $0\in\VV\length_{\co}$. Since $\VV\length_{\co}$ is convex, if its interior where empty then  $\{0\}=\VV\length_{\co}$. Thus, for every $\g\in\pi_1S$ one has $\varjor\g{\co}=0,$ contradicting Theorem \ref{indeplambda1}.\end{proof}

\begin{lemma}\label{estrata2forma} Let $\rho\in\hitchin_{\A_{d-1}}(S)$ and $\length\in\conodual\rho$. \begin{enumerate}\item\label{parques} If $\rho(\pi_1S)$ has Zariski-closure $\SO(n,n+1)$ or $\PSp(2n,\R)$ according to the parity if $d$, or Zariski-closure $\Gedosl$ if $d=7$, and if we let $\ge$ be the corresponding Lie algebra, then $$\factor\ge=\bigoplus_{e\textrm{ even}}V_e$$ is an adjoint factor of $\rho$ and  for every non-zero $\co\in H^1_{\Ad\rho}(\pi_1S,\factor\ge)$
$$\VV{\length}_{\co}\subset\factor\ge$$ has non-empty interior. \item\label{7} If $d=7$ and the Zariski closure of $\rho(\pi_1S)$ is $\Gedosl$, then for any non-trivial cocycle $\co\in H^1_{\Ad\rho}(\pi_1S,V_3)$ the set $\VV\length_\co\subset\kt^3$ has non-empty interior and contains $\{0\}$ in its interior. In particular $\PP^\length(\co)\neq0$.\end{enumerate}\end{lemma}

\begin{proof} By Remark \ref{evenw0} $\longest$ acts trivially on even exponent spaces, so the first item is a consequence of Remark \ref{nomade} with Propositions \ref{margulisderivada} and \ref{adfactors}. To deal with the second item, we restrict ourselves to $\Gedosl<\SO(3,4)$ with Cartan subspace $\a_{\so(3,4)}=\R^3$, Weyl chamber $\{a\in\R^3: a_1\geq a_2\geq a_3\geq0\}  $ and simple roots $$\slroot_1(a)=a_1-a_2,\ \slroot_2(a)=a_2-a_3\textrm{ and }\eps_3(a)=a_3.$$ The subalgebra $\gedosl$ has Cartan subspace $\a_{\gedosl}\subset \a_{\so(3,4)}$ given by $$\a_{\gedosl}=\{(a_1,a_2,a_1-a_2):a_1,a_2\in\R\}$$ and simple roots $\{\slroot_1,\slroot_2\}$, see Figure \ref{hasseEx}.

Since $\a_{\gedosl}=\ker(\eps-\slroot_1)$, any convex combination $t\eps-(1-t)\slroot_1$ coincides with $\slroot_1$ when restricted to $\a_{\gedosl}.$ Since $\entropy{\slroot_1}\rho=1$ by Theorem \ref{PS}, it follows that the affine line $\{t\eps+(1-t)\slroot_1:t\in\R\}$ is contained in the critical hypersurface $\EuScript Q_\rho$ of $\rho:\pi_1\to\Gedosl<\SO(3,4),$ see Figure \ref{Q34}. Fix such a combination, $\varphi=(1/2)(\eps+\slroot_1)$ for example and observe, from Table \ref{ejemplosKV} that $\kt^3=\R\cdot(1,-1,-1)$ and that $\varphi(1,-1,-1)=1/2\neq0$.

If we let $\rho_t$ be tangent to $\co\in H^1_{\Ad\rho}(\pi_1S,V_3)$ then again Theorem \ref{PS} implies that $\{\eps,\slroot_1\}\subset\EuScript Q_{\rho_t},$ giving that $\entropy{\varphi}{\rho_t}$ is critical at $\rho$. We apply now Lemma \ref{identicaly0}\eqref{iiN} (with the roles of $\varphi$ and $\length$ reversed) to obtain that $\VV\length_\co\subset\kt^3$ has non-empty interior. The last statement now follows from Lemma \ref{identicaly0}\eqref{iiP=0}.
\end{proof}

\begin{figure}\centering
\begin{tikzpicture}[scale=.4]

\draw (-2,3) parabola bend (0,0) (1.5,5);
\draw (-3,5) -- (-3,-2.7) -- (3,-0.7) -- (3,7) -- cycle;
\fill[color=gray!20]
(0,0) parabola (-2,3) 

-- (0,0) -- (5,0)

-- (5,0) parabola (3,3) -- (-2,3);

\draw[thick] (3,3) parabola bend (5,0) (6.5,5);
\draw (-2,3) -- (3,3);
\draw (0,0) -- (5,0);
\draw (1.5,5) -- (6.5,5);
\draw[dashed] (-2,3)--(-3.1,3); \draw (-3.1,3) -- (-6,3);
\draw[dashed] (0,0)--(-2,0);
\draw[dashed] (1.5,5)--(0.5,5);

\node[above] at (-2,3) {$\slroot_1$};
\node[above] at (-5,3) {$\eps_3$};

\draw[fill]  (-2,3) circle [radius=0.07cm];
\draw[fill]  (-5,3) circle [radius=0.07cm];

\node at (7.4,4) {$\EuScript Q_{\rho}$};
\node at (7,-2) {$(\a_{\SO(3,4)})^*$};
\node[left] at (-3,-2) {$(\a_{\gedosl})^*$};

\fill[color=gray!60](3,3) parabola bend (5,0) (6.5,5) 
--(6.5,5) -- (1.5,5) -- (1.18,3) -- (3,3);

\draw[fill]  (1.5,5) circle [radius=0.07cm] node[above right] {$\slroot_2$};

\end{tikzpicture}
 \caption{The critical hypersurface in $\a_{\SO(3,4)}^*$ of a Hitchin representation $\rho$ whose Zariski closure is $\Gedosl$.}\label{Q34}
 \end{figure}
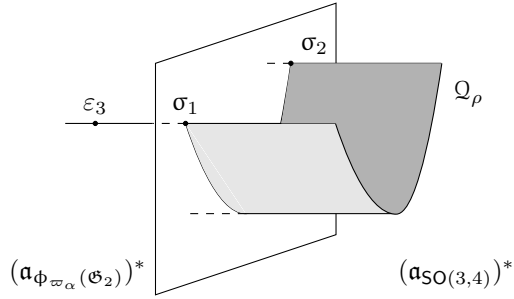

\subsection{Pressure forms: some information on lower strata}Let now $\rho:\gh\to\sf G$ be a $\simple$-Anosov representation and consider an integrable cocycle $\co\in H^1_{\Ad\rho}(\gh,\ge)$, we also denote by $\co\in\sf T_\rho\caracteres$ the associated tangent vector.

\begin{lemma}\label{identicaly0}Let $\rho\in\Anosov_\simple(\gh,\sf G)$ have semi-simple Zariski closure $\sf H$. Fix $\length\in\conodual{\rho}$ and a disjoined adjoint factor $\factor{\sf H}$ of $\sf H$. Consider an integrable cocycle $\co\in H^1_{\Ad\rho}(\gh,\factor{\sf H})$ such that there exists $\g\in\gh$ with $\varjor\g\co\neq0$, then: \begin{enumerate}
\item\label{i} If $\factor{\sf H}\cap\a\subset \ker\length,$ then $\PP^\length$ degenerates at $\co$. 
\item\label{iiN} If $\factor{\sf H}\cap\a\cap\ker\length=\{0\}$ and $\mathrm{d}\entropy\psi{}(\co)=0$ then for every $\varphi\in\conodual\rho$ with $\factor{\sf H}\cap\a\cap\ker\varphi=\{0\}$ the set $\VV{\varphi}_{\co}$ has non-empty interior and $0\in\inte\VV{\varphi}_{\co}$.

\item\label{iiP=0}  If $\factor{\sf H}\cap\a\cap\ker\length=\{0\}$ and $\PP^\length(\co)=0$ then $\VV{\length}_{\co}$ is reduced to a point and is non-zero. 

\item\label{iii} Assume $\sf H$ has rank $1$. If $\length,\varphi\in\conodual\rho$ both have kernel whose intersection with $\factor{\sf H}\cap\a$ vanishes, then there exists $c>0$ such that for all $\sf v\in H^1_{\Ad\rho}(\gh,\factor{\sf H})$ one has $\PP^{\length}_\rho(\sf v)=c\PP^{\varphi}_\rho(\sf v).$\end{enumerate}\end{lemma}

\begin{proof} Corollary \ref{factor-cotangente} implies that for all $\g\in\gh$ \begin{equation}\label{eqahora}\varjor\g{\co}\in \factor{\sf H}\cap\a.\end{equation} If we assume that $\factor{\sf H}\cap\a\subset\ker\length$ then $\VV{\length}_{\co}\subset\ker\length$ and thus Remark \ref{varphi=cte} shows degeneracy. If $\factor{\sf H}\cap\a\cap\ker\length=\{0\}$ then necessarily $\dim\factor{\sf H}\cap\a=1$ and $\VV\length_{\co}$ is an interval contained in this line (possibly reduced to a point).

Let us deal now with item \eqref{iiN}, we first establish the result for $\length$ and deal afterwards with the general case. Since $\mathrm{d}\entropy\psi{}(\co)=0$ Eq. \eqref{masaentropia} gives that $\length(\mass\length{\co})=0$, however $\mass\length{\co}\in\VV\length_{\co}\subset\factor{\sf H}\cap \a$ which only intersects $\ker\length$ at $\{0\}$. We conclude that $$\mass\length{\co}=0\in\VV\length_{\co}.$$ 

If $\VV\length_{\co}=\{0\}$ then for all $\g\in\gh $ $\varjor\g\co=0$ contrary to our assumptions. We obtain thus that $\VV\length_{\co}$ is an interval with non-empty interior. This implies in turn that $\length(\vec\ledrappier)$ and $\length(\ledrappier)$ are not Liv\v sic-cohomologous and thus that $0=\mass\length{\co}\in\inte\VV\length_{\co}.$ This gives item (\ref{iiN}) for $\length$ but also gives a bit more: there exists $\g,h\in\gh$ such that \begin{equation}\label{decadalado}\length(\varjor\g\co)<0 < \length(\varjor h\co).\end{equation}

\noindent If $\varphi\in\conodual\rho$ is such that $\factor{\sf H}\cap \a\cap\ker\varphi=\{0\}$, then there exists $c\neq0$ with$\varphi|\factor{\sf H}\cap \a=c\length|\factor{\sf H}\cap \a$. Assume that $c>0$, then Equation \eqref{decadalado} yields $$\varphi(\varjor\g\co)<0 < \varphi(\varjor h\co),$$ which implies that $\varphi(\VV\varphi_{\co})$ is an interval with $0$ in its interior, giving the result.

We now deal with item \eqref{iiP=0}. If $\PP^\length(\co)=0$ then Remark \ref{varphi=cte} implies that $\VV{\length}_{\co}$ is contained in a level set of $\length,$ and is thus a point. If it where zero, for every $\g\in\gh$ we have $\varjor\g{\co}=0$. Since we assumed this was not the case, $\VV\length_{\co}\neq\{0\}.$

We now focus on item \eqref{iii}, se we assume $\sf H$ has rank $1$. As before $\factor{\sf H}\cap\a$ is one-dimensional. Then one has: \begin{itemize}\item[-] there exists $c>0$ such that for every $u\in \a_\sf H$ one has $\varphi(u)=c\length(u)$, \item[-] there exists $C\neq0$ such that for every $v\in\factor{\sf H}\cap \a$ one has $\varphi(v)=C\length(v)$.
\end{itemize}

Consequently, for every $\sf v\in H^1_{\Ad\rho}(\gh,\factor{\sf H})$ Equation \eqref{eqahora} implies that $\varphi(\vec\ledrappier_\sf v)=C\length(\vec\ledrappier_\sf v)$ and $\varphi(\ledrappier_\rho)=c\length(\ledrappier_\rho)$. Moreover, Corollary \ref{derivadaentropia} implies then that $$\deriva t0 \entropy\varphi{\rho_t}\varphi(\ledrappier_{\rho_t})=\frac{C}{c}\deriva t0 \entropy\length{\rho_t}\length(\ledrappier_{\rho_t}).$$ Thus,
\begin{alignat*}{2}\PP^{\varphi}_\rho(\sf v)& =\frac{\Var\Big(\deriva t0 \entropy\varphi{\rho_t}\varphi(\ledrappier_{\rho_t}),m_{-\entropy{\varphi}{\rho}\varphi(\ledrappier)}\Big)}{\entropy{\varphi}\rho\int\varphi(\ledrappier_\rho)dm_{-\entropy{\varphi}{\rho}\varphi(\ledrappier)}}\\ & =\frac{\Var\Big(\frac{C}{c}\deriva t0 \entropy\length{\rho_t}\length(\ledrappier_{\rho_t}),m_{-\entropy{\length}{\rho}\length(\ledrappier)}\Big)}{\entropy{\length}\rho\int\length(\ledrappier_{\rho})dm_{-\entropy{\length}{\rho}\length(\ledrappier)}}=  \frac{C^2}{c^2}\PP^\length_\rho(\sf v).\end{alignat*} \end{proof}

\subsection{Pressure forms at the Fuchsian locus I} For a Fuchsian $\prin$ we have

\begin{equation}\label{decoortogonal}\sf T_\rho\hitchin_{\A_{d-1}}(S)=\bigoplus_{e\textrm{ exponent}}\tangcero\prin e.\end{equation}

Let $e\in\lb1,d-1\rb$ and let $q\in H^0(K^{e+1})$ be a holomorphic differential of degree $e+1$ on the Riemann surface associated to $\prin.$ The Hitchin parametrization provides a \emph{normalized deformation} $\nrm(q)\in\sf T_\prin\hitchin_{\A_{d-1}}(S)$ as in Labourie-Wentworth \cite{Labourie-Wentworth}, moreover \cite[Corollary 3.5.2]{Labourie-Wentworth} implies that if we let $\coclase_{\nrm(q)}\in H^1_{\Ad\prin}\big(\pi_1S,\sl_d(\R)\big)$ be the associated cocycle then $\coclase_{\nrm(q)}\in\tangcero\prin e,$ and \cite[Proposition 6.5.7]{Labourie-Wentworth} states that if $p\in H^0(K^{f+1})$ with $f\neq e$ then \begin{equation}\label{horto}\PP^{\peso_1}_\prin\big(\nrm (q),\nrm (p)\big)=0.\end{equation}

We can now establish the following.


\begin{lemma}\label{ortogonalfuchs} Let $\ge$ have type $\A$, $\B$, $\Ce$, $\D$, or $\Ge$. Let $\prin\in\hitchin_\ge(S)$ be a Fuchsian representation, then for every $\length\in\conodual\prin$ and exponents $e\neq f$ the subspaces $\tangcero\prin e$ and $\tangcero\prin f$ are $\PP^\length_\prin$-orthogonal.
\end{lemma}

\begin{proof} Since $\kt^e$ is $1$-dimensional and does not lie in $\ker\peso_1$ (Remark \ref{peso1}), there exist $c_e\in\R-\{0\}$ such that $\length|\kt^e=c_e\peso_1|\kt^e.$

Consider now $\co\in \tangcero{\prin}{e}$ and $\sf v\in \tangcero{\prin}{f}$, and write $\coclase=\coclase_{\nrm (q)}$ for some holomorphic differential $q$ and similarly for $\sf v$ and a differential $p$. Corollary \ref{factor-cotangente} implies that $\vec\ledrappier_\co$ has values in $\kt^e$ and $\vec\ledrappier_\sf v$ has values in $\kt^f$, and thus \begin{alignat*}{2}\length(\vec\ledrappier_\co)& =c_e\peso_1(\vec\ledrappier_\co),\\ \length(\vec\ledrappier_\sf v)& =c_f\peso_1(\vec\ledrappier_\sf v).\end{alignat*}

\noindent
Moreover, since $\prin$ is Fuchsian the argument of Lemma \ref{identicaly0}\eqref{iii} yields, by Eq \eqref{horto} $$\PP^\length_\prin\big(\nrm (p),\nrm (q)\big)=\frac{c_ec_f}{c_1^2}\PP^{\peso_1}_\prin\big(\nrm(p),\nrm (q)\big)=0.$$

This deals with types $\A,\,\B,\,\Ce$ and $\Ge$, and for all exponents for the type $\D_k$ except $\ktv^{k-1,\mathrm{a}}$. However by \S\,\ref{typeD} the decomposition $H^1_{\Ad\prin}(\pi_1S,\so(k-1,k))\oplus H^1_{\Ad\prin}(\pi_1S,V_{k-1,\mathrm a})$ consists on fixed and anti-fixed point of the involution $\mathrm d_\prin\undiX$, which is an isometry of $\PP^\length$ by Lemma \ref{tauisom}, yielding the result.\end{proof}

\subsection{Proof of Theorem \ref{degLietheoretic}}By means of Theorem \ref{clausuras}, we proceed with an analysis according to the Zariski closure of $\rho(\pi_1S)$. If $\rho(\pi_1S)$ is Zariski-dense then Corollary \ref{nondeg} implies that every $\length\in\conodual\rho$ induces a Riemannian $\PP^\length$.

At the other end, if $\rho=\prin$ is Fuchsian then we have \begin{equation}\label{decom}\sf T_\prin\hitchin_\ge(S)=\bigoplus_{e\textrm{ exponent}}\tangcero\prin e.\end{equation} 

If follows from Lemma \ref{ortogonalfuchs}  that, for type $\A$, $\B$, $\Ce$, $\D$ and $\Ge$, the above decomposition is orthogonal for every pressure form $\PP^\length_\prin$. So we study each  $\tangcero\prin e$.

Item \eqref{i} from Lemma \ref{identicaly0} implies that $\PP^\length$ degenerates on every $\tangcero\prin e$ with $\kt^e\subset\ker\length$. We have to show thus non-degeneracy of $\PP^\length$ on the adjoint factors with $\length(\ktv^e)\neq0$. Let $e$ be such that this happens. 

Lemma \ref{interiorV} states that the set of normalized variations $\VV{\length}_\tangente$ has non-empty interior. If $\PP^\length$ degenerates in $\tangcero\prin e$ then Lemma \ref{identicaly0}\eqref{ii} states that $\VV\length_{\tangente}$ is reduced to a point, yielding thus a contradiction. This concludes the Fuchsian points.

We deal now with type $\sf A$ and intermediate strata, i.e. $\rho(\pi_1S)$ has Zariski-closure either $\SO(n,n+1)$ or $\PSp(2n)$ according to the parity of $d$ (we deal later with the $\Gedossplit$-case), let $\ge$ be the associated Lie algebra. In this situation, by Proposition \ref{adfactors}, $\ge=\bigoplus_{e \textrm{ odd}}V_e$ and $\factor\ge=\bigoplus_{e\textrm{ even}}V_e$ are the two adjoint factors.

The first factor is settled by Corollary \ref{nondeg}, so we focus on the latter. Lemma \ref{estrata2forma} states that for any $\length\in\conodual\rho$ the set $\VV\length_\tangente$ has non-empty interior. So the only possibility for $\VV\length_\tangente$ to be contained on a level set if $\length$ is that $\factor\ge\subset\ker\length$, or equivalently $\length$ is $\ii$-invariant. Thus if $\length$ is not $\ii$-invariant then both restrictions $\PP^\length|H^1_{\Ad\rho}(\pi_1S,\ge)$ and $\PP^\length|H^1_{\Ad\rho}(\pi_1S,\factor\ge)$ give definite pressure forms. To conclude it suffices to show that $$H^1_{\Ad\rho}(\pi_1S,\ge)\perp_{\PP^\length}H^1_{\Ad\rho}(\pi_1S,\factor\ge).$$ However, by Corollary \ref{tauisom} the involution $\undiX$ is an isometry of $\PP^\length$ and the decomposition above coincides with the decomposition $$\sf T_\rho\hitchin_{\A_{d-1}}(S)=\Fix d_\rho\undiX\oplus\Fix(-d_\rho\undiX),$$ thus the decomposition is $\PP^\length$-orthogonal, giving non-degeneracy on $\sf T_\rho\hitchin_{\A_{d-1}}(S).$

We finally deal with the case where $\rho(\pi_1S)$ has Zariski-closure $\Gedosl$, this is the most involved case. By Proposition \ref{adfactors} there are three adjoint factors \begin{alignat*}{2}\sl_7(\R) & =V_3\oplus\big(V_1\oplus V_5)\oplus\big(V_2\oplus V_4\oplus V_6)\\ &= V_3\oplus\rep_{\peso_\aa}(\frak G_2)\oplus \big(V_2\oplus V_4\oplus V_6). \end{alignat*}

By Corollary \ref{nondeg} $\PP^\length$ is non-degenerate on deformations along $V_1\oplus V_5=\gedosl$. The factors $V_2\oplus V_4\oplus V_6$ and $V_3$ are dealt with in Lemma \ref{estrata2forma} items \eqref{parques} and \eqref{7} respectvely. This deals with the adjoint factors individually. As in the previous paragraph, the $H^1$ associated to even exponents and the $H^1$ associated to odd exponents are $\PP^\length$-orthogonal, so to prove non-degeneracy it remains to understand the restriction of $\PP^\length$ to $H^1_{\Ad\rho}(\pi_1S,V_3)\oplus H^1_{\Ad\rho}(\pi_1S,V_1\oplus V_5)$. As each factor has already been dealt with, we consider non-vanishing $\co_3$ and $\co_{1,5}$ in $H^1_{\Ad\rho}(\pi_1S,V_3)$ and $H^1_{\Ad\rho}(\pi_1S,V_1\oplus V_5)$ respectively and study the deformation associated to $$\co=\co_3+\co_{1,5}\in H^1_{\Ad\rho}(\pi_1S,V_3)\oplus H^1_{\Ad\rho}(\pi_1S,V_1\oplus V_5).$$ We have to show that any $\length\in\conodual\rho$ verifies $\PP^\length(\co)\neq0$.

In this situation, we can restrict ourselves to $\SO(3,4)$ with Cartan subspace $\a_{\so(3,4)}=\R^3$, Weyl chamber $\{a\in\R^3: a_1\geq a_2\geq a_3\geq0\} $ and simple roots $$\slroot_1(a)=a_1-a_2,\ \slroot_2(a)=a_2-a_3\textrm{ and }\eps_3(a)=a_3.$$ The subalgebra $\gedosl$ has Cartan subspace $\a_{\gedosl}\subset \a_{\so(3,4)}$ given by $$\a_{\gedosl}=\{(a_1,a_2,a_1-a_2):a_1,a_2\in\R\}$$ and simple roots $\{\slroot_1,\slroot_2\}$, see Figure \ref{hasseEx}. The proof of non-degeneracy is split into:

\begin{enumerate}\item\label{sl2in} $\slroot_2\in\supp\length$, \item\label{sl2not} $\slroot_2\notin\supp\length$.
\end{enumerate}

We deal first with item \eqref{sl2in}, the proof uses Lemma \ref{intermedioModif} applied to $\length=\varphi$, so we assume by contradiction that $\PP^\length(\co)=0$, or equivalently that $\length(\vec\ledrappier)$ and $c\length(\ledrappier)$ are Liv\v sic-cohomologous for some $c\neq0$. By Theorem \ref{PS} $\entropy{\slroot_1}\rho=\entropy{\slroot_2}\rho=1$ so Theorem \ref{JJ} implies that there exists $\g\in\pi_1S$ such that $$\slroot_2(\rho\g)<\slroot_1(\rho\g)=\eps_3(\rho\g).$$ Since by assumption $\slroot_2\in\supp\length$ we have $\slroot_2$ strictly minimizes $\rho(\g)$ among $\supp\length$ and we can apply Lemma \ref{intermedioModif} (with $\aa=\slroot_2$). Moreover, by Proposition \ref{todopar}, every non-commuting pair $g,h\in\pi_1S$ has images $\rho(g),\rho(h)$ that are $\simple$-transversally proximal. By Benoist \cite[Proposition 5.1]{limite}, and because $\ii=\id$ in this case, we find a subgroup $\G'<\pi_1S$ such that $\rho(\G')\subset\Gedosl$ is Zariski-dense, $\simple$-Anosov and has limit cone contained in $\{a:\slroot_2(a)<\slroot_1(a)\}$. So applying Lemma \ref{intermedioModif} we see that \begin{equation}\label{kersl2}\jordanlim_{\tangente|\G'}\subset\ker\slroot_2.
 \end{equation} 
 
However, the representation $\Ad_{\SO(3,4)}|\Gedosl$ is disjoined, indeed it has only two factors and one of them has more restricted weights than the other (recall Definition \ref{disjoined}), whence Corollary \ref{factor-cotangente} states that $\jordanlim_{\tangente|\G'}$ has non-empty interior and cannot be contained in $\ker\slroot_2.$
 
We now turn to item \eqref{sl2not}, i.e. $\slroot_2\notin\supp\length$. Since we're working in $\SO(3,4)$ we think of $\length$ as an element of $\a_{\SO(3,4)}^*$, which is spanned by the fundamental weights 
$$\peso_{\slroot_1}(a) =a_1,\qquad \peso_{\slroot_2}(a)=a_1+a_2,\qquad\peso_{\eps_3}(a)=a_1+a_2+a_3.$$Since, by assumption $\slroot_2\notin\supp\length$, up to scaling $\length$, which does not change the pressure form $\PP^\psi$, we have that for some $b\in\R$, either of the following hold: $$\length =\peso_{\slroot_1}+b\peso_{\eps_3}\,\textrm{ or }\length=b\peso_{\slroot_1}+\peso_{\eps_3}.$$ Assume the first one holds (the other is analogous), so $\length=\peso_{\slroot_1}+b\peso_{\eps_3}$.



The form $\psi$, on $\a_{\Gedosl}=\kt^1\oplus\kt^5$, verifies $$\length(a_1,a_2,a_1-a_2)=a_1+b(a_1+a_2+(a_1-a_2))=(1+2b)\peso_{\slroot_1}(a),$$ and one has moreover $\peso_{\slroot_1}|\kt^3=-\peso_{\eps_3}|\kt^3$ so, upon writing $v=v_3+v_{1,5}$ in the decomposition $\kt^3\oplus(\kt^1\oplus\kt^5)$, \begin{alignat*}{2}\length(v) & =\length(v_3)+\length(v_{1,5})= -b\peso_{\slroot_1}(v_3)+(1+2b)\peso_{\slroot_1}(v_{1,5})\\
&=\peso_{\slroot_1}\big(-bv_3+(1+2b)v_{1,5}\big). \end{alignat*}

Assuming by contradiction that there exists $c\neq0$ such that for all $\g\in\pi_1S$ one has $\length(\varjor\g\co))=c\length(\jordan(\rho\g))$ we obtain that, for all $\g$ one has $$\peso_{\slroot_1}\Big(-b\varjor\g\co_3+(1+2b)\varjor\g\co_{1,5}\Big)=c(1+2b)\peso_{\slroot_1}\big(\jordan(\rho\g)\big).$$ Considering $\co'=-b\co_3+(1+2b)\co_{1,5}$, linearity of the Margulis invariants gives $$\varjor\g{\co'}=-b\varjor\g\co_3+(1+2b)\varjor\g\co_{1,5},$$ so one has $$\peso_{\slroot_1}(\varjor\g{\co'})=c(1+2b)\peso_{\slroot_1}\big(\jordan(\rho\g)\big).$$

Since $\rho(\pi_1S)$ acts irreducibly on $\R^7$ we can apply Theorem \ref{indeplambda1} to obtain that $\co'$ is trivial, giving that either $\co_3$ or $\co_{1,5}$ is trivial which contradicts our starting assumption. This completes the proof for types $\A$, and thanks to the inclusions \eqref{ptfixes} we have also dealt with types $\B$, $\Ce$ and $\Ge$.

We end this section dealing with type $\D$, so let $\rho\in\hitchin(S,\D_n)$ with $n\geq5$ (we deal later with $\D_4$) and $\length\in\conodual\rho$.

As before we make use of the classification of Zariski closures of $\rho(\pi_1S)$. By Carvajales-Dey-Pozzetti-Wienhard \cite[Lemma 7.6]{CDPW} the Zariski-closure is semi-simple and Theorem \ref{clausuras} gives thus that $\overline{\rho(\pi_1S)}{}^\mathrm Z$ is:
\begin{itemize}
\item[-] $\SO(n,n)$, in which case Corollary \ref{nondeg} implies that $\PP^\length$ is Riemannian;
\item[-] a principal $\SL_2$, this case is dealt with by Lemmas \ref{interiorV} and \ref{ortogonalfuchs};

\item[-]a representation $\SO(n-1,n)\to\SO(n,n)$ stabilizing a non-isotropic line. 
\end{itemize}

It remains to deal with the last item. In this case we have a group involution $\undi$ of $\SO(n,n)$ whose fixed points are the corresponding $\SO(n-1,n)$. Thus, $\Ad\rho$ has two adjoint factors given by Equation \eqref{adjnn}, which are the fixed points and anti-fixed point set $d_{\scr e}\undi$; Corollary \ref{tauisom} implies that \begin{equation}\label{perpen}H^1_{\Ad\rho}(\pi_1S,\Fix d_{\scr e}\undi)\perp_{\PP^\length}H^1_{\Ad\rho}(\pi_1S,\mathrm{AntiFix}\, d_{\scr e}\undi)\end{equation}so we only have to deal with each factor independently.

One factor corresponds to deformations inside $\SO(n-1,n)$ and is settled by Corollary \ref{nondeg}, the other one has one-dimensional neutralizing space, namely $\kt^{n-1,\mathrm{a}}$ which is dealt with by means of the combination if items \eqref{iiN} and \eqref{iiP=0} of Lemma \ref{identicaly0}. Indeed, we only need to find a linear form $\varphi\in\conodual\rho$ whose entropy has vanishing derivative along a given $\co\in H^1_{\Ad\rho}(\pi_1S,V_{n-1,\mathrm{a}})$. If we let $\a_{\so_{n,n}}=\R^n$ be a Cartan subspace of $\so_{n,n}$ and consider the set of simple roots $\simple=\{\slroot_1,\ldots,\slroot_{n-1},\upalpha_n\}$ then the Cartan subspace of $\a_{\so_{n-1,n}}$ is $\ker(\slroot_{n-1}-\upalpha_n)$. Since by Theorem \ref{PS} $\simple\subset \EuScript Q_{\eta}$ for every $\eta\in\hitchin_{\D_n}(S)$, the linear form $(1/2)(\slroot_{n-1}+\upalpha_n)$ has critical entropy at $\rho$ (the argument is verbatim from the $\Gedossplit$-case in Figure \ref{Q34}), as desired.

We finally deal with $\D_4$. The above discussion holds verbatim, except that we have one more possibility for the Zariski closure of $\rho(\pi_1S)$, namely the fundamental representation for the short root of $\SO(3,4)$. From \S\,\ref{tricota} the adjoint factors of $\Ad\rho$ are deformations along $\SO(3,4)$ and cocycles with values in $\underline\uptau(V^{3,\mathrm a})$, these spaces correspond also to fixed and anti-fixed points of a Lie-algebra involution (namely $\underline\uptau\undi\underline\uptau^{-1}$), and the above discussion works verbatim, giving, by Equation \eqref{ktresa}, the resulting condition for $\length$ to have degenerate pressure form.

\section{Pressure forms at the Fuchsian locus II}\label{compatibleS}

We restrict now to $\ge=\sl(d,\R)$. Recall from Labourie-Wentworth \cite{Labourie-Wentworth} that if $\prin$ is Fuchsian and $q$ is a holomorphic differential over $S_\prin$, then there is a natural tangent vector $\nrm(q)\in\sf T_\prin\hitchin_{\A_{d-1}}(S)$, called the \emph{normalized} deformation. We will use our techniques and the results from \cite{Labourie-Wentworth} to homogenize pressure metrics on $\nrm(q).$

\subsection{Description of pressure metrics at the Fuchsian locus} Let us fix $\peso_1\in\a^*$ as reference the functional and consider $\length\in\a^*$ with $\length(\ktv^1)>0$. For each $e\in\lb1,d-1\rb$ we let $c_e\in\R$ be defined by $$\length|\kt^e:=c_e\peso_1|\kt^e.$$ Equivalently, by Remark \ref{peso1}, $c_e:=\length(\ktv^e)/(e!(d-1)^{\lf{e}}).$ The pressure form $\PP^{\length}$ is defined on  $\EuScript U_{\length }$ (recall Eq. \eqref{Upsi}) which, since $c_1>0$, contains the Fuchsian locus. The proof of Lemma \ref{identicaly0}\eqref{iii} readily gives:


\begin{prop}\label{prescribir} For every Fuchsian $\prin$ and $v\in\tangcero{\prin}{e}$ one has ${\displaystyle \PP_\prin^{\length}=\Big(\frac{c_e}{c_1}\Big)^2\PP^{\peso_1}_\prin.}$
\end{prop}



\subsection{A pressure form compatible at the Fuchsian locus}\label{comS}

Consider a Fuchsian $\prin:\pi_1S\to\PSL_d(\R),$ $e\in\lb1,d-1\rb$ and $q\in H^0(K^{e+1})$ a holomorphic differential of degree $e+1$ on $S_\prin.$ Then Hitchin's parametrization provides a \emph{normalized deformation} $\nrm(q)\in\sf T_\prin\hitchin_{\A_{d-1}}(S)$ and one has the following.

\begin{thm}[{Labourie-Wentworth \cite[Cor. 3.5.2 and Cor. 6.1.2]{Labourie-Wentworth}}]\label{LW}\label{standard-cociclo}Let $\prin$ be Fuchsian, $q$ a holomorphic differential on $S_\prin$ of degree $e+1$ and $\coclase_{\nrm(q)}\in H^1_{\Ad\prin}\big(\pi_S,\sl_d(\R)\big)$ the cocycle associated to $\nrm(q).$ Then $\coclase_{\nrm(q)}\in H^1_{\Ad\prin}\big(\pi_S,V_e)$ and \begin{alignat*}{2}\PP^{\peso_1}_\prin\big(\nrm(q)\big)&= \frac{(d-1)!^2}{2^e}\frac{(d+1)d(d-1)}{3\cdot2}\frac{(2e+1)!}{(d+e)!(d-e-1)!}{\displaystyle \frac{\int_S\|q\|^2d\area_\prin}{\pi|\chi(S)|}} \\&= \frac{(d-1)^{\lf{e}} }{(d+e)^{\lf{e-1}} } \frac{(d-1)}{3}\frac{(2e+1)!}{2^e}{\displaystyle \frac{\int_S\|q\|^2d\area_\prin}{2\pi|\chi(S)|}}.\end{alignat*}\end{thm}

We consider then the following functional.

\begin{defi}\label{defKahl}We let $\kahl\in\a^*$ be defined by, for all $e\in\lb1,d-1\rb,$  $${\displaystyle\kahl|\kt^e=\sqrt{\frac{(d+e)^{\lf{e-1}}}{(d-1)^{\lf e}}\frac{3\cdot 2^e}{(\dim V_e)!(d-1)}}\cdot\peso_1|\kt^e.}$$\end{defi}

By Remark \ref{peso1} $\peso_1(\ktv^e)\neq0$, whence by definition $\kahl(\ktv^e)\neq0$ and thus Theorem \ref{degLietheoretic} entails that $\PP^\kahl$ is Riemannian on $\EuScript U_\kahl$. We have:

\begin{cor}\label{compa1}Let $\prin$ be a Fuchsian representation and $q\in \bigoplus_{e=1}^{d-1}H^0(K^e)$ then $$\PP^\kahl_\prin\big(\nrm(q)\big)=\frac{(d-1)^2}{2\pi|\chi(S)|}\int_{S_\prin}\|q\|^2d\area_\prin$$ and $\kahl$ is the only linear form so that this equation holds at the Fuchsian points. Thus, there exists $\lambda>0$ so that the operator $\mathrm{j}$ defined by $\PP^\kahl(u,v)=\symp(\mathrm{j} u,v)$ squares $-\lambda$ on the tangent space to $\hitchin_{\A_{d-1}}(S)$ at the Fuchsian points $\cal T(S).$ \end{cor}

\begin{proof}
By Theorem \ref{standard-cociclo} $\coclase_{\nrm(q)}\in H^1_{\Ad\rho}(\pi_1S,V_e)$ so the result follows from Theorem \ref{LW} and Proposition \ref{prescribir}. The last assertion follows from \cite[Lemma 5.1.1]{Labourie-Wentworth}.
\end{proof}

If we want to find the form $\kahl_\ge$ restricted to the Hitchin components of type $\B,\,\Ce$ or $\Ge$, then we keep the coefficients of Definition \ref{defKahl} for odd exponents and impose $\kahl_\ge(\kt^e)=0$ for even exponents, for type $\Ge$ we further impose $\kahl_{\Gedosl}(\kt^3)=0.$

\begin{obs}\label{calculokahl}
Eq. \eqref{kahlrg2} contains the computation of $\kahl$ for the rank $2$ simple split algebras, we compute here $\kahl$ for rank $3$. These computations are straightforward consequence of the definition of $\kahl$ and the formul\ae\  for $\ktv^e$ from Lemma \ref{calculo}:

\begin{alignat*}{2}
20\kahl_{\sl(4,\R)}(a)&=\Big(6+\frac{\sqrt{10}}{15}\Big)a_1+\Big(4-\frac{\sqrt{10}}{15}-\frac{\sqrt{10}\sqrt3}3\Big)a_2+\Big(2+\frac{2\sqrt{10}}{15}-\frac{\sqrt{10}\sqrt{3}}3\Big)a_3;\\
35\kahl_{\sp(6,\R)}(a)&=\Big(5 + \frac{211\sqrt{35}}{3780}\Big)a_1+ \Big(3 - \frac{299\sqrt{35}}{3780}\Big)a_2 + \Big(1 - \frac{79\sqrt{35}}{1890}\Big)a_3;\\
28\kahl_{\so(3,4)}(a)&=\Big(3+\frac{\sqrt{42}\sqrt{10}}{90} +\frac{ \sqrt{42}}{3780}\Big)a_1
+\Big(2 - \frac{\sqrt{42}\sqrt{10)}}{90} - \frac{\sqrt{42}}{945}\Big)a_2\\ & \ \ \ +\Big(1 - \frac{\sqrt{42}\sqrt{10}}{90} + \frac{\sqrt{42}}{756}\Big)a_3.
\end{alignat*}
\end{obs}

\section{Hausdorff dimension degenerations}\label{Hffdeg}
%
%
%
%
%
%
%

Since Hitchin representations are $\simple$-positive (Fock-Goncharov \cite{FG}), Corollary \ref{HffTpos} deals with the Hessian of Hausdorff dimension at Zariski-dense points. We now apply Theorem \ref{thmPrinc} to understand degenerations for the lower strata, which reduces the question to the Fuchsian locus and to two exceptional situations. We focus on the former. Combining with Theorem \ref{BPSWHyper} we get:


\begin{cor}\label{corcho}Let $v\in\sf T_\prin\hitchin_\ge(S)$ with $\prin$ Fuchsian and $\sroot\in\simple$. If $v\in\tangcero\prin e$ and $\kt^e\subset\ker\sroot$, then $\hess_\prin\Hff_\sroot(\sf Jv)=0.$ If $\ge$ has classical type then the converse is also true: if $\hess_\prin\Hff_\sroot(\sf Jv)=0$ then $v\in\bigoplus_{e:\kt^e\subset\ker\sroot}\tangcero\prin e$.
\end{cor}

The question is thus reduced to understanding the triplets $(d,e,j)$ such that \begin{equation}\label{defSingular}\slroot_j(\ktv^e)=0.\end{equation} Such line will be called a \emph{simple-singular Kostant line}. In the following we study this equation in some situations, however the general case is not understood.

\subsection{Elementary Families}\label{elemfam}

\begin{prop}The following are simple-singular Kostant lines for $\A_{d-1}$:
\begin{enumerate}\item $d=2n,$ even exponent and the middle simple root $\slroot_n,$\item the second root $\slroot_2$ and, for every exponent $e,$ $d=1+e(e+1)/2,$
\item the triplets ($d$, exponent, root) defined as, for every $m\in\N$ $$(4m+3,2m+1,2m),$$
\item $e=3$ and the pairs $(d,j):=\Big(\begin{smallmatrix}4 & -5\\ 1 & -1\end{smallmatrix}\Big)^k \Big(\begin{smallmatrix}7\\ 2\end{smallmatrix}\Big),$ for any integer $k\geq0$,\item the 4th exponent $e=4$ and the pairs $(d,j)$ of the form $(d,j)=\Big(\begin{smallmatrix}6 & -7\\ 1 & -1\end{smallmatrix}\Big)^k \Big(\begin{smallmatrix}11\\ 2\end{smallmatrix}\Big),$ or $(d,j)=\Big(\begin{smallmatrix}6 & -7\\ 1 & -1\end{smallmatrix}\Big)^k \Big(\begin{smallmatrix}17\\ 3\end{smallmatrix}\Big),$ for any integer $k\geq0$.
\end{enumerate}
\end{prop}

\begin{proof} Item (i) follows from $\ii$-invariance of $\kt^e$ for even $e$ (Remark \ref{evenw0}). The next two items follow by direct computation, let us do (iii). Indeed, using Lemma \ref{calculo} and replacing $j$ by $2m,$ $e$ by $2m+1$ and $d$ by $4m+3$ one has
\begin{alignat*}{2}\frac{\slroot_{2m}(\ktv^{2m+1})}{(-1)^e(e+1)!} & =\sum_{t=1}^{2m+1}(-1)^{t}\binom{2m+1}{t}\binom{2m+1}{t-1} (2m-1)^{\lf {2m+1-t}}\big(2m+2\big)^{\lf {t-1}}\\
 & = \sum_{t=2}^{2m+1}(-1)^{t}\tbinom{2m+1}{t}\tbinom{2m+1}{t-1} \frac{(2m-1)!}{(t-2)!}\frac{(2m+2)!}{(2m+3-t)!}\\ & = (2m-1)!(2m+2)\sum_{t=2}^{2m+1}(-1)^{t}\tbinom{2m+1}{t}\tbinom{2m+1}{t-1} \tbinom{2m+1}{t-2} =0,
\end{alignat*}
since by considering $k=2m+1-t$ in the above, the sum equals its negative\footnote{We thank Germain Poullot for the above argument.}

The next two also follow from Lemma \ref{calculo} however in these cases one has to find the integer solutions of an integral equation $q=c$. For the third exponent one has $q_3(d,j)=-d^2+5dj-5j^2=1$ which is solved by considering the cyclic group $\SO(q_3,\Z)=\<\Big(\begin{smallmatrix}4 & -5\\ 1 & -1\end{smallmatrix}\Big)\>.$ The last case is analogous. \end{proof}

%
%
%
%
%

\subsection{Degenerations for the 3rd root}\label{3raraiz} We complete the proof of Corollary \ref{Hff3}.

\begin{cor} An element $\tangente\in\tangcero\prin e$ is such that $\hess_\prin\Hff_{\slroot_3}(\sf J\tangente)=0$ if and only if the pair $(d,e)$ verifies $1<e<d$ and satisfies the Diophantine equation \begin{equation}\label{s3}e^4-6de^2+2e^3+6d^2-6de+11e^2-18d+10e+12=0.\end{equation}\end{cor}

\begin{proof} By Corollary \ref{corcho} we have to solve $\slroot_3(\ktv^e)=0$ for which Lemma \ref{calculo}  gives \begin{equation}\label{tgen}\sum_{t=1}^e(-1)^{t}\binom{e}{t}\binom{e}{t-1} 2^{\lf {e-t}}\big(d-4\big)^{\lf {t-1}}=0.\end{equation} We begin by observing that $2^{\lf{e-t}}\neq0$ if and only if $e-t\in\{0,1,2\},$ and that $(d-4)^{\lf {t-1}}\neq0$ iff $t<d-2$. Moreover, since $t\leq e$ and $\ktv^{d-1}$ is never singular by Eq. \eqref{ejsktv}, we can restrict to $t\leq e \leq d-2$. If we let $p(e,t)$ denote the general term in the sum \eqref{tgen}, then we want to compute the alternated sum $$p(e,e)-p(e,e-1)+p(e,e-2)=0,$$ together with the constrain $e\leq d-2$. Explicit computation gives
\begin{alignat*}{2} p(e,e) &  =\Big(e (d-4)^{\lf {e-3}}\Big)(d-e-1)(d-e-2);\\
p(e,e-1) &= \Big(e(d-4)^{\lf{e-3}}\Big)e(e-1)(d-e-1); \\p(e,e-2)&=\Big(e(d-4)^{\lf{e-3}}\Big)\frac{e(e-1)^2(e-2)}6.
\end{alignat*}

\noindent
So alternating the sum and removing the common factor gives Equation \eqref{s3}.\end{proof}

We complete the proof of Corollary \ref{Hff3} by solving \eqref{s3} over $\Z$. A first remark is that it is preserved by the involution $(d,e)\mapsto(d,-e-1)$, so we only need to find (and care about) its solutions for $e\geq0$.

\begin{prop}\label{eqsols} The integer solutions of  \eqref{s3} with $e\geq0$ are 

\begin{table}[h!]
\begin{center}
\begin{tabular}{c|c|c|c|c|c}
$e$ & $0$ & $1$ & $2$ & $4$ & $8$ \\ \hline\hline
$d$ & $2$ & $3$ & $6$ & $17$ & $58$ \\ \hline
$d$ & $1$ & $2$ & $3$ & $6$ & $17$
\end{tabular}
\end{center}
\end{table}
\end{prop}


We prove now Proposition \ref{eqsols}. Clearing the variable $d$ gives $d=\frac{e^2+e+3}2\pm\frac16\sqrt{3(e^4+2e^3-e^2-2e+3)}$ so we now focus on the Diophantine equation \begin{equation}\label{despejar}f(x):= 3(x^4+2x^3-x^2-2x+3)=y^2,\end{equation} which, by means of the rational solution $(-1, 3)$, can be transformed by a $\Q$-birational map to an elliptic equation:

\begin{lemma} Consider the elliptic curve over $\Q$ defined by \begin{equation}\label{eliptica}E:\ y^2=x^3-147x+610.\end{equation} Then rational solutions of Equation \eqref{despejar} are parametrized by $E(\Q)$ via the maps

\begin{alignat*}{2} X_1(x,y) & =\left(-\frac{-7x + 35 + y}{17 + y - x}, 3\frac{-2x^3 + y^2 + 9x^2 + 28y + 169}{(17 + y - x)^2}\right);\\
X_2(x,y) & =\left(-\frac{7x - 35 + y}{-17  + x+y}, 3\frac{-2x^3 + y^2 + 9x^2 - 28y + 169}{(-17 + x + y)^2}\right).\end{alignat*}
\end{lemma}

\begin{proof} This is standard given there exists a rational solution of \eqref{despejar}, in this case $x=-1$, $y=3$. We begin by replacing $x$ by $x-1$ which gives $$ 3x^4-6x^3-3x^2+6x+9-y^2=0,$$ followed by replacing $x$ by $1/x$ and $y$ by $3y/x^2$, which gives, by considering the numerator  $$9 x^4+6x^3-3x^2-6x+3-9y^2=0=x^4+\frac{2}{3}x^3-\frac{1}{3}x^2-\frac{2}{3}x+\frac{1}{3}x-y^2.$$ 

We now replace $x$ by $x-2/(3\cdot4)$ and $y$ by $x^2+y-(1/6\cdot2)$ to obtain a quadratic polynomial on $x$, indeed the additive terms are chosen to make disappear the higher-degree terms on $x$. The solutions on $x$ of the obtained quadratic polynomial are $$x=\frac{-14 \pm \sqrt{-5832y^3 + 2646y + 610}}{18(6y + 1)},$$
\noindent
which are obtained by describing the rational solutions of the equation $\Delta=v^2$. This latter equation is $$v^2=-5832u^3 + 2646u + 610=\big(-3^22u\big)^3-147\big(-3^22u\big)+610,$$ so replacing $-3^22u$ by $u$ we obtain the desired elliptic Equation \eqref{eliptica}, and the composition of the above local change of variables give the stated rational map.\end{proof}

We now proceed via the $\mathfrak{Ellog}$ method and, more precisely, use Tzanakis \cite{cuartica}. The Mordel-Weil group of $E(\Q)$ consists on one torsion point $(5,0)$  and the points $$R_1=(9,4)\ \textrm{ and }\ R_2=(11,18)$$ form a basis of the free part. If $(x,y)$ is a rational solution of Equation \eqref{despejar} then via Lemma \ref{eliptica} $X_i(x,y)\in E(\Q)$ and thus can be written as $$X_i(x,y)=m_0(5,0)+m_1R_1+m_2R_2$$ for some integers $m_i$, (with $m_0\in\{0,1\}$), additivity denotes the group law of $E(\Q)$.

The method consists on providing an upper bound for $M=\max\{|m_1|,|m_2|\}$ under the assumption that $(x,y)$ is a pair of integers, which reduces the problem to an explicit computation that can be carried out by a computer. To find this upper bound we nedd some data about the curve $E$, most of the following computations are computer-assisted and required to work on Maple with 13 decimal digits.

Consider the polynomial $q(u)=u^3-147u+610$, then

\begin{itemize} \item the solutions to $q=0$ are ${\displaystyle e_3=\frac{-5-3\sqrt{57}}2,\ e_2=5,\ e_1=\frac{-5+3\sqrt{57}}2,}$ \item The discriminant of $q$ is $2659392$ and $\Delta=2^4*2659392=42550272$,
\item the minimal real period of $E$ is ${\displaystyle \omega=2\int_{e_1}^\infty\frac{dt}{\sqrt{q(t)}}\asymp 0.9810124566....}$
\item The fundamental periods are, if $\mathrm M$ denotes the arithmetic-geometric mean, \begin{alignat*}{2}\omega_1 & =\frac{2\pi}{\mathrm M(\sqrt{e_1-e_3},\sqrt{e_1-e_2})}=2\omega\asymp1.962095763...,\\ \omega_2 & =\frac{2\pi}{\mathrm M(\sqrt{e_1-e_3},\sqrt{-1}\sqrt{e_2-e_3)}}\asymp 1.177161295... - 1.128478211...\sqrt{-1}.\end{alignat*} Since $\omega_2/\omega_1$ does not belong to the Gauss fundamental domain of the modular surface, we consider $$\tau=\big(\begin{smallmatrix} 1 & 0 \\ 1 & 1\end{smallmatrix}\big)(-\omega_2/\omega_1)\asymp 0.1849446113... + 1.171782212...\sqrt{-1}$$ with modulus $|\tau|\asymp1.186287512...$

\item The $j$-invariant is $j_E=470596/57$ and so the Archimedean contribution to its height is $h_\infty(j)=\log|470596/57|\asymp9.018703988....$
\end{itemize}

By means of \cite[Theorem 1.1]{SilvermanPaper} we have, for every $p\in E(\Q)$, that \begin{equation}\label{weilcan}\hat h(p)-\frac12 h(x(p))\leq\frac{h(\Delta)+h_\infty(j)}{12}+1.07=c_{11}\asymp 3.285408400...\end{equation}
Let $h$ be the logarithmic height, defined for a rational $p/q$ in lowest terms, defined by $h(p/q)=\log\max\{|p|,|q|\}$, and  if $(p_i/q_i)\in\Q^n$ we let  $$h(p_1/q_1,\ldots,p_n/q_n)=\log\max\{q,q|p_i|/q_i:i\in\lb1,n\rb\},$$  where $q=\mathrm{lcm}\{q_i:i\in\lb1,n\rb\}$. We denote by \begin{alignat*}{2}h_E & :=\max\{1,h(-147/4,610/16),h(j_E)\}\\&=h(j_E)=\log470596\asymp13.06175526....\end{alignat*}

We will also need the point $x_0=6\sqrt3-1>e_1$, we let $\sigma=1$ and we consider $$R_0=\big(x_0,6(3-\sqrt 3)\big)\in E(\Q(\sqrt3)).$$

If we let $E_0(\R)$ denote the unbounded component of $E(\R)$, then for $p\in E_0(\R)$ with coordinates $(u(p),v(p))$, the map $\phi:E_0(\R)\to\R/\Z$ given by $$\phi(p):=\left\{\begin{array}{lll} 0 \mod 1 & \textrm{if $p=O$},\\{\displaystyle \frac1\omega \int_{u(p)}^\infty\frac{du}{\sqrt{q(u)}} \mod 1} & \textrm{if }v(p)\geq0,\\ -\phi(-P) \mod 1& \textrm{if }v(p)\leq0,\end{array}\right.$$ is a group homomorphism. Observe that $\{R_0,R_1,R_2\}\subset E_0(\R)$. One has \begin{alignat*}{2} \omega\phi(R_1) & \asymp 0.8918445254...\\ \omega\phi(R_2) & \asymp 0.6925571056...\\  \omega\phi(R_0) & \asymp 0.8235278325... \end{alignat*}

\noindent Tzanakis \cite{cuartica} requires us to choose numbers $A_0,\ldots,A_3,\cal E$ such that \begin{alignat*}{2}A_0& \geq\max\Big\{ h_E,\frac{3\pi\omega^2}{|\omega_1|^2\frak I(\tau)}\Big\}=h_E\asymp 13.06175526....;\\
A_{i+1}&\geq \max\Big\{ h_E,\frac{3\pi\omega^2\phi(R_i)^2}{|\omega_1|^2\frak I(\tau)},\hat h(R_i):\Big\}\\ &=\max\big\{ h_E,\hat h(R_i)\big\}=h_E;\\ {\sf e} & \leq \cal E\leq \sf e\min\left\{\frac{|\omega_1|}{\omega}\cdot\sqrt{\frac{2A_0\frak I\tau}{3\pi}},\frac{|\omega_1|}{\omega\phi(R_i)}\cdot\sqrt{\frac{2A_{i+1}\frak I\tau}{3\pi}},i=0,1,2\right\}; \end{alignat*} where we have used Equation \eqref{weilcan} to find an upper bound of $\hat h (R_i)$,  and $\sf e$ is the Euler number. We can choose then $A_i=13.5$ for $i=0,1,2,3$, and $\cal E=9$. Then we compute $c_4,c_5$ and $c_6$ from \cite[\S\,7]{cuartica} given by David \cite[Th\'eor\`eme 2.1]{sinnou}:
\begin{alignat*}{2}c_4 &=2.9\cdot10^{30}\cdot 2^{10}\cdot4^{32}\cdot5^{80.3}(\log\cal E)^{-9}(13.5)^4\\&\asymp 2.043497279...\cdot10^{110}\\ c_5 & =\log(2\cal E)=\log (18)\\ c_6&=\log(18)+h_E\asymp 15.95212702....\end{alignat*}

We also consider the regulator matrix associated to the basis $\{R_1,R_2\}$, it is the matrix associated to the quadratic form on $E(\Q)$ defined by $\<P,Q\>=\hat h(P+Q)-\hat h(P)-\hat h(Q).$ We let $c_1$ be its smallest eigenvalue, one has $c_1\leq 0.303868...$

We now have to find a constant $c_9$ such that $$\frac1\omega\int_U^\infty\frac{dx}{\sqrt{f(u)}}\leq \frac{c_9}{\omega}\frac1U$$ which can be easily shown to be $c_9=1/\sqrt3.$ We now compute a constant $c_{10}$ so that for every $u\geq1$ $$h\Big(\frac{6\sqrt{Q(u)}-6u+18}{u^2}\Big)\leq c_{10}+2\log u,$$ and we get $c_{10}=3$. We finally have that the constants $c_{12}=1$ and $c_{13}=0$ and we obtain the first upper bound on $M\geq16$:

$$c_1M^2\leq \log c_9+\frac12c_{10}+c_{11}+c_4(\log M+c_5)(\log\log M+c_6)^5,$$ which gives $M\leq 6.123\cdot10^{59}.$

We now proceed with the reduction of the upper bound for $M$ applying \cite[\S\,5]{cuartica}. Since $R_0$ does not belong to $E(\Q)$, the $\R/\Z$ elements $$\{\phi(R_0),\phi(R_1),\phi(R_2)\}$$ are linearly independent over $\Q$, and thus our situation is Case 2 in that section, we are hence bound to use \cite[Proposition 4]{cuartica}, which gives the reduction $M\leq30$. At this point we proceed with a case by case computation using, for example, Maple.

\bibliography{ktlinesArx}

\begin{thebibliography}{10}

\bibitem{babled}
M.~Babillot and F.~Ledrappier.
\newblock {Lalley's theorem on periodic orbits of hyperbolic flows}.
\newblock {\em Ergod. Th. \& Dynam. Sys.}, 18, 1998.

\bibitem{Benoist-HomRed}
Y.~Benoist.
\newblock {Actions propres sur les espaces homog{\`e}nes r{\'e}ductifs}.
\newblock {\em Ann. of Math.}, 144:315--347, 1996.

\bibitem{limite}
Y.~Benoist.
\newblock {Propri{\'e}t{\'e}s asymptotiques des groupes lin{\'e}aires}.
\newblock {\em Geom. Funct. Anal.}, 7(1):1--47, 1997.

\bibitem{benoist2}
Y.~Benoist.
\newblock {Propri{\'e}t{\'e}s asymptotiques des groupes lin{\'e}aires II}.
\newblock {\em Adv. Stud. Pure Math.}, 26:33--48, 2000.

\bibitem{Benoist-QuintLibro}
Y.~Benoist and J.-F. Quint.
\newblock {\em Random walks on reductive groups}, volume~62 of {\em A Series of
  Modern Surveys in Mathematics}.
\newblock Springer, 2016.

\bibitem{BGLPW}
J.~Beyrer, O.~Guichard, F.~Labourie, B.~Pozzetti, and A.~Wienhard.
\newblock Positivity, cross-ratios and the collar lemma.
\newblock \url{https://arxiv.org/abs/2409.06294}.

\bibitem{BP}
J.~Beyrer and B.~Pozzetti.
\newblock Positive surface group representations in ${\mathsf{po}(p,q)}$.
\newblock To appear in \emph{JEMS}, {\footnotesize
  \url{https://arxiv.org/abs/2106.14725v1}}.

\bibitem{BPS}
J.~Bochi, R.~Potrie, and A.~Sambarino.
\newblock {Anosov Representations and dominated splittings}.
\newblock {\em J. Eur. Math. Soc.}, 21(11):3343--3414, 2019.

\bibitem{B-intersection}
F.~Bonahon.
\newblock {The geometry of Teichm{\"u}ller space via geodesic currents}.
\newblock {\em Invent. Math.}, 92:139--162, 1988.

\bibitem{bowen-quasicircles}
R.~Bowen.
\newblock {Hausdorff dimension of quasi-circles}.
\newblock {\em Publ. Math. I.H.E.S.}, 50(1):11--25, 1979.

\bibitem{bowenruelle}
R.~Bowen and D.~Ruelle.
\newblock {The ergodic theory of $\textrm{A}$xiom $\textrm{A}$ flows}.
\newblock {\em Invent. Math.}, 29:181--202, 1975.

\bibitem{pressurecusped}
H.~Bray, R.~Canary, L-Y. Kao, and G.~Martone.
\newblock Pressure metric for cusped $\mathrm{H}$itchin components.
\newblock {\em Adv. in Math.}, 435:24, 2023.

\bibitem{martincriticos}
M.~Bridgeman.
\newblock {Hausdorff dimension and the $\textrm{W}$eil-$\textrm{P}$etersson
  extension to quasifuchsian space}.
\newblock {\em Geom. \& Top.}, 14(2):799--831, 2010.

\bibitem{pressure}
M.~Bridgeman, R.~Canary, F.~Labourie, and A.~Sambarino.
\newblock {The pressure metric for Anosov representations}.
\newblock {\em Geom. Funct. Anal.}, 25(4):1089--1179, 2015.

\bibitem{Liouvillepressure}
M.~Bridgeman, R.~Canary, F.~Labourie, and A.~Sambarino.
\newblock {Simple roots flows for Hitchin representations}.
\newblock {\em Geom. Dedic. William Goldman's 60th birthday special edition},
  192:57--86, 2018.

\bibitem{HessianHff}
M.~Bridgeman, B.~Pozzetti, A.~Sambarino, and A.~Wienhard.
\newblock {Hessian of Hausdorff dimension on purely imaginary directions}.
\newblock {\em Bull. London Math. Soc.}, 54(3):1027--1050, 2022.

\bibitem{WP-QF}
M.~Bridgeman and E.~Taylor.
\newblock {An extension of the Weil-Petersson metric to quasi-Fuchsian space}.
\newblock {\em Math. Annalen}, 341:927--943, 2008.

\bibitem{burger}
M.~Burger.
\newblock {Intersection, the $\textrm{M}$anhattan curve, and
  $\textrm{P}$atterson-$\textrm{S}$ullivan $\textrm{T}$heory in rank 2}.
\newblock {\em Inter. Math. Research. Not.}, 7, 1993.

\bibitem{CDPW}
L.~Carvajales, X.~Dai, B.~Pozzetti, and A.~Wienhard.
\newblock Thurston's asymmetric metric for anosov representations.
\newblock To appear in \emph{Groups, geometry and dynamics},
  {\footnotesize\url{https://arxiv.org/pdf/2210.05292.pdf}}, 2022.

\bibitem{CrokeFathi}
C.~Croke and A.~Fathi.
\newblock {An inequality between energy and intersection}.
\newblock {\em Bull. London Math. Soc.}, 22:489--494, 1990.

\bibitem{DGKLorentz}
J.~Danciger, F.~Gu{\'e}ritaud, and F.~Kassel.
\newblock Geometry and topology of complete lorentz spacetimes of constant
  curvature.
\newblock {\em Ann. Sci. {\'E}cole Norm. Sup.}, 49:1--56, 2016.

\bibitem{sinnou}
S.~David.
\newblock {\em Minorations de formes lin{\'e}aires de logarithmes elliptiques}.
\newblock Number~62 in M{\'e}moires de la Soci{\'e}t{\'e} Math{\'e}matique de
  France, S\'erie 2. Soci{\'e}t{\'e} math{\'e}matique de France, 1995.

\bibitem{FPVBowen}
J.~Farre, B.~Pozzetti, and G.~Viaggi.
\newblock Geometry of hyperconvex representations of surface groups.
\newblock \footnotesize{\url{https://arxiv.org/pdf/2407.20071}}.

\bibitem{fathi-flaminio}
A.~Fathi and L.~Flaminio.
\newblock Infinitesimal conjugacies and weil-petersson metric.
\newblock {\em Ann. Inst. Fourier}, 43(1):279--299, 1993.

\bibitem{FG}
V.~Fock and A.~Goncharov.
\newblock {Moduli spaces of local systems and higher Teichm{\"u}ller theory}.
\newblock {\em Publ. Math. de l'I.H.E.S.}, 103:1--211, 2006.

\bibitem{FultonHarris}
W.~Fulton and J.~Harris.
\newblock {\em Representation theory, a first course}.
\newblock Number 129 in {Graduate Texts in Mathematics}. Springer Verlag, New
  York, 1991.

\bibitem{ghoshJordan}
S.~Ghosh.
\newblock Deformation of fuchsian representations and proper affine actions.
\newblock \footnotesize{\url{https://arxiv.org/abs/2312.16655}}.

\bibitem{SouravIso}
S.~Ghosh.
\newblock Isospectrality of $\textrm{M}$argulis-$\textrm{S}$milga spacetimes
  for irreducible representations of split semi-simple $\mathrm{L}$ie groups.
\newblock arXiv:2009.12746v1.

\bibitem{Souravpersonalcomm}
S.~Ghosh.
\newblock Personal communcation.

\bibitem{GoldmanSymplectic}
W.~Goldman.
\newblock The symplectic nature of fundamental groups of surfaces.
\newblock {\em Adv. in Math.}, 54:200--225, 1984.

\bibitem{GLM}
W.~Goldman, F.~Labourie, and G.~Margulis.
\newblock Proper affine actions and geodesic flows of hyperbolic surfaces.
\newblock {\em Ann. of Math. (2)}, 170(3):1051--1083, 2009.

\bibitem{GGKW}
F.~Gu{\'e}ritaud, O.~Guichard, F.~Kassel, and A.~Wienhard.
\newblock {Anosov representations and proper actions}.
\newblock {\em Geom. \& Top.}, 21:485--584, 2017.

\bibitem{clausura}
O.~Guichard.
\newblock Personal communication, 2012.

\bibitem{GLW}
O.~Guichard, F.~Labourie, and A.~Wienhard.
\newblock {Positivity and representations of surface groups}.
\newblock To appear in \emph{Forum Pi}, {\footnotesize
  \url{https://arxiv.org/abs/2106.14584}}, 2021.

\bibitem{olivieranna}
O.~Guichard and A.~Wienhard.
\newblock {Anosov representations: domains of discontinuity and applications}.
\newblock {\em Invent. Math.}, 190:357--438, 2012.

\bibitem{GW-positive}
O.~Guichard and A.~Wienhard.
\newblock {Positivity and higher $\textrm{T}$eichm{\"u}ller theory}.
\newblock In {\em {European Congress of Mathematics}}, pages 289--310. Eur.
  Math. Soc., 2018.

\bibitem{Hitchin}
N.~J. Hitchin.
\newblock {Lie groups and $\textrm{T}$eichm{\"u}ller space}.
\newblock {\em Topology}, 31(3):449--473, 1992.

\bibitem{james}
J.~Humphreys.
\newblock {\em {Introduction to Lie algebras and representation theory}}.
\newblock {Graduate Texts in Mathematics}. Springer-Verlag, 1972.

\bibitem{LAG}
J.~Humphreys.
\newblock {\em {Linear algebraic groups}}.
\newblock Springer, 1998.

\bibitem{KLP-Morse}
M.~Kapovich, B.~Leeb, and J.~Porti.
\newblock {A Morse Lemma for quasigeodesics in symmetric spaces and euclidean
  buildings}.
\newblock {\em Geom. \& Top.}, 22:3827--3923, 2018.

\bibitem{KS}
F.~Kassel and I.~Smilga.
\newblock Affine properness criterion for $\mathrm{A}$nosov representations and
  generalizations.
\newblock preprint, 2023.

\bibitem{Jcritico}
A.~Katok, G.~Knieper, and H.~Weiss.
\newblock {Formulas for the derivative and critical points of topological
  entropy for Anosov and Geodesic flows}.
\newblock {\em Comm. Math. Phys.}, 138:19--31, 1991.

\bibitem{Kim-Zhang}
I.~Kim and G.~Zhang.
\newblock K\"ahler metrtic on the space of convex real projective structures on
  surface.
\newblock {\em J. Differential Geometry}, 106:127--137, 2017.

\bibitem{knapp}
A.~Knapp.
\newblock {\em {Lie groups beyond an introduction}}.
\newblock Birkh{\"a}user, 2002.

\bibitem{kni95}
G.~Knieper.
\newblock {Volume gowth, entropy and the geodesic stretch}.
\newblock {\em Math. Research Letters}, 2:39--58, 1995.

\bibitem{kostant}
B.~Kostant.
\newblock {The principal three-dimensional subgroup and the {B}etti numbers of
  a complex simple {L}ie group}.
\newblock {\em Amer. J. Math.}, 81:973--1032, 1959.

\bibitem{Labourie-FuchsAffine}
F.~Labourie.
\newblock {Fuchsian affine actions of surface groups}.
\newblock {\em J. Differential Geom.}, 59:15--31, 2001.

\bibitem{labourie}
F.~Labourie.
\newblock {Anosov Flows, Surface Groups and Curves in Projective Space}.
\newblock {\em Invent. Math.}, 165:51--114, 2006.

\bibitem{LabourieLectures}
F.~Labourie.
\newblock {\em Lectures on representations of surface groups}.
\newblock Zurich Lectures in Advances Mathematics. Eur. Math. Soc., Z\"urich,
  2013.

\bibitem{labourie-cyclicsurfaces}
F.~Labourie.
\newblock {Cyclic surfaces and $\textrm{H}$itchin components in rank 2}.
\newblock {\em Ann. of Math. (2)}, 185:1--58, 2017.

\bibitem{LabourieAffine}
F.~Labourie.
\newblock Entropy and affine actions for surface groups.
\newblock {\em J. Topol.}, 15(3):1017--1033, 2022.

\bibitem{Labourie-Wentworth}
F~Labourie and R.~Wentwoth.
\newblock {Variations along the Fuchsian locus}.
\newblock {\em Ann. Sci. {\'E}cole Norm. Sup.}, 51(2):487--547, 2018.

\bibitem{ledrappier}
F.~Ledrappier.
\newblock {Structure au bord des vari{\'e}t{\'e}s {\`a} courbure n{\'e}gative}.
\newblock {\em S{\'e}minaire de th{\'e}orie spectrale et g{\'e}om{\'e}trie de
  Grenoble}, 71:97--122, 1994-1995.

\bibitem{MargulisAffineRuso}
G.~Margulis.
\newblock Free completely discontinuous groups of affine transformations.
\newblock {\em Dokl. Akad. Nauk SSSR}, 272(4):785--788, 1983.

\bibitem{MargulisAffine}
G.~Margulis.
\newblock Complete affine locally flat manifolds with a free fundamental group.
\newblock {\em Journal of Soviet Mathematics}, 36(1):129--139, 1987.

\bibitem{McMWP}
C.~McMullen.
\newblock {Thermodynamics, dimension and the Weil-Petersson metric}.
\newblock {\em Invent. Math.}, 173:365--425, 2008.

\bibitem{Mess}
G.~Mess.
\newblock Lorentz spacetimes of constant curvature (1990).
\newblock {\em Geom. Dedic.}, 126:3--45, 2007.

\bibitem{parrypollicott}
W.~Parry and M.~Pollicott.
\newblock {\em {Zeta Functions and the periodic orbit structure of hyperbolic
  dynamics}}, volume 187-188.
\newblock Ast{\'e}risque, 1990.

\bibitem{smaleflows}
M.~Pollicott.
\newblock {Symbolic dynamics for $\textrm{S}$male flows}.
\newblock {\em Amer. Journ. of Math.}, 109(1):183--200, 1987.

\bibitem{exponentecritico}
R.~Potrie and A.~Sambarino.
\newblock {Eigenvalues and entropy of a $\textrm{Hitchin}$ representation}.
\newblock {\em Invent. Math.}, 209:885--925, 2017.

\bibitem{PS}
B.~Pozzetti and A.~Sambarino.
\newblock Metric properties of boundary maps: $\textrm{H}$ilbert entropy and
  non-differentiability.
\newblock To appear in \emph{Annales Henri Lebesgue},
  {\footnotesize{\url{https://arxiv.org/abs/2310.07373}}}, 2024.

\bibitem{PSW1}
B.~Pozzetti, A.~Sambarino, and A.~Wienhard.
\newblock {Conformality for a robust class of non conformal attractors}.
\newblock {\em J. Reine Angew. Math.}, 2021(774):1--51, 2021.

\bibitem{PSW2}
B.~Pozzetti, A.~Sambarino, and A.~Wienhard.
\newblock {Anosov representations with Lipschitz limit set}.
\newblock {\em Geom. \& Top.}, 27(8):3303--3360, 2023.

\bibitem{RaghunathanBook}
M.~S. Raghunathan.
\newblock {\em Discrete subgroups of Lie groups}, volume~68 of {\em Ergebnisse
  der Mathematik und ihrer Grenzgebiete. 2. Folge}.
\newblock Springer Berlin, Heidelberg, 1972.

\bibitem{quantitative}
A.~Sambarino.
\newblock {Quantitative properties of convex representations}.
\newblock {\em Comment. Math. Helv.}, 89(2):443--488, 2014.

\bibitem{orbitalcounting}
A.~Sambarino.
\newblock {The orbital counting problem for hyperconvex representations}.
\newblock {\em Ann. Inst. Fourier}, 65(4):1755--1797, 2015.

\bibitem{dichotomy}
A.~Sambarino.
\newblock {A report on an ergodic dichotomy}.
\newblock {\em Ergod. Th. \& Dynam. Sys.}, 44(1):236 -- 289, 2024.

\bibitem{clausurasPos}
A.~Sambarino.
\newblock {Infinitesimal Zariski closures of positive representations}.
\newblock {\em J. Differential Geometry}, 128(2):861 -- 901, 2024.

\bibitem{SilvermanPaper}
J.~H. Silverman.
\newblock The difference between the weil height and the canonical height on
  elliptic curves.
\newblock {\em Mathematics of computation}, 55(192):723--743, 1990.

\bibitem{smilga1}
I.~Smilga.
\newblock Proper affine actions on semisimple $\textrm{L}$ie algebras.
\newblock {\em Ann. Inst. Fourier}, 66:785--831, 2016.

\bibitem{SmilgaAnnalen}
I.~Smilga.
\newblock Proper affine actions: a sufficient criterion.
\newblock {\em Math. Annalen}, 382:513--605, 2022.

\bibitem{sullivan}
D.~Sullivan.
\newblock {The density at infinity of a discrete group of hyperbolic motions}.
\newblock {\em Publ. Math. de l'I.H.E.S.}, 50:171--202, 1979.

\bibitem{Tits}
J.~Tits.
\newblock {Repr{\'e}sentations lin{\'e}aires irr{\'e}ductibles d'un groupe
  r{\'e}ductif sur un corps quelconque}.
\newblock {\em J. Reine Angew. Math.}, 247:196--220, 1971.

\bibitem{cuartica}
N.~Tzanakis.
\newblock Solving elliptic diophantine equations by estimating linear forms in
  elliptic logarithms. $\mathrm{T}$he case of quartic equations.
\newblock {\em Acta Arithmetica}, LXXV(2):165--190, 1996.

\bibitem{TzanakisBook}
N.~Tzanakis.
\newblock {\em Elliptic Diophantine equations}.
\newblock Number~2 in Series in Discrete Mathematics and Applications. de
  Gruyter, 2013.

\bibitem{yue}
C.~Yue.
\newblock {The ergodic theory of discrete isometry groups on manifolds of
  variable negative curvature}.
\newblock {\em Trans. of the A.M.S.}, 348(12):4965--5005, 1996.

\end{thebibliography}
\bibliographystyle{plain}

\author{\vbox{\footnotesize\noindent 
	Andr\'es Sambarino\\
	CNRS - Sorbonne Universit\'e \\ IMJ-PRG (CNRS UMR 7586)\\ 
	4 place Jussieu 75005 Paris France\\
	\texttt{andres.sambarino@imj-prg.fr}
\bigskip}}

\end{document}